\documentclass[a4paper,11pt,reqno,noindent]{amsart}
\usepackage[centertags]{amsmath}
\usepackage{amsfonts,amssymb,amsthm,dsfont,cases,amscd,esint,enumerate}
\usepackage[T1]{fontenc}
\usepackage[english]{babel}
\usepackage[applemac]{inputenc}
\usepackage{newlfont}
\usepackage{color}
\usepackage[body={15cm,21.5cm},centering]{geometry} 
\usepackage{fancyhdr}
\usepackage{enumerate}

\theoremstyle{plain}
\newtheorem{theor0}{Theorem}[section]
\newenvironment{theor}
  {\pushQED{\qed}\begin{theor0}}
  {\popQED\end{theor0}}
\newtheorem{lem0}[theor0]{Lemma}
\newenvironment{lem}
  {\pushQED{\qed}\begin{lem0}}
  {\popQED\end{lem0}}
\newtheorem{prop0}[theor0]{Proposition}
\newenvironment{prop}
  {\pushQED{\qed}\begin{prop0}}
  {\popQED\end{prop0}}
\newtheorem{cor0}[theor0]{Corollary}

\newtheorem{hyp0}[theor0]{Hypothesis}
\newenvironment{hyp}
  {\pushQED{\qed}\begin{hyp0}}
  {\popQED\end{hyp0}}
\newtheorem{defin0}[theor0]{Definition}
\newenvironment{defin}
  {\pushQED{\qed}\begin{defin0}}
  {\popQED\end{defin0}}

\numberwithin{equation}{section}

\newcommand{\e}{\varepsilon}

\newcommand{\Pc}{\mathcal{P}}

\newcommand{\calS}{\mathcal S}

\newcommand{\calT}{\mathcal T}

\newcommand{\df}{\mathfrak d}
\newcommand{\dist}{\operatorname{dist}}

\newcommand{\R}{\mathbb R}
\newcommand{\Z}{\mathbb Z}

\newcommand{\Ic}{\mathcal I}

\newcommand{\Md}{\mathbb M}

\newcommand{\Nc}{\mathcal N}
\newcommand{\cvf}{\rightharpoonup}
\newcommand{\loc}{{\operatorname{loc}}}
\newcommand{\Id}{\operatorname{Id}}

\newcommand{\E}{\mathbb{E}}
\newcommand{\per}{{\operatorname{per}}}
\newcommand{\D}{\operatorname{D}}
\newcommand{\ee}{e}
\newcommand{\supess}{\operatorname{sup\,ess}}

\newcommand{\infess}{\operatorname{inf\,ess}}

\newcommand{\bb}{\bar{\boldsymbol b}}
\newcommand{\Bb}{\bar{\boldsymbol B}}
\newcommand{\Ff}{\bar{\boldsymbol F}}
\newcommand{\Cc}{\bar{\boldsymbol C}}
\newcommand{\cc}{\bar{\boldsymbol c}}

\newcommand{\Ld}{\operatorname{L}}
\newcommand{\Div}{{\operatorname{div}}}
\newcommand{\Sym}{{\operatorname{sym}}}
\newcommand{\Skew}{{\operatorname{skew}}}

\newcommand{\tot}{{\operatorname{tot}}}
\newcommand{\act}{{\operatorname{act}}}
\newcommand{\pas}{{\operatorname{pas}}}
\newcommand{\step}[1]{\noindent \textit{Step} #1.}
\newcommand{\substep}[1]{\noindent \textit{Substep} #1.}
\newcommand{\Pm}{\mathbb{P}}
\newcommand{\expec}[1]{\mathbb{E}\left[ #1 \right]}
\newcommand{\expecm}[1]{\mathbb{E}\big[ #1 \big]}
\newcommand{\expecM}[1]{\mathbb{E}\bigg[ #1 \bigg]}
\newcommand{\var}[1]{\mathrm{Var}\left[#1\right]}
\newcommand{\Ups}{\Upsilon}

\usepackage[colorlinks,citecolor=black,urlcolor=black]{hyperref}

\title[Homogenization for active suspensions]{Homogenization of active suspensions\\and reduction of effective viscosity}

\author[A. Bernou]{Armand Bernou}
\address[Armand Bernou]{Sorbonne Universit\'e, CNRS, Universit\'e de Paris, Laboratoire Jacques-Louis Lions, 75005~Paris, France}
\email{armand.bernou@sorbonne-universite.fr}
\author[M. Duerinckx]{Mitia Duerinckx}
\address[Mitia Duerinckx]{Universit\'e Libre de Bruxelles, D\'epartement de Math\'ematique, 1050~Brussels, Belgium}
\email{mitia.duerinckx@ulb.be}
\author[A. Gloria]{Antoine Gloria}
\address[Antoine Gloria]{Sorbonne Universit\'e, CNRS, Universit\'e de Paris, Laboratoire Jacques-Louis Lions, 75005~Paris, France \& Institut Universitaire de France \& Universit\'e Libre de Bruxelles, D\'epartement de Math\'ematique, 1050~Brussels, Belgium}
\email{antoine.gloria@sorbonne-universite.fr}

\begin{document}
\begin{abstract}
We consider a suspension of active rigid particles (swimmers) in a steady Stokes flow,
 where particles are distributed according to a stationary ergodic random process, and we study its homogenization in the macroscopic limit.
A key point in the model is that swimmers are allowed to adapt their propulsion to the surrounding fluid deformation:  swimming forces are not prescribed a priori, but are rather obtained through the retroaction of the fluid.
Qualitative homogenization of this nonlinear model requires an unusual proof that crucially relies on a semi-quantitative two-scale analysis. After introducing new correctors that accurately capture spatial oscillations created by swimming forces, we identify the contribution of the activity to the effective viscosity. 
In agreement with the physics literature, an analysis in the dilute regime shows that the activity of the particles can either increase or decrease the effective viscosity (depending on the swimming mechanism), {which differs from the well-known case of passive suspensions.}
\end{abstract}

\maketitle
\tableofcontents

\section{Introduction and main results}
 
\subsection{General overview}
This work is devoted to the large-scale rheology of suspensions of active particles in viscous fluids, where \emph{active} particles are devices that can propel themselves in the fluid (in a direction that can adapt to the surrounding fluid flow itself).
Important examples include suspensions of bacteria~\cite{Sokolov2009}, micro-algae~\cite{Yasa2018}, nanomotors~\cite{Paxton_2004}, etc.  Compared to passive systems, active suspensions exhibit a particularly rich phenomenology, with the experimental observation of pattern formations~\cite{Cisneros_2007} and unsteady whirls and jets~\cite{Mendelson_1999}. Due to this complexity, the response to an external forcing can defy intuition, with rheological measurements displaying in some settings a transition to a superfluid-like behavior~\cite{Clement-15}.
We refer to~\cite{Hatwalne2004,Sokolov-Aranson-09,Rafai2010,Potomkin_2017} for more physical context. {In the present contribution, our main purpose is to establish rigorously, starting from a simple (yet realistic) \emph{microscopic} fluid-particle model, that the presence of active particles in a fluid can indeed drastically reduce its \emph{effective} viscosity.

Microscopic fluid-particle models are challenging to analyze as they involve collective fluid-structure interactions.
More precisely, the fluid should be described by the Navier--Stokes equations outside the particles, while the latter are assumed to be rigid and thus described by their translational and angular velocities. As a first approximation, we naturally assume no-slip boundary conditions at the boundary of the particles, so that the velocity of the fluid coincides with that of the particles at their boundaries.
The fluid flow thus depends on the whole set of particles via boundary conditions, and the propulsion mechanism of active particles further yields forces on the surrounding fluid.
Reciprocally, both the fluid and the propulsion mechanism exert forces on the particles, the dynamics of which is given by the corresponding Newton equations.
Due to the multibody and long-range nature of particle interactions via the fluid flow, this fluid-particle system is highly difficult to understand  --- let alone analyze rigorously --- in the macroscopic limit with a large number of particles.

On large scales, fluid-particle systems are expected to be approximated by multiscale models: On the one hand, the fluid flow would satisfy a macroscopic fluid equation including an additional effective stress due to the presence of the particles and their self-propulsion. On the other hand, this effective stress would depend on the local geometry of the set of particles on small scales, so the macroscopic fluid equation would be coupled to a microscopic evolution equation for the latter. This amounts to a scale separation in the description of the system:
the macroscopic fluid equation is coupled to the local microstructure dynamics.
We emphasize the nonlinear structure of such models: the fluid flow depends on the local microstructure, the evolution of which depends itself on the fluid flow. The resulting flow-induced microstructure leads to possibly nonlinear response to external forces, thus explaining the well-known non-Newtonian behavior of particle suspensions.

As models for the microstructure dynamics are difficult to formulate concretely,
a first simplification popular in physics and in applied mathematics consists in considering the  dilute regime. In that case, particles interact little and can be viewed as approximately isolated, which allows to reduce the multibody nature of the system. In the spirit of Einstein's formula, the effective stress in the macroscopic fluid equations then takes an explicit form that only depends on the local mean-field distribution of particle orientations, and the macroscopic fluid equation is simply coupled to a kinetic equation for the latter.
We refer for instance to~\cite{DoiEdwards88,Saintillan2014,Saintillan_2018} for a review of multiscale models in this vein. The mathematical justification from dilute fluid-particle systems has attracted considerable interest in recent years and has been completed in a few settings~\cite{Hofer-Schubert-21,Hofer-Leocata-Mecherbet-22,Hofer-Mecherbet-Schubert-22,MD-23,Hofer-Schubert-23}. In particular, the recent work~\cite{MD-23} by the second author addresses the full particle dynamics in the (semi)dilute regime and derives rigorously some multiscale Doi-type model in the spirit of~\cite{DoiEdwards88,Saintillan_2018,Saintillan2014}.
In the same spirit, let us also mention the FENE model for suspensions of free polymer chains in a fluid flow~\cite{MR1383323,MR830199,MR2503655,MR2227763}: this model further takes into account the extensibility of the polymers and has been analyzed in~\cite{MR2039220,MR3010381}, but its mathematical justification from a microscopic fluid-particle system is still an open problem.
Yet, despite their success, those dilute multiscale models neglect the spatial correlations of the particles on the microscale, which become important for less dilute suspensions, cf.~\cite{Hofer-Mecherbet-Schubert-22,MD-23}.

In the present work, we take another route and consider a non-dilute regime
that is beyond the reach of mean-field approaches. To simplify the analysis, we neglect the inertia of the fluid, thus considering the steady Stokes equations instead of Navier--Stokes (vanishing Reynolds number), and we further neglect the inertia of the particles.
What remains is the following: given instantaneous positions and orientations of the particles (at microscale $\e \ll 1$), the fluid flow $u_\e$ is given by solving the steady Stokes equations outside the particles, while their positions and orientations are updated as above.
Note that this simplified fluid-particle system keeps its fundamental difficulty, namely the multibody and long-range nature of particle interactions via the fluid flow.
Rather than studying the full dynamics of this system, which seems out of reach beyond dilute regimes,
we further assume that a scale separation holds in the following sense: the large-scale behavior of the fluid is described by an effective flow $(t,x)\mapsto\bar u(t,x)$, while the positions and orientations of the particles at time~$t$ around a point~$x$ are locally given by a realization of some stationary ergodic decorated point process depending on the previous history of the effective flow $\{\bar u(s,\cdot)\}_{s\le t}$ around $x$.
As non-Newtonian effects are mainly due to the collective orientation of the particles~\cite{DoiEdwards88,Saintillan2014,Saintillan_2018}, we focus on the latter and we further neglect memory effects for simplicity, thus rather assuming the following: the positions of the particles at time~$t$ around a point~$x$ are locally given by a realization of some stationary ergodic point process~$\Pc$ (independent of $t,x$), and their orientations are given by some probability distribution depending on the surrounding effective flow $\bar u(t,\cdot)|_{B_\delta(x)}$ --- say in some neighborhood of mesoscopic size~$\delta>0$.
This dependence is assumed to be given as a data, having in mind that it should be related to invariant measures of the particle dynamics under the surrounding effective flow.
We are then left with the following fluid-structure interaction problem: at a given time $t$, the microscopic fluid flow $u_\e$ is described by the steady Stokes equations outside a random ensemble of particles with positions given by~$\Pc$ and with orientations depending on the surrounding effective flow $\bar u(t,\cdot)|_{B_\delta(x)}$.
As the latter is a large-scale approximation of the microscopic solution $u_\e$ itself, we choose to explicitly close the problem by replacing $\bar u(t,x)$ with~$\chi_\delta\ast u_\e(t,x)$ for some convolution kernel $\chi_\delta$ supported in $B_\delta$.
Note that this dependence of particle orientations on the fluid flow makes the problem explicitly nonlinear.
We refer to~\eqref{eq:Stokes-re} below for a precise formulation.

The aim of the present contribution is to analyze this original nonlinear steady-state model in the macroscopic limit of a large number of small particles with fixed density.
We prove a homogenization result in form of the convergence of the microscopic fluid flow to some effective flow $\bar u$, which is compatible with the closure assumption in the sense that $\lim_{\delta \downarrow 0} \lim_{\e \downarrow 0} \chi_\delta\ast u_\e=\bar u$, and we identify the effective fluid equation in terms of some nonlinear effective viscosity.
This result can be seen as a building block for a multiscale modelling of the fluid-particle dynamics beyond the dilute regime, which should ultimately be combined 
with the study of the very challenging question of the nondilute microstructure dynamics (that we do not address in this contribution).
As a corollary, we justify the drastic reduction of the effective viscosity for so-called pusher particles.
In addition, as a sanity check for the model, we justify in the dilute regime an Einstein-type approximation for our nonlinear effective viscosity, see formula~\eqref{eq:altern-Bact-1} below for the active contribution, which coincides with the standard expression for the active elastic stress in dilute multiscale models~\cite{Haines2008,Haines_2009,Ryan-Aranson-11,Potomkin_2017,Saintillan_2018,Saintillan2014,MD-23}.
%when considered in the dilute regime (that is, in the limit of small density of particles), the relative positions of the particles should become irrelevant and the effective stress should be given at leading order by the standard formulas used in~\cite{Saintillan_2018,Saintillan2014}. 
%In the present contribution,  we prove a homogenization result for the nonlinear steady-state model \eqref{eq:Stokes-re}  for active suspensions, we identify the effective fluid equation on large scales in terms of some nonlinear effective viscosity, and
%we justify the corresponding Einstein-type approximation in the dilute regime. In particular, one of our main achievements is the rigorous derivation of formula~\eqref{eq:altern-Bact-1} below for the active contribution to the effective viscosity, which coincides with the standard expression for the active elastic stress in dilute multiscale models~\cite{Haines2008,Haines_2009,Ryan-Aranson-11,Potomkin_2017,Saintillan_2018,Saintillan2014,MD-23}.
%We further justify the reduction of the effective viscosity for so-called pusher particles.
%Before we describe more precisely the model, l
In a similar spirit, let us also mention that Girodroux-Lavigne recently analyzed in~\cite{Girodroux-Lavigne-22} a steady Stokes model with a dilute suspension of active particles with orientations and swimming forces that are completely prescribed in advance (independently of the fluid flow): this constitutes a linear version of the model analyzed here.

\medskip
The sequel of this introduction is organized as follows:
In Section~\ref{sec:passive}, we recall the steady-state model for a steady Stokes fluid with a suspension of {\it passive} rigid particles, and we state the associated homogenization result previously obtained by the last two authors in that setting~\cite{D21,DG-21a, DG-21,DG-22-review}.
In Section~\ref{sec:model-part}, we introduce our new nonlinear steady-state model for active suspensions in a steady Stokes flow, as inspired in particular by~\cite{Saintillan_2018, Saintillan2014}.
In Section~\ref{sec:heuristic}, we relate this model
%to the building block of the multiscale model,
to the physics literature.
%and show how it can be used to recover the Einstein-like results of the physics literature on a heuristic level.
Section~\ref{sec:main-results} is dedicated to the main results of this paper: the rigorous homogenization of the nonlinear model, and the analysis of the effective rheology in the dilute regime. Last, in Section~\ref{sec:visco-reduc}, based on these results, we investigate the contribution of active particles to the effective viscosity, and rigorously establish that a significant reduction can take place in the case of so-called pusher particles.}

\subsection{Reminder on passive suspensions}\label{sec:passive}

Given an underlying probability space $(\Omega,\Pm)$ (with expectation $\expec{\cdot}$),
let $\{x_n\}_n$ be a random point process on the ambient space $\R^d$,
consider an associated collection of random shapes $\{I_n^\circ\}_n$, where each $I_n^\circ$ is a connected random open subset of the unit ball $B$ centered at the origin (in the sense that $\int_{I_n^\circ}y\,dy=0$), and then define the corresponding inclusions
\[I_n\,:=\,x_n+I_n^\circ.\]
Note that random shapes are not required to be independent of the point process $\{x_n\}_n$. We then consider the random set
\[\Ic\,:=\,\bigcup_n I_n,\]
which we assume to satisfy the following for some $\vartheta>0$.

\begin{hyp}[Particle suspension]\label{hyp:part}$ $
\begin{enumerate}[(a)]
\item \emph{Stationarity and ergodicity:} The random set $\Ic$ is stationary and ergodic.
\item \emph{Uniform $C^2$ regularity:} The random shapes $\{I_n^\circ\}_n$ satisfy interior and exterior ball conditions with radius $\vartheta$ almost surely.
\item \emph{Uniform hardcore condition:} For some $\ell\ge\vartheta$, there holds $(I_n+\ell B)\cap(I_m+\ell B)=\varnothing$ almost surely for all $n\ne m$. We let $\ell$ be the largest such value, that is, half the interparticle distance
\begin{equation}\label{eq:def-ell}
\ell\,:=\,\tfrac12\inf_{n\ne m}\dist(I_n,I_m).\qedhere
\end{equation}
\end{enumerate}
\end{hyp}

Now consider a tank, described as a bounded Lipschitz domain $U\subset\R^d$, and assume that it is filled with a steady Stokes fluid with a suspension of particles of size $\e$, described as the $\e$-rescaling of $\Ic$.
More precisely, we only consider particles that are included in $U$ and remove those close to the boundary:
define $\Nc_\e(U)$ as the set of indices $n$ such that $\e (I_n + \ell B) \subset U$, and set
\[\Ic_\e(U):=\bigcup_{n\in\Nc_\e(U)}\e I_n.\]
 We write $u_\e$ for the fluid velocity, $P_\e$ for the corresponding pressure. We assume Dirichlet conditions $u_\e = 0$ on $\partial U$, and we extend the fluid velocity inside particles with the rigidity constraint
\[ \D(u_\e) :=\tfrac12(\nabla u_\e+(\nabla u_\e)^T)= 0 \qquad \hbox{ in } \Ic_\e(U).\]
Recall the definition of the Cauchy stress tensor
\[ \sigma(u_\e, P_\e) \,:=\, 2 \D(u_\e) - P_\e \Id ,\]
where $\Id$ denotes the identity matrix.
Given an internal force $h$, the fluid velocity $u_\e \in H^1_0(U)^d$ and the associated pressure $P_\e \in\Ld^2(U\setminus\Ic_\e(U))$ are then given as the solutions of the Stokes system
\begin{equation}\label{eq:Stokes-pas}
\left\{\begin{array}{ll}
-\triangle u_\e+\nabla P_\e  = h  &\text{in $U\setminus\Ic_\e(U)$},\\
\Div( u_\e)=0,&\text{in $U\setminus\Ic_\e(U)$},\\
\D( u_\e)=0,&\text{in $\Ic_\e(U)$},\\
\int_{\e\partial I_n}\sigma(u_\e,P_\e)\nu = 0,&\forall n \in \Nc_\e(U),
\\
\int_{\e\partial I_n}\Theta(x-\e x_n)\cdot\sigma(u_\e,P_\e)\nu  = 0, &\forall n \in \Nc_\e(U),\, \Theta \in \Md^\Skew,
\end{array}\right.
\end{equation}
where $\nu$ stands for the outward normal to $\partial I_n$,
where $\Md^\Skew$ is the set of skew-symmetric matrices,
and where we assume the additional anchoring condition 
\[ \int_{U \setminus \Ic_\e(U)} P_\e = 0, \]
which we shall abbreviate as choosing $P_\e \in\Ld^2(U\setminus\Ic_\e(U))/\R$.
The homogenization of the Stokes system~\eqref{eq:Stokes-pas} was the object of \cite{D21,DG-21a, DG-21}, where the last two authors proved that  $(u_\e,P_\e\mathds1_{U\setminus\Ic_\e(U)})$ converges almost surely weakly in 
$H^1_0(U)\times \Ld^2(U)/\R$ to the unique solution $(\bar u, \bar P)$ of the homogenized Stokes system 
\begin{equation}\label{eq:Stokes-pas-hom}
\left\{\begin{array}{ll}
-\Div(2\Bb_\pas\D(\bar u))+\nabla \bar P  =(1-\lambda) f  &\text{in $U$},\\
\Div( \bar u)=0,&\text{in $U$},
\end{array}\right.
\end{equation}
where $\lambda$ stands for the particle volume fraction
\begin{equation}\label{eq:part-vol-frac}
\lambda\,:=\,\expec{\mathds1_{\Ic}},
\end{equation}
and where the effective viscosity~$\Bb_\pas$ is a symmetric linear map on the set $\Md^\Sym_0$ of symmetric trace-free matrices. We recall that the latter satisfies $E : \Bb_\pas E>|E|^2$ for all $E \in \Md^\Sym_0$ as soon as $\lambda>0$, meaning that the presence of (passive) rigid particles always increases the effective viscosity. We refer to~\cite{DG-22-review} for a review of the topic.

\medskip

In view of the quantitative homogenization results that we shall need later, we occasionally make quantitative ergodicity assumptions in form of the validity of the following multiscale variance inequality introduced by the last two authors in~\cite{DG-FI2, DG-FI1}. This assumption holds for instance for hardcore Poisson point process with exponentially decaying $\pi$.

\begin{hyp}[Quantitative mixing assumption]\label{hyp:part-strong}$ $
There exists a non-increasing weight function $\pi:\R^+\to\R^+$ with superalgebraic decay (that is, $\pi(t)\le C_p\langle t\rangle^{-p}$ for all $p<\infty$) such that the random set $\Ic$ satisfies, for all $\sigma(\Ic)$-measurable random variables $Y(\Ic)$,
\[\var{Y(\Ic)}\,\le\,\expecM{\int_0^\infty\int_{\R^d}\Big(\partial_{\Ic,B_t(x)}^{\operatorname{osc}}Y(\Ic)\Big)^2dx\,\langle t\rangle^{-d}\pi(t)\,dt},\]
where the ``oscillation'' $\partial^{\operatorname{osc}}$ of the random variable $Y(\Ic)$ is defined by
\begin{align*}
&\partial^{\operatorname{osc}}_{\Ic,B_t(x)}Y(\Ic)\,:=\,\supess\Big\{Y(\Ic')\,:\,\Ic'\cap(\R^d\setminus B_t(x))=\Ic\cap(\R^d\setminus B_t(x))\Big\}\\
&\hspace{3cm}-\infess\Big\{Y(\Ic')\,:\,\Ic'\cap(\R^d\setminus B_t(x))=\Ic\cap(\R^d\setminus B_t(x))\Big\}.
\qedhere
\end{align*}
\end{hyp}

\subsection{Hydrodynamic model for active suspensions}\label{sec:model-part}
As opposed to \textit{passive} particles, \textit{active} particles propel themselves by applying a force on the surrounding fluid.
In a steady-state perspective, we assume that we are given a random ensemble of particle positions and swimming directions, and we aim to evaluate the associated large-scale rheology.
Swimming directions should not be taken as uniformly distributed, but should  depend on the surrounding fluid deformation,
which leads to a nontrivial interaction between the fluid flow and particles' swimming forces.
More precisely, our model is based on the following assumption: if the fluid is locally deformed, then the distribution of orientations depends on some local average of the symmetrized velocity gradient of the surrounding fluid around each particle.
Although this steady-state perspective is certainly simplistic,
our model does not prescribe the retroaction of the fluid on the particles a priori, but leaves it as part of the problem.
We start by modeling the swimming mechanism for a single particle, before combining it with~\eqref{eq:Stokes-pas} into a model for the whole active suspension.

\subsubsection{Single-particle swimming mechanism}
Let us place ourselves at the scale of an isolated particle $I$, and denote by $u$ the fluid velocity outside $I$.
Given a nonnegative smooth kernel $\chi$ with unit mass $\int_{\R^d}\chi=1$,
the locally averaged fluid deformation felt by the particle is taken of the form
\begin{equation}\label{eq:defin-EIu}
E_I(u)\,:=\,\fint_I\chi\ast\D(u).
\end{equation}
The precise choice of this operator does not matter in our analysis, provided that it is a compact operator applied to a restriction of $\D(u)$ around $I$.
Given a value $E_I(u)=E$ of this averaged fluid deformation, the particle adapts its random swimming direction: we denote by $\bar f(E)\in\R^d$ the resulting propulsion force and by $\tilde f(E)\in\Md^\Skew$ the resulting torque on the particle.
By the action-reaction principle, this force and torque must result from an action of the particle on the surrounding fluid. The detail of this action depends on the details of the swimming mechanism (flagella, cilia, etc.). The force field exerted by the particle on the fluid is denoted by $f(E)=f(\cdot,E)$, depending on the deformation~$E$, and is taken to be supported in the immediate neighborhood $(I+B)\setminus I$ of the particle.
By the action-reaction principle for forces and torques, we must have
\begin{eqnarray}\label{eq:neutrality.0}
\bar f(E)+\int_{(I+B)\setminus I}f(E)&=&0,\\
\Theta:\tilde f(E)+\int_{(I+B)\setminus I}\Theta x\cdot f(E)&=&0,\qquad\forall\Theta\in\Md^\Skew,\nonumber
\end{eqnarray}
since the barycenter of particle $I$ is $\int_Iy\,dy=0$. The relation \eqref{eq:neutrality.0} reads as a local neutrality condition that actually entails that swimming forces act as dipoles in the fluid equations.
We emphasize that this is fundamentally different from the sedimentation problem studied in~\cite{DG-21+}, for which the force on particles originates from gravity and is not compensated locally by opposite forces in the surrounding fluid --- the backflow is then uniform and leads to more important large-scale effects.

We further make the following assumptions on the regularity of the swimming force with respect to the fluid deformation.

\begin{hyp}[Swimming mechanism]\label{hyp:swim}
The random force field defines almost surely a smooth map $\Md^\Sym_0\to \Ld^\infty((I+B)\setminus I)^d:E\mapsto f(E)$ such that, for all $k\ge1$,
\begin{align*}
\|f(E)\|_{\Ld^\infty((I+B)\setminus I)}&~\le~C\langle E\rangle,\\
\|\partial^kf(E)\|_{\Ld^\infty((I+B)\setminus I)}&~\le~C_k.\qedhere
\end{align*}
\qedhere
\end{hyp}
In view of quantitative homogenization results, we shall occasionally need to further assume that for large strain rates $E$ the swimming direction becomes a deterministic function of $E$. This technical assumption is physically reasonable.
\begin{hyp}[Swimming in large strain rate]\label{hyp:swim-large}
There exist a deterministic direction field $f^\infty:\Md^\Sym_0\to\R^d$, a random strength field $\calS \in\Ld^\infty((I+B)\setminus I)$, and an exponent~$\gamma>0$, such that for all $E\in\Md_0^\Sym$ and $k\ge1$ we have almost surely
\begin{align*}
\|f(E)-\calS f^\infty(E)\|_{\Ld^\infty((I+B)\setminus I)}&~\le~ C\langle E\rangle^{1-\gamma},\\
\|\partial^kf(E)-\calS \partial^kf^\infty(E)\|_{\Ld^\infty((I+B)\setminus I)}&~\le~ C_k\langle E\rangle^{-\gamma}.\qedhere
\end{align*}
\end{hyp}

\subsubsection{Resulting system for many particles}
Before including the above single-particle swimming mechanism into the passive suspension model~\eqref{eq:Stokes-pas}, we start by making an assumption on the joint law of particles' swimming forces.

\begin{hyp}[Joint swimming forces]\label{hyp:joint}
Let $\{f_n(E)\}_n$ be a sequence of random maps, such that $f_n$ satisfies Hypothesis~\ref{hyp:swim} with $I=I_n$ for all $n$,
and such that
$\Ic$ and $\sum_nf_n(E)$ are jointly stationary for all $E$.
\end{hyp}

In order to include these swimming forces into the model~\eqref{eq:Stokes-pas} for a suspension of small particles $\{\e I_n\}_n$, they need to be properly rescaled.
The natural scaling happens to be~$O(\frac1\e)$, which is indeed the only scaling giving rise to a nontrivial and finite contribution in the macroscopic limit~$\e\downarrow 0$.
We add a coupling parameter $\kappa$, which stands for the activity strength and will need to be chosen small enough to perform the analysis.
In this $\e$-rescaling, the kernel $\chi$ defining the local averaged fluid deformations felt by the particles~\eqref{eq:defin-EIu} should naturally be replaced by $\chi_\e:=\e^{-d}\chi(\frac\cdot\e)$. This however leads to important difficulties due to the highly oscillatory local behavior of the fluid flow. Instead, we need to replace it by $\chi_\delta:=\delta^{-d}\chi(\frac\cdot\delta)$, for some intermediate averaging scale $\e\ll\delta\ll1$. (This ``meso-scale'' is also present in \cite{Saintillan_2018, Saintillan2014}.)
The resulting hydrodynamic model takes on the following guise,
\begin{equation}\label{eq:Stokes}
\left\{\begin{array}{ll}
-\triangle u_\e+\nabla P_\e&\\
 \qquad = h+\frac{\kappa}{\e} \sum_{n \in \mathcal{N}_\e(U)} f_{n,\e}(\fint_{\e I_n}\chi_\delta \ast\D( u_{\e}) ), &\text{in $U\setminus\Ic_\e(U)$},\\
\Div( u_\e)=0,&\text{in $U\setminus\Ic_\e(U)$},\\
u_\e=0,&\text{on $\partial U$},\\
\D( u_\e)=0,&\text{in $\Ic_\e(U)$},\\
\int_{\e\partial I_n}\sigma(u_\e,P_\e)\nu+\tfrac \kappa \e \bar f_{n}(\fint_{\e I_n}\chi_\delta \ast\D( u_{\e}) ) = 0,&\forall n \in \Nc_\e(U),\\
\int_{\e\partial I_n}\Theta(\cdot-\e x_n)\cdot\sigma(u_\e,P_\e)\nu&\\
\hspace{2.5cm}+\frac\kappa\e\,\Theta: \tilde f_{n}(\fint_{\e I_n}\chi_\delta \ast\D( u_{\e}) ) = 0, &  \forall n \in \Nc_\e(U),\, \Theta \in \Md^\Skew,
\end{array}\right.
\end{equation}
where we have set
\[f_{n,\e}(E)\,:=\,f_{n,\e}(\cdot,E)\,:=\,\e^{-d}f_n\big(\tfrac{\cdot}\e,E\big).\]
As above, the pressure is anchored via $\int_{U \setminus \Ic_\e(U)} P_\e=0$.

In what follows, it will be convenient to use an equivalent formulation of swimming forces. While each force field $f_n$ is supported in the particle neighborhood $(I_n+B)\setminus I_n$ in the fluid domain, we may naturally extend $f_n$ inside the particle domain~$I_n$ to match its propulsion force and torque.
More precisely, we can uniquely define $f_{n}$ in~$I_n$ as an affine function such that
\begin{equation}\label{eq:extension-fn}
\bar f_{n}(E)\,=\,\int_{I_n}f_{n}(E),\qquad\tilde f_{n}(E)\,=\,\int_{I_n}f_{n}(E)\otimes(x-x_n).
\end{equation}
In terms of these extensions, the neutrality condition~\eqref{eq:neutrality.0} takes the simpler form
\begin{eqnarray}\label{eq:neutrality}
\int_{I_n+B}f_n(E)&=&0,\\
\int_{I_n+B}\Theta (x-x_n)\cdot f_n(E)&=&0,\qquad\forall\Theta\in\Md^\Skew,\nonumber
\end{eqnarray}
and the system~\eqref{eq:Stokes} then becomes
\begin{equation}\label{eq:Stokes-re}
\left\{\begin{array}{ll}
-\triangle u_\e+\nabla P_\e&\\
 \qquad = h+\frac{\kappa}{\e} \sum_{n \in \mathcal{N}_\e(U)} f_{n,\e}(\fint_{\e I_n}\chi_\delta \ast\D( u_{\e}) ), &\text{in $U\setminus\Ic_\e(U)$},\\
\Div( u_\e)=0,&\text{in $U\setminus\Ic_\e(U)$},\\
u_\e=0,&\text{on $\partial U$},\\
\D( u_\e)=0,&\text{in $\Ic_\e(U)$},\\
\int_{\e\partial I_n}\sigma(u_\e,P_\e)\nu+\tfrac \kappa \e \int_{\e I_n} f_{n,\e}(\fint_{\e I_n}\chi_\delta \ast\D( u_{\e}) )=0 ,&\forall n \in \Nc_\e(U),\\
\int_{\e\partial I_n}\Theta(x-\e x_n)\cdot\sigma(u_\e,P_\e)\nu&\\
\quad+\tfrac\kappa\e\int_{\e I_n}\Theta(x-\e x_n)\cdot f_{n,\e}(\fint_{\e I_n}\chi_\delta \ast\D( u_{\e}) )=0, &\forall n \in \Nc_\e(U),\,\Theta \in \Md^\Skew.
\end{array}\right.
\end{equation}
With this reformulation, we may readily check that the solution $u_\e$ can be viewed as the orthogonal projection on $\{u\in H^1_0(U)^d:\D(u)|_{\Ic_\e(U)}=0,\,\Div(u)=0\}$ of the solution of
\begin{equation}\label{eq:Stokes-rre}
\left\{\begin{array}{ll}
-\triangle v_\e+\nabla R_\e = h+\frac{\kappa}{\e} \sum_{n \in \mathcal{N}_\e(U)} f_{n,\e}(\fint_{\e I_n}\chi_\delta \ast\D( u_{\e}) ), &\text{in $U$},\\
\Div( v_\e)=0,&\text{in $U$},\\
v_\e=0,&\text{on $\partial U$}.
\end{array}\right.
\end{equation}
This observation is not used in the sequel.

\subsection{Heuristics and relation to the physics literature}\label{sec:heuristic}

In this section, we relate the problem we consider to the building block used in \cite{Saintillan_2018, Saintillan2014}.
In the physics literature, rather than the full problem~\eqref{eq:Stokes} (involving boundary conditions, and a general forcing term $h$), one usually considers a forcing term in form of an imposed strain rate $E\in\Md_0^\Sym$ at infinity --- in which case it is natural to replace \eqref{eq:defin-EIu} by $E$ itself (which renders the problem linear).
The velocity field of the suspension on microscopic scale is then given by $u_E+Ex$, where~$u_E$ is a suitable solution of the following infinite-volume problem,
\begin{align}
\label{eq:model_imposed_strain}
\left\{\begin{array}{ll}
-\triangle u_E+\nabla P_E=\kappa\sum_nf_n(E),&\text{in $\R^d\setminus \Ic$},\\
\Div(u_E)=0,&\text{in $\R^d\setminus \Ic$},\\
\D(u_E+Ex)=0,&\text{in $\Ic$},\\
\int_{\partial I_n}\sigma(u_E+Ex,P_E)\nu+\kappa\int_{I_n}f_n(E)=0,&\forall n,\\
\int_{\partial I_n}\Theta(x-x_n)\cdot\sigma(u_E+Ex,P_E)\nu\\ \hspace{2cm} +\kappa\int_{I_n}\Theta(x-x_n)\cdot f_n(E)=0,&\forall n,\,\Theta\in\Md^\Skew.
\end{array}\right.
\end{align}
In these terms, the effective viscosity $\Bb_\tot(E)$ of the suspension in direction $E$ is obtained as the associated ensemble-averaged stress. Splitting the contributions of the stress in the fluid domain and in the particles, and taking into account swimming forces, we get for all~$E'\in\Md^\Sym_0$,
\begin{multline}\label{eq:viscosity-total}
E':2\Bb_\tot(E)\,=\,E':\expecm{\sigma(u_E+Ex,P_E)\mathds1_{\R^d\setminus\Ic}}\\
+E':\expecm{\sigma_E\mathds1_{\Ic}}-\kappa \expecM{\sum_nE'(x-x_n)\cdot f_n(E)},
\end{multline}
where $\sigma_E$ stands for the stress inside the particles, which we shall define in the proof of Lemma~\ref{lem:KE} below via  the extension problem
\begin{equation}\label{eq:sigma-E}
\left\{\begin{array}{ll}
-\Div(\sigma_E)= \kappa f_n(E),&\text{in $I_n$},\\
\sigma_E\nu=\sigma(u_E+Ex,P_E)\nu,&\text{on~$\partial I_n$.}
\end{array}\right.
\end{equation}
Noting that rigidity constraints $\D(u_E+Ex)=0$ in $\Ic$ yield
\[\sigma(u_E+Ex,P_E)\mathds1_{\R^d\setminus\Ic}\,=\,2\D(u_E+Ex)-P_E\Id\mathds1_{\R^d\setminus\Ic},\]
and recalling that $\expec{\D(u_E)}=0$ (as the average of a gradient), the first contribution in~\eqref{eq:viscosity-total} takes the form
\[E':\expecm{\sigma(u_E+Ex,P_E)\mathds1_{\R^d\setminus\Ic}}\,=\,2E':E.\]
Next, using~\eqref{eq:sigma-E}, integrating by parts, and using  \eqref{eq:extension-fn}, and the skew-symmetry of $\tilde f_n(E)$,
one can reformulate the second contribution in~\eqref{eq:viscosity-total} as the ensemble-averaged stresslet on the particles
\begin{eqnarray*}
\expecm{\sigma_E\mathds1_{\Ic}}&=&\expecM{\sum_n\tfrac{\mathds1_{I_n}}{|I_n|}\int_{I_n}\sigma_E}\\
&=&\expecM{\sum_n\tfrac{\mathds1_{I_n}}{|I_n|}\int_{\partial I_n}\sigma(u_E+Ex,P_E)\nu\otimes_s(x-x_n)}.
\end{eqnarray*}
The effective viscosity~\eqref{eq:viscosity-total} thus takes the form
\begin{multline}\label{eq/Bact-0}
E':2\Bb_\tot(E)
\,=\,2E':E\\
+\expecM{\sum_n\tfrac{\mathds1_{I_n}}{|I_n|}\int_{\partial I_n}E'(x-x_n)\cdot\sigma(u_E+Ex,P_E)\nu}
-\kappa \expecM{\sum_nE'(x-x_n)\cdot f_n(E)}.
\end{multline}
For convenience, we shall distinguish the passive from the active contributions in this expression, and we decompose the solution $u_E$ as $u_E=\psi_E+\kappa \phi_E$ in terms of the so-called passive and active correctors $\psi_E$ and $\phi_E$, defined as suitable solutions of
\[\left\{\begin{array}{ll}
-\triangle \psi_E+\nabla \Sigma_E=0,&\text{in $\R^d\setminus \Ic$},\\
\Div(\psi_E)=0,&\text{in $\R^d\setminus \Ic$},\\
\D(\psi_E)+E=0,&\text{in $\Ic$},\\
\int_{\partial I_n}\sigma(Ex+\psi_E,\Sigma_E)\nu=0,&\forall n,\\
\int_{\partial I_n}\Theta(x-x_n)\cdot\sigma(Ex+\psi_E,\Sigma_E)\nu=0,&\forall\Theta\in\Md^\Skew,~\forall n.
\end{array}\right.\]
and
\[\left\{\begin{array}{ll}
-\triangle \phi_E+\nabla \Pi_E=\sum_nf_n(E),&\text{in $\R^d\setminus \Ic$},\\
\Div(\phi_E)=0,&\text{in $\R^d\setminus \Ic$},\\
\D(\phi_E)=0,&\text{in $\Ic$},\\
\int_{\partial I_n}\sigma(\phi_E,\Pi_E)\nu+\int_{I_n}f_n(E)=0,&\forall n,\\
\int_{\partial I_n}\Theta(x-x_n)\cdot\sigma(\phi_E,\Pi_E)\nu+\int_{I_n}\Theta(x-x_n)\cdot f_n(E)=0,&\forall\Theta\in\Md^\Skew,~\forall n.
\end{array}\right.\]
Precise definitions of these correctors are postponed to Section~\ref{sec:cor}. In these terms, the effective viscosity takes the form
\begin{equation}\label{eq/Bact-rewr}
\Bb_\tot(E)\,=\,\Bb_\pas E+ \kappa \Bb_\act(E),
\end{equation}
where the passive and active contributions are given by
\begin{eqnarray*}
E':2\Bb_\pas E&:=&2E':E+\expecM{\sum_n\tfrac{\mathds1_{I_n}}{|I_n|}\int_{\partial I_n}E'(x-x_n)\cdot\sigma(\psi_E+Ex,\Sigma_E)\nu},\\
E':2\Bb_\act(E)&:=& \expecM{\sum_n\tfrac{\mathds1_{I_n}}{|I_n|}\int_{\partial I_n}E'(x-x_n)\cdot\sigma(\phi_E,\Pi_E)\nu}\\
&&-\expecM{\sum_nE'(x-x_n)\cdot f_n(E)}.
\end{eqnarray*}
Using equations for correctors, these expressions are equivalently given by
\begin{eqnarray}
E:2\Bb_\pas E&=&\expecm{2|\!\D(\psi_E)+E|^2},\nonumber\\
E':2\Bb_\act(E)&=&-\expecM{\sum_n(\psi_{E'}+E'(x-x_n))\cdot f_n(E)}.\label{eq/Bact-rewr-1}
\end{eqnarray}
As one could have expected, the active contribution $E':\Bb_\act(E)$ coincides with the averaged swimming force along the passive corrector in direction $E'$. Note in particular that the active corrector $\phi_E$ \emph{does not appear} in that formulation. 
The first main contribution of the present article is to properly justify these effective viscosity formulas using homogenization theory, cf.~Theorem~\ref{th:homog}.

\medskip

While these general formulas are difficult to analyze in practice without resorting to numerical simulations, it is a classical problem in the physics community to derive simpler approximate formulas in the dilute regime, which are easier to interpret and provide a useful grasp at the physical behavior of suspensions. This is made possible by replacing correctors by explicit solutions of single-particle problems: we refer to Theorem~\ref{theor:dilute} and Section~\ref{sec:visco-reduc} below for justification of such dilute approximations based on the methods introduced by the last two authors in~\cite{DG-20}.Note that the formula for the active contribution of the effective viscosity in the dilute approximation coincides with the active elastic stress first formally derived in~\cite{Saintillan_2008,Haines_2009,Saintillan-10a,Saintillan2014,Potomkin_2016}, cf.~\eqref{eq:altern-Bact-1} below; see also~\cite{Girodroux-Lavigne-22}.

\subsection{Main results: well-posedness, homogenization, and dilute regime}\label{sec:main-results}
We turn to the statement of our main results
and we start with the well-posedness of the hydrodynamic model~\eqref{eq:Stokes}.
It requires either the coupling constant~$\kappa$ to be small enough or the interparticle distance $\ell$ to be large enough.
Note that condition~\eqref{eq:cond-small-kappa} below is nearly almost optimal in general: the same condition with $\eta=0$ is required to ensure the perturbative well-posedness of the homogenized equation~\eqref{eq:homog}, see first paragraph of Section~\ref{sec-qual-homog}.
\begin{prop}[Well-posedness of the hydrodynamic model]\label{lem:well-posed}
Let Hypotheses~\ref{hyp:part}, \ref{hyp:swim}, and~\ref{hyp:joint} hold, and assume $\e\ell\le\delta$.
Provided for some $\eta>0$ we have that
\begin{equation}\label{eq:cond-small-kappa}
\kappa\ell^{\eta-d}\ll1
\end{equation}
is small enough (only depending on $\eta$, the dimension $d$, the domain $U$, and the unscaled kernel~$\chi$),
the system~\eqref{eq:Stokes} 
is well-posed almost surely for any $h \in \Ld^2(U)^d$: there exists a unique almost sure weak solution   
$(u_\e,P_\e)\in\Ld^2(\Omega;H^1_0(U)^d \times \Ld^2(U\setminus\Ic_\e(U))/\R )$ and it satisfies almost surely
\begin{equation}\label{lem:well-posed-bd}
\int_U |\nabla u_\e|^2+\int_{U\setminus\Ic_\e(U)} (P_\e)^2
\,\lesssim_\eta\,(1+\kappa^2)\Big(\kappa^2\ell^{-d}+\int_{U\setminus\Ic_\e(U)} |h|^2\Big).\qedhere
\end{equation}
\end{prop}
We now state the homogenization result for this model in the macroscopic limit $\e\downarrow0$, in the simplified situation when the mesoscopic averaging scale~$\delta$ is fixed (see the proof for the associated corrector result).

\begin{theor}[Homogenization at fixed $\delta>0$]\label{th:homog}
Let Hypotheses~\ref{hyp:part}, \ref{hyp:swim}, and~\ref{hyp:joint} hold,
as well as the smallness condition~\eqref{eq:cond-small-kappa} to ensure well-posedness.
For any $h \in W^{1,\infty}(U)^d$, as~$\e\downarrow0$ with $\delta>0$ fixed,
the almost sure weak solution $(u_\e,P_\e)$ of~\eqref{eq:Stokes} 
satisfies almost surely
\[\begin{array}{ccll}
u_\e&\cvf& \bar u_\delta,&\text{in $H^1_0(U)$},\\
P_\e\mathds1_{U\setminus\Ic_\e(U)}&\cvf&(1-\lambda)\bar P_\delta+(1-\lambda)\bb:\D(\bar u_\delta)&\\
&&\hspace{1.2cm}+(1-\lambda)\kappa\big( \cc(\chi_\delta\ast\D(\bar u_\delta))-\fint_U \cc(\chi_\delta\ast\D(\bar u_\delta))\big),\quad&\text{in $\Ld^2(U)$,}
\end{array}\]
where~$(\bar u_\delta,\bar P_\delta)\in H^1_0(U)^d\times\Ld^2(U)/\R$ is the unique solution of the well-posed macroscopic system
\begin{equation}\label{eq:homog}
\left\{\begin{array}{ll}
-\Div(2\Bb_\pas\D(\bar u_\delta))-\Div(2\kappa \Bb_\act(\chi_\delta\ast\D(\bar u_\delta)))+\nabla \bar P_\delta\,=\,(1-\lambda)h,&\text{in $U$},\\
\Div(\bar u_\delta)=0,&\text{in $U$},\\
\bar u_\delta=0,&\text{on $\partial U$},
\end{array}\right.
\end{equation}
where $\lambda:=\expec{\mathds1_{\Ic}}$ is the particle volume fraction, and where the effective tensors $\Bb_\pas$, $\Bb_\act$, $\bb$, $\cc$ are defined as follows, in terms of the correctors $(\psi,\Sigma)$ and $(\phi,\Pi)$ given in~\eqref{e.cor-1} and~\eqref{eq:phi-cor} below,
\begin{enumerate}[$\bullet$]
\item The \emph{passive effective viscosity} $\Bb_\pas$ is a positive definite symmetric linear map on the space of symmetric trace-free matrices $\Md_0^\Sym$: together with the associated symmetric trace-free matrix $\bb\in\Md_0^\Sym$, it is defined for all $E\in\Md_0^\Sym$ by
\[2(\Bb_\pas-\Id) E+(\bb:E)\Id\,:=\,\expecM{\sum_n\tfrac{\mathds1_{I_n}}{|I_n|}\int_{\partial I_n}\sigma(\psi_E+Ex,\Sigma_E)\nu\otimes_s(x-x_n)},\]
or equivalently,
\begin{eqnarray}
E : 2\Bb_\pas E &:=& \expecm{2|\!\D(\psi_E) + E|^2},\label{eq:def-B}\\
\bb : E &:=&\tfrac1{d}\, \expecM{\sum_n \tfrac{\mathds{1}_{I_n}}{|I_n|} \int_{\partial I_n} (x-x_n) \cdot \sigma(\psi_E + Ex, \Sigma_E) \nu}.\label{eq:def-b}
\end{eqnarray}
\item The \emph{active effective viscosity} $\Bb_\act$ is given by
\begin{equation}\label{eq:def-Bact}
\Bb_\act\,:=\,\Cc+ \tfrac12 \Ff,
\end{equation}
where the map $\Cc:\Md_0^\Sym\to \Md_0^\Sym$, together with the associated map $\cc:\Md_0^\Sym\to \R$, is defined for all $E\in\Md_0^\Sym$ by
\begin{equation}\label{eq:def-C}
2\Cc(E) +\cc(E)\Id\,:=\,\expecM{\sum_n\tfrac{\mathds1_{I_n}}{|I_n|}\int_{\partial I_n} \sigma(\phi_E, \Pi_E)\nu\otimes_s(x-x_n)},
\end{equation}
and where the map $\Ff:\Md_0^\Sym\to \Md_0^\Sym$ is defined for all $E,E'\in\Md_0^\Sym$ by
\begin{equation}\label{eq:def-EF-0}
E':\Ff(E)\,:=\,-\expecM{\sum_n\tfrac{\mathds1_{I_n}}{|I_n|}\int_{I_n+B}E'(x-x_n)\cdot f_n(E)},
\end{equation}
or equivalently for all $E,E'\in \Md_0^\Sym$,
\begin{align}\label{eq:def-C+}
E':2\Bb_\act(E)&~~=~~
-\expecM{\sum_n \tfrac{\mathds1_{I_n}}{|I_n|} \int_{I_n+B} (\psi_{E'}+E'(x-x_n))\cdot f_n(E)},\\
\cc(E)&~~=~~ \tfrac1d\expecM{\sum_n\tfrac{\mathds1_{I_n}}{|I_n|}\Big(\int_{\partial I_n} (x-x_n) \cdot\sigma(\phi_E, \Pi_E)\nu\Big)}.\hspace{1.2cm}\qedhere
\end{align}
\end{enumerate}
\end{theor}
Next, we combine the above (nonlocal) homogenization limit with the (local) limit $\delta\downarrow0$.
In order to get beyond a purely diagonal regime, cf.~\eqref{eq:diagonal} below, we  need to appeal to the quantitative homogenization techniques developed in~\cite{DG-21b} by the last two authors in the context
of the Stokes equation with rigid inclusions. This requires a quantitative mixing assumption such as Hypothesis~\ref{hyp:part-strong} (which holds in particular for hardcore Poisson point processes).
Although from the modeling viewpoint the choice $\delta\sim\e$ could be more natural, it is not accessible to our analysis, and we are restricted to~\eqref{eq:cond-delta-small} below.
\begin{theor}[Homogenization as $\delta\downarrow0$]\label{theor:loc-nonlin-diag}
Let Hypotheses~\ref{hyp:part}, \ref{hyp:swim}, and~\ref{hyp:joint} hold,
as well as the smallness condition~\eqref{eq:cond-small-kappa} to ensure well-posedness,
and further let the quantitative mixing Hypothesis~\ref{hyp:part-strong} and the technical Hypothesis~\ref{hyp:swim-large} hold.
Then, for any $h \in W^{1,\infty}(U)^d$,
{as $\e,\delta\downarrow0$, in the regime
\begin{equation}\label{eq:cond-delta-small}
\delta^{-s}\e~\to~0,\qquad\text{for some $s>d+1$},
\end{equation}}
the almost sure weak solution $(u_\e,P_\e)$ of~\eqref{eq:Stokes} 
satisfies almost surely
\[\begin{array}{ccll}
u_\e&\cvf& \bar u,&\text{in $H^1_0(U)^d$},\\
P_\e\mathds1_{U\setminus\Ic_\e(U)}&\cvf&(1-\lambda)\bar P+(1-\lambda)\bb:\D(\bar u)&\\
&&\hspace{1.2cm}+(1-\lambda)\kappa\big( \cc(\D(\bar u))-\fint_U \cc(\D(\bar u))\big),\quad&\text{in $\Ld^2(U)$,}
\end{array}\]
where~$(\bar u,\bar P)\in H^1_0(U)^d\times\Ld^2(U)/\R$ is the unique solution of
\begin{equation}\label{eq:homog-0}
\left\{\begin{array}{ll}
-\Div(2\Bb_\tot(\D(\bar u)))+\nabla \bar P\,=\,(1-\lambda)h,&\text{in $U$},\\
\Div(\bar u)=0,&\text{in $U$},\\
\bar u=0,&\text{on $\partial U$}.
\end{array}\right.
\end{equation}
in terms of the total effective viscosity
\[\Bb_\tot(E)\,:=\,\Bb_\pas E+  \kappa\Bb_\act(E),\]
where we recall that $\lambda,\Bb_\pas,\Bb_\act,\bb,\cc$ are defined in Theorem~\ref{th:homog} above.
\end{theor}

The above shows that the effective stress-strain constitutive relation $E\mapsto2\Bb_\pas E$ is replaced by the nonlinear (non-Newtonian) relation~$E\mapsto2\Bb_\pas E+2\Bb_\act(E)$ due to the effect of particle activity.

Our last result concerns the analysis of the latter in the dilute regime, and we establish the active counterpart to Einstein's effective viscosity formula.
Before stating the result, we need to recall some notation from~\cite{DG-20}: Denote by $Q_r(x) = x + [-\tfrac{r}{2}, \tfrac{r}{2})^d$ the cube of sidelength $r$ centered at $x$, and set $Q(x) = Q_1(x)$, $Q_r = Q_r(0)$, and $Q = Q_1(0)$.  The intensity of the point process $\{x_n\}_n$ is
\[\lambda_1\,:=\,\E\big[{\sharp\{n:x_n\in Q\}}\big],\]
and we further define the two- and three-point intensities as
\begin{eqnarray*}
\lambda_2&:=&\ell^{-2d}\sup_{x\in\R^d}\E\Big[{\sharp\big\{(n,m):n\ne m,~x_n\in Q_\ell,~x_m\in Q_\ell(x)\big\}}\Big],
\\
\lambda_3&:=&\ell^{-3d}\sup_{x_1,x_2\in\R^d}\E\Big[\sharp\big\{(n_0,n_1,n_2):~n_0,n_1,n_2 \text{ pairwise distinct},\\
&&\hspace{4cm}~x_{n_0}\in Q_\ell,~x_{n_1}\in Q_\ell(x_{1}),~x_{n_2}\in Q_\ell(x_{2})\big\}\Big],
\end{eqnarray*}
where we recall that $\ell$ stands for (half) the interparticle distance, cf.~\eqref{eq:def-ell}. 
Note that by definition the intensity can be compared to the particle volume fraction $\lambda$, cf.~\eqref{eq:part-vol-frac},
and we have for $k=2,3$,
\[\lambda_1\simeq\lambda\lesssim\ell^{-d}\qquad\text{and}\qquad\lambda_k \lesssim\ell^{-kd}.\]
We further recall that the two- and three-point densities $g_2,g_3$ of the point process are defined by the following relations, for all $\zeta_2\in C^\infty_c((\R^d)^2)$ and  $\zeta_3\in C^\infty_c((\R^d)^3)$,
\begin{eqnarray}
\iint_{(\R^d)^2}\zeta_2(x,y)\,g_2(x,y)\,dxdy&=&\expecM{\sum_{n\ne m}\zeta_2(x_n,x_m)},\label{eq:def-g2}
\\
\iint_{(\R^d)^3}\zeta_3(x,y,z)\,g_3(x,y,z)\,dxdydz&=&\expecM{\sum_{n_0,n_1,n_2\atop\text{distinct}}\zeta_3(x_{n_0},x_{n_1},x_{n_2})}.\nonumber
\end{eqnarray}
In these terms, the above definitions of the two- and three-point intensities are equivalently written as
\begin{eqnarray}
\lambda_2&=&\sup_{x\in\R^d}\fint_{Q_\ell}\fint_{Q_\ell(x)}g_2(y,z)\,dydz,\label{eq:def-lambda2}\\
\lambda_3&=&\sup_{x_1,x_2\in\R^d}\fint_{Q_\ell}\fint_{Q_\ell(x_1)}\fint_{Q_\ell(x_2)}g_3(x,y,z)\,dxdydz.\nonumber
\end{eqnarray}
With this notation at hand, we may now state the following result on the first-order dilute expansion of the effective viscosity. The expansion of the passive contribution $\Bb_\pas$ was already established in~\cite{DG-20}, and we extend it here to the active setting.

\begin{theor}[Dilute expansion of the effective viscosity]\label{theor:dilute}
Let Hypotheses~\ref{hyp:part}, \ref{hyp:swim}, and~\ref{hyp:joint} hold, and further assume
\begin{enumerate}[(a)]
\item \emph{Independence condition:} The random shapes and swimming forces $\{I_n^\circ,f_n\}_n$ are iid copies of a given random open subset $I^\circ$ and of a random map $f^\circ$, independently of the point process $\{x_n\}_n$.
\item \emph{Decay of correlations:} The point process $\{x_n\}_n$ is strongly mixing, and the two- and three-point correlation functions
\begin{eqnarray*}
h_2(x,y)&:=&g_2(x,y)-\lambda_1^2,\\
h_3(x,y,z)&:=&g_3(x,y,z)-\lambda_1^3-\lambda_1(h_2(x,y)+h_2(x,z)+h_2(y,z)),
\end{eqnarray*}
have algebraic decay: there exists $\gamma>0$ such that for all $x,y,z$,
\begin{equation}\label{eq:h2-bnd}
\quad|h_2(x,y)|\,\lesssim\,\langle x-y\rangle^{-\gamma}, \quad |h_3(x,y,z)| \,\lesssim\,\langle x-y\rangle^{-\gamma} \wedge \langle x-z\rangle^{-\gamma} \wedge \langle y-z\rangle^{-\gamma}.
\end{equation}
\end{enumerate}
Then, we have
\begin{multline}\label{antoine.5}
\Big|\Bb_\tot(E)-\big(E+\lambda_1\Bb_\pas^{(1)}E+\kappa\lambda_1\Bb_\act^{(1)}(E)\big)\Big|\\
\,\lesssim\,\langle E\rangle\Big( \lambda_2|\!\log\lambda_1|+(\lambda_1 \lambda_2+\lambda_1 \lambda_3|\!\log\lambda_1|)^\frac12\Big),
\end{multline}
where we have set
\begin{eqnarray*}
E':2\Bb_\pas^{(1)}E&:=&\expecM{\int_{\partial I^\circ}E'x\cdot\sigma(\psi_E^\circ+Ex,\Sigma_E^\circ)\nu},\\
E':2\Bb_\act^{(1)}(E)&:=&-\expecM{\int_{2B}(\psi_{E'}^\circ+E'x)\cdot f^\circ(E)},
\end{eqnarray*}
in terms of the solution $(\psi_E^\circ,\Sigma_E^\circ)$ of the single-particle problem
\begin{equation}\label{eq:single-part}
\left\{\begin{array}{ll}
-\triangle\psi_{E}^\circ+\nabla \Sigma_{E}^\circ=0,&\text{in $\R^d\setminus I^\circ$},\\
\Div(\psi_{E}^\circ)=0,&\text{in $\R^d\setminus I^\circ$},\\
\D(\psi_{E}^\circ+Ex)=0,&\text{in $I^\circ$},\\
\int_{\partial I^\circ}\sigma(\psi_{E}^\circ+Ex,\Sigma_{E}^\circ)\nu=0,&\\
\int_{\partial I^\circ}\Theta x\cdot\sigma(\psi_{E}^\circ+Ex,\Sigma_{E}^\circ)\nu=0,&\forall\Theta\in\Md^\Skew.
\end{array}\right.
\end{equation}
In particular, in case of spherical particles $I^\circ=B$, these expressions take the explicit forms
\begin{align}
E':2\Bb_\pas^{(1)}E&~:=~(d+2)|B|E':E,\nonumber\\
E':2\Bb_\act^{(1)}(E)&~:=~-\int_{\R^d\setminus B}\big(1-\tfrac1{|x|^{d+2}}\big)E'x\cdot\expec{f^\circ(E)}\label{eq:form-Bact-1}\\
&\hspace{3cm}+\tfrac{d+2}2\int_{\R^d\setminus B}\!\!\big(1-\tfrac1{|x|^2}\big)\tfrac{(x\cdot E'x)x}{|x|^{d+2}}\cdot\expec{f^\circ(E)}.\qedhere
\end{align}
Note that the error bound in \eqref{antoine.5} is $\ll\lambda_1$ provided that $\lambda_2|\!\log\lambda_1|\ll\lambda_1$, and is in particular bounded by $\ell^{-3d/2}\ll\ell^{-d}$ in the regime $\ell\gg1$.
\end{theor}

Alternatively, arguing similarly as for the reformulation~\eqref{eq:def-C+} of~\eqref{eq:def-C} \&~\eqref{eq:def-EF-0},
the dilute active effective viscosity can be written as
\[E':2\Bb^{(1)}_\act(E)\,=\,E':\expecM{\sum_n\tfrac{\mathds1_{I_n}}{|I_n|}\bigg(\int_{\partial I_n}\sigma(\phi_E^n,\Pi_E^n)\nu\otimes_s(x-x_n)-\int_{I_n+B}f_n(E)\otimes_s(x-x_n)\bigg)}.\]
In particular, if $f^\circ(E)$ is rotationally symmetric around the direction $\bar f^\circ(E)$, we find by symmetry,
\begin{equation}\label{eq:altern-Bact-1}
2\Bb^{(1)}_\act(E)\,=\,\lambda_1\expecM{\alpha(E)\Big(\tfrac{\bar f^\circ(E)\otimes\bar f^\circ(E)}{|\bar f^\circ(E)|^2}-\tfrac1d\Id\Big)},
\end{equation}
for some random prefactor $\alpha(E)\in\R$, the sign of which is actually of critical interest and depends on the nature of the swimming mechanism. Note that formula~\eqref{eq:altern-Bact-1} coincides with the active elastic stress first formally derived in~\cite{Saintillan_2008,Haines_2009,Saintillan-10a,Saintillan2014,Potomkin_2016}; see also~\cite{Girodroux-Lavigne-22}.

Before we discuss the possible reduction of viscosity due to particle activity, let us comment on one assumption that we make throughout the paper: the distribution of the positions of particles is stationary and ergodic.
This is physically way too stringent since the dynamics in a bounded domain
%most presumably
would break this stationarity, and therefore allow the local density of particle positions to be non-constant (as indeed observed in experiments). From our homogenization point of view, stationarity is fortunately not essential and the difficulty of the analysis is not there: one could indeed easily weaken stationarity into a notion of ``local stationarity'' and allow the density of active particles to depend on the macroscopic space variable.

\subsection{Viscosity reduction}\label{sec:visco-reduc}
The main motivation of this work is to introduce a nontrivial (and hopefully somewhat realistic) model for active suspensions and rigorously establish a reduction of the viscosity of the plain fluid due to the activity of the particles. 
In the dilute regime, the above formulas provide a rigorous contribution to this celebrated topic. Although the question can  be reduced to understanding the sign of $\alpha(E)$ in~\eqref{eq:altern-Bact-1}, we take a shorter path here based on~\eqref{eq:form-Bact-1}.
To make  computations explicit, we restrict ourselves to the following simplified model for the swimming mechanism, see  e.g.~\cite{Girodroux-Lavigne-22,Haines2008},
\[f^\circ(E)\,:=\,\bar f(E)(|B|^{-1}\mathds1_B-\delta_{x(E)}),\]
where the force exerted by the particle on the surrounding fluid is reduced to a Dirac force at a surrounding point $x(E)\in\R^d\setminus B$.
Consider a shear deformation
$E=\tfrac s2(\ee_1\otimes\ee_2+\ee_2\otimes\ee_1)$ for some $s\in \R$.
In these terms, the formula~\eqref{eq:form-Bact-1} reads
\begin{multline*}
E:\Bb_\act^{(1)}(E)\,=\,\tfrac s2\big(1-\tfrac1{|x(E)|^{d+2}}\big)\big(x(E)_1\bar f(E)_2+x(E)_2\bar f(E)_1\big)\\
-s\tfrac{d+2}2\big(1-\tfrac1{|x(E)|^2}\big)\tfrac{x(E)_1x(E)_2}{|x(E)|^{d+2}}x(E)\cdot\bar f(E).
\end{multline*}
For a spherical particle with its swimming device viewed as a rigid elongated particle, the motion in shear flow has been well-studied on the formal level: in the limit of a strong angular diffusion, a standard heuristic computation shows that the preferred orientation of the particle is $\ee:=\frac1{\sqrt2}(\ee_1+\ee_2)$ or its opposite; see e.g.~\cite[Section~V.8]{Frenkel-46}. This leads us to choosing
\[\bar f(E)=\pm|\bar f(E)|\ee \qquad\text{and}\qquad x(E)=\pm\gamma|x(E)|\ee,\]
where $\gamma=1$ in case when the Dirac swimming force is ahead of the particle (``puller'' particle) and $\gamma=-1$ in case when it is behind the particle (``pusher'' particle). Hence, we  get
\begin{equation*}
E:\Bb_\act^{(1)}(E)\,=\,\gamma\tfrac s2|x(E)||\bar f(E)|\Big(1-\tfrac{d+2}2\tfrac1{|x(E)|^{d}}+\tfrac d2\tfrac1{|x(E)|^{d+2}}\Big),
\end{equation*}
which is negative (resp. positive) in case of a pusher (resp. puller).  This is in full agreement with well-known experiments and predictions of~\cite{Sokolov-Aranson-09, Sokolov2009, Clement-15}.

To conclude on the extent of the possible viscosity reduction, we come back to Theorem~\ref{theor:dilute} in case $\ell \gg 1$. As stated, the error bound in the result is then of order $O(\ell^{-3d/2})$,
\begin{align}
\label{eq:dilute_expansion_interpretation}
\Bb_\tot(E) = \big( 1 + \tfrac{d+2}{2}|B| \lambda_1 \big)E + \kappa \lambda_1 \Bb_\act^{(1)}(E) + O(\ell^{-3d/2}).
\end{align}
If $\Bb_\act^{(1)}(E) < 0$, we infer that the total effective viscosity is smaller than the viscosity of the plain fluid,
\[ E:\Bb_\tot(E) \,<\, |E|^2,\]
provided that the activity is strong enough $\kappa \gg 1$ in such a way that the active contribution exceeds the passive one.
Moreover, as $\lambda_1 = O(\ell^{-d})$, we see that the viscosity reduction could become of order $1$ if $\kappa = O(\ell^d)$. This drastic reduction of viscosity is however prohibited by the (only nearly optimal) condition~\eqref{eq:cond-small-kappa},  which is used to ensure the well-posedness of the microscopic model. In some sense, the above analysis can be compared to the enhancement of elastostriction by active charges analyzed in~\cite{FGLP}.

\bigskip\noindent
{\bf Outline of the article.}
The article is organized as follows.
Section~\ref{sec:cor} is dedicated to the introduction of correctors, the key quantities in homogenization.
The proofs of the homogenization results of Theorems~\ref{th:homog} and~\ref{theor:loc-nonlin-diag} are the object of Section~\ref{sec:hom}, while the dilute analysis and the proof of Theorem~\ref{theor:dilute} are postponed to Section~\ref{sec:dil}.

\bigskip\noindent
{\bf Notation.}
\begin{enumerate}[---]
\item For vector fields $u, u'$, and matrix fields $T, T'$, we set $(\nabla u)_{ij} = \nabla_j u_i$, $\Div (T)_i = \nabla_j T_{ij}$, $T : T' = T_{ij} T'_{ij}$, $(u \otimes u')_{ij} = u_i u'_j$, $(T^s)_{ij} = \frac12(T_{ij} + T_{ji})$, $\D(u) = (\nabla u)^s$, $(u \otimes_s u') = (u \otimes u')^s$. For a 3-tensor field $S$, the matrix $\Div(S)$ is defined by $\Div(S)_{ij} = \nabla_k S_{ijk}$. For a matrix $E$ and a vector field $u$, we write $\partial_E u = E: \nabla u$. We systematically use Einstein's summation convention on repeated indices. 
\item For a vector field $u$ and a scalar field $P$, we recall the notation $\sigma(u,P) = 2\D(u) - P \Id$ for the Cauchy stress tensor. 
\item We denote by $C \ge 1$ any constant that only depends on dimension $d$, on the constant~$\vartheta$ in Hypothesis~\ref{hyp:part}, and on the reference domain $U$. We use the notation $\lesssim$ (resp. $\gtrsim$) for $\le C \times $(resp. $\ge \frac1{C} \times$) up to such a multiplicative constant $C$. We write $\ll$ (resp. $\gg$) for $\le C \times $ (resp. $\ge C \times $) up to a sufficiently large multiplicative constant $C$. We add subscripts to $C$, $\lesssim$, $\gtrsim$ in order to indicate dependence on other parameters.
\item We write $\Md_0 \subset \R^{d \times d}$ for the subset of trace-free matrices, $\Md_0^\Sym$ for the subset of symmetric trace-free matrices, and $\Md^\Skew$ for the subset of skew-symmetric matrices. 
\item The ball centered at $x$ of radius $r$ in $\R^d$ is denoted by $B_r(x)$, and we set $B(x) = B_1(x)$, $B_r = B_r(0)$ and $B = B_1(0)$.
\item We use the standard notation $\langle a\rangle:=(1+|a|^2)^{1/2}$.
\end{enumerate}

\section{Corrector problems}\label{sec:cor}

We start by recalling the relevant correctors for the passive suspension problem as introduced in~\cite{D21,DG-21a, DG-21,DG-21b}, and then we define new correctors for the active problem.

\subsection{Passive corrector problems}
Correctors for passive suspensions were first defined in~\cite{DG-21} and are key to the definition of the associated effective viscosity $\Bb$, cf.~\eqref{eq:Stokes-pas-hom}.

\begin{lem}[Passive correctors \cite{DG-21}]\label{lem:pass-cor}
Let Hypothesis~\ref{hyp:part} hold.
For all $E\in\Md_0^\Sym$, 
there exist a unique random field $\psi_E\in \Ld^2(\Omega;H^1_\loc(\R^d)^d)$ and a unique pressure field $\Sigma_E\in  \Ld^2(\Omega;\Ld^2_\loc(\R^d\setminus\Ic))$ such that:
\begin{enumerate}[---]
\item almost surely, realizations of $\psi_E$ and $\Sigma_E$ satisfy
\begin{equation}\label{e.cor-1}
\quad\left\{\begin{array}{ll}
-\triangle\psi_E+\nabla \Sigma_E=0,&\text{in $\R^d\setminus\Ic$},\\
\Div(\psi_E)=0,&\text{in $\R^d\setminus\Ic$},\\
\D(\psi_E+Ex)=0,&\text{in $\Ic$},\\
\int_{\partial I_n}\sigma(\psi_E+Ex,\Sigma_E)\nu=0,&\forall n,\\
\int_{\partial I_n}\Theta(x-\e x_n)\cdot\sigma(\psi_E+Ex,\Sigma_E)\nu=0,&\forall n,\,\forall\Theta\in\Md^\Skew,
\end{array}\right.
\end{equation}
\item the corrector gradient $\nabla\psi_E$ and the pressure $\Sigma_E\mathds1_{\R^d\setminus \Ic}$ are stationary,
with
\begin{gather*}
\expecm{\nabla\psi_E}=0,\quad\expecm{\Sigma_E\mathds1_{\R^d \setminus \Ic}}=0,\\
\expecm{|\nabla\psi_E|^2}+\expecm{\Sigma_E^2\mathds{1}_{\R^d \setminus \Ic}}\,\lesssim\,\lambda |E|^2,
\end{gather*}
and with the anchoring condition $\int_{B}\psi_E=0$.
\end{enumerate}
In addition, the following properties hold.
\begin{enumerate}[(i)] 
\item \emph{Ergodic theorem:} almost surely,
\begin{equation*}\begin{array}{rcll}
(\nabla\psi_E)(\tfrac\cdot\e)&\rightharpoonup&\expec{\nabla\psi_E}=0,\quad &
 \\\vspace{-0.3cm}\\
(\Sigma_E\mathds1_{\R^d \setminus \Ic})(\tfrac\cdot\e)&\rightharpoonup&\expecm{\Sigma_E\mathds1_{\R^d \setminus \Ic}}=0,\quad&\text{weakly in $\Ld^2_\loc(\R^d)$ as $\e\downarrow0$}.
\end{array}
\end{equation*}
\item \emph{Sublinearity:} almost surely, for all $q<\frac{2d}{d-2}$,
\[\e\psi_E(\tfrac\cdot\e)\,\to\,0\quad   \text{strongly in $\Ld^q_\loc(\R^d)^d$ as $\e\downarrow0$.} \qedhere\]
\end{enumerate}
\end{lem}

As in~\cite{D21,DG-21b}, in view of quantitative estimates, we further need to define an associated extended flux $J$ and a flux corrector $\Ups$. More precisely, $J$ is a solenoidal extension of the natural flux $\sigma(\psi_E+Ex,\Sigma_E)$ outside the particles, and $\Ups$ is the associated vector potential in the Coulomb gauge.

\begin{lem}[Passive flux correctors \cite{D21,DG-21b}]\label{lem:Jzet}
Let Hypothesis~\ref{hyp:part} hold.
For all $E \in \mathbb{M}_0^{\mathrm{sym}}$, there exists a stationary random $2$-tensor field $J_E = (J_{E;ij})_{1 \le i,j \le d}$ with finite second moment such that, almost surely,
\begin{eqnarray*}
J_E \mathds1_{\R^d \setminus \Ic} &=& \sigma(\psi_E + Ex, \Sigma_E) \mathds1_{\R^d \setminus \Ic}, \\
\Div(J_E) &=& 0,
\end{eqnarray*}
and for all $n$, 
\begin{align}
\label{eq:J_E_I_n}
 \|J_E\|_{\Ld^2(I_n)} \,\lesssim\,  \|\sigma(\psi_E+Ex,\Sigma_E)\|_{\Ld^2((I_n + B) \setminus I_n)}.
 \end{align} 
Moreover, there exists a unique random $3$-tensor field $\Ups_E = (\Upsilon_{E;ijk})_{1 \le i,j,k \le d}$ such that:
\begin{enumerate}[---]
\item for all $i,j,k$, almost surely, realizations of $\Ups_{E;ijk}$ belong to $H^1_{\loc}(\R^d)$ and satisfy
\begin{equation}
\label{eq:Delta_zeta}
- \Delta \Ups_{E;ijk} \,=\, \partial_j J_{E;ik} - \partial_k J_{E;ij};
\end{equation}
\item the random field $\nabla \Ups_E$ is stationary, has vanishing expectation, has finite second moment, and satisfies the anchoring condition $\fint_B \Ups_E = 0$.
\end{enumerate}
\smallskip
In addition, the following properties hold.
\begin{enumerate}[(i)]
\item \emph{Skew-symmetry:}
almost surely, $\Ups_{E;ijk} = - \Ups_{E;ikj}$ for all $i,j,k$.
\item \emph{Vector potential:}
almost surely, for all $i$,
\[\qquad \Div(\Ups_{E;i}) = J_{E;i} - \mathbb{E}[J_{E;i}],\]
where we have set $\Ups_{E;i} = (\Ups_{E;ijk})_{1 \le j,k \le d}$ and $J_{E;i} = (J_{E;ij})_{1 \le j \le d}$.
\item \emph{Ergodic theorem:} almost surely,
\begin{equation*} 
(\nabla\Ups_E)(\tfrac\cdot\e)\,\rightharpoonup\,\expec{\nabla\Ups_E}=0\text{weakly in $\Ld^2_\loc(\R^d)$ as $\e\downarrow0$}.
\end{equation*}
\item \emph{Sublinearity:} almost surely, for all $q<\frac{2d}{d-2}$,
\[\qquad\e\Ups_E(\tfrac{\cdot}{\e}) \rightharpoonup 0\quad \text{strongly in $\Ld^q_\loc(\R^d)^d$ as $\e\downarrow0$.} \]
\item \emph{Effective constants:} the expectation of $J_E$ takes the form
\[\qquad\expec{J_E}= 2 \Bb_\pas E + (\bb:E) \Id,\]
in terms of the effective constants $\Bb_\pas$ and $\bb$ defined in~\eqref{eq:def-B} and~\eqref{eq:def-b}. 
\qedhere
\end{enumerate}
\end{lem}

\subsection{Active corrector problems}
We turn to the definition of suitable correctors for the active suspension problem.
These new correctors characterize the contribution of swimming forces of the particles in a uniform fluid velocity gradient $E\in\Md_0^\Sym$.
\begin{lem}\label{lem:cor-act}
Let Hypotheses~\ref{hyp:part}, \ref{hyp:swim}, and~\ref{hyp:joint} hold.
For all $E\in\Md_0^\Sym$, there exist a unique random field $\phi_E\in \Ld^2(\Omega;H^1_\loc(\R^d)^d)$ and a unique pressure field $\Pi_E\in  \Ld^2(\Omega;\Ld^2_\loc(\R^d\setminus\Ic))$ such that:
\begin{enumerate}[---]
\item almost surely, realizations of $\phi_E$ and $\Pi_E$ satisfy
\begin{equation}\label{eq:phi-cor}
\left\{
\begin{array}{ll}
- \Delta \phi_E + \nabla \Pi_E =\sum_n f_n(E), &\text{in $\R^d \setminus \Ic$},\\
\Div(\phi_E) = 0, &\text{in $\R^d \setminus \Ic$},\\
\D(\phi_E) = 0, &\text{in $\Ic$},\\
\int_{\partial I_n}\sigma(\phi_E,\Pi_E)\nu+\bar f_n(E)=0,&\forall n,\\
\int_{\partial I_n}\Theta(x- x_n)\cdot\sigma(\phi_E,\Pi_E)\nu  +\Theta:\tilde f_n(E)=0,&\forall n,\,\forall\Theta\in\Md^\Skew,
\end{array}
\right.
\end{equation}
\item the corrector gradient $\nabla\phi_E$ and the pressure $\Pi_E\mathds1_{\R^d\setminus \Ic}$ are stationary, with
\begin{gather}
\expecm{\nabla\phi_E}=0,\quad\expecm{\Pi_E\mathds1_{\R^d \setminus \Ic}}=0,\nonumber\\
\textstyle\expecm{|\nabla\phi_E|^2}+\expecm{\Pi_E^2\mathds{1}_{\R^d \setminus \Ic}}\,\lesssim\, \lambda\langle E\rangle^{2},\label{eq:bnd-unif-E-phiE}
\end{gather}
and with the anchoring condition $\int_{B}\phi_E=0$.
\end{enumerate}
In addition, the following properties hold.
\begin{enumerate}[(i)]
\item \emph{Ergodic theorem:} almost surely,
\begin{equation*}\begin{array}{rcll}
(\nabla\phi_E)(\tfrac\cdot\e)&\rightharpoonup&\expec{\nabla\phi_E}=0,
\\\vspace{-0.3cm}\\
(\Pi_E\mathds1_{\R^d \setminus \Ic})(\tfrac\cdot\e)&\rightharpoonup&\expecm{\Pi_E\mathds1_{\R^d \setminus \Ic}}=0,
\quad&\text{weakly in $\Ld^2_\loc(\R^d)$ as $\e\downarrow0$}.
\end{array}
\end{equation*}
\item \emph{Sublinearity:} almost surely, for all $q<\frac{2d}{d-2}$,
\[\textstyle\e\phi_E(\tfrac\cdot\e)\,\to\,0\qquad \text{strongly in $\Ld^q_\loc(\R^d)$ as $\e\downarrow0$.}\qedhere\]
\end{enumerate}
\end{lem}

\begin{proof}
The argument is similar to that in~\cite[Proposition~2.1]{DG-21} for Lemma~\ref{lem:pass-cor} above, and we only briefly show the needed adaptations: we describe the structure of equation~\eqref{eq:phi-cor}, following the first step of the proof of~\cite[Proposition~2.1]{DG-21}, while the rest of the proof in~\cite{DG-21} is then easily repeated and is skipped here for brevity. More precisely, we only show the following: if $\phi_E$ is a solution of~\eqref{eq:phi-cor} with $\nabla\phi_E,\Pi_E$ stationary with finite second moments, then it satisfies for all stationary fields $v\in \Ld^2(\Omega;H^1_\loc(\R^d)^d)$ with $\Div(v)=0$ and $\D(v)|_{\Ic}=0$,
\begin{equation}\label{eq:reform-phiE-proba}
\expec{\nabla v:\nabla\phi_E}
\,=\,-\expecM{\sum_n\tfrac{\mathds1_{I_n}}{|I_n|}\int_{I_n+B}v\cdot f_n(E)},
\end{equation}
where, by the Cauchy--Schwarz and the Poincaré inequalities, together with the hardcore assumption
and the property $\int_{I_n+B} f_n(E)=0$ (cf.~\eqref{eq:neutrality}), the right-hand side is bounded by
\begin{equation*}
\bigg|\expecM{\sum_n\tfrac{\mathds1_{I_n}}{|I_n|}\int_{I_n+B}v\cdot f_n(E)}\bigg|
\,\lesssim\,\expecm{|\nabla v|^2}^\frac12\bigg(\lambda\,\expecM{\int_{I_\circ+B}|f_\circ(E)|^2}\bigg)^\frac12.
\end{equation*}
To prove this claim, we start by noting that the hardcore assumption allows to construct almost surely for all $R>0$ a cut-off function $\eta_R$ such that
\[\eta_R|_{B_R}=1,\qquad\eta_R|_{\R^d\setminus B_{R+5}}=0,\qquad|\nabla\eta_R|\lesssim1,\]
and such that $\eta_R$ is constant in $I_n+  B$ for all $n$. As $\phi_E$ is divergence-free, testing equation~\eqref{eq:phi-cor} with $\eta_Rv$ and integrating by parts, we find
\begin{multline}\label{eq:test-as-etaR-phiE}
2\int_{\R^d \setminus \Ic}\D(\eta_Rv):\D(\phi_E)-\int_{\R^d\setminus \Ic}\nabla\eta_R\cdot v\Pi_E\\
\,=\,\sum_n\int_{(I_n+B)\setminus I_n}\eta_Rv\cdot f_n(E)-\sum_n\int_{\partial I_n}\eta_Rv\cdot\sigma(\phi_E,\Pi_E)\nu.
\end{multline}
Since $\D(\phi_E)|_\Ic=0$ and $\nabla \eta_R|_{\Ic}=0$, the left-hand side of \eqref{eq:test-as-etaR-phiE} writes
$$
2\int_{\R^d \setminus \Ic}\D(\eta_Rv):\D(\phi_E)-\int_{\R^d\setminus \Ic}\nabla\eta_R\cdot v\Pi_E=\int_{\R^d}\nabla(\eta_Rv):\nabla\phi_E-\int_{\R^d}\nabla\eta_R\cdot v\Pi_E.
$$
As $\eta_R$ is constant in $I_n+ B$ and  $\D(v)=0$ in $I_n$, we can rewrite the last right-hand side term as
\begin{multline*}
\int_{\partial I_n}\eta_Rv\cdot\sigma(\phi_E,\Pi_E)\nu
\,=\,\eta_R(x_n)\bigg(\Big(\fint_{I_n}v\Big)\cdot\int_{\partial I_n}\sigma(\phi_E,\Pi_E)\nu\\
+\Big(\fint_{I_n}(\nabla v)^\Skew\Big):\int_{\partial I_n}\sigma(\phi_E,\Pi_E)\nu\otimes(x-x_n)\bigg),
\end{multline*}
and thus, in view of the boundary conditions for $(\phi_E,\Pi_E)$ in~\eqref{eq:phi-cor}, using the notation~\eqref{eq:extension-fn},
\begin{eqnarray*}
\int_{\partial I_n}\eta_Rv\cdot\sigma(\phi_E,\Pi_E)\nu
&=&-\eta_R(x_n)\bigg(\Big(\fint_{I_n}v\Big)\cdot\bar f_n(E)
+\Big(\fint_{I_n}\nabla v\Big):\tilde f_n(E)\bigg)\\
&=&-\int_{I_n}\eta_Rv\cdot  f_n(E).
\end{eqnarray*}
Inserting this into~\eqref{eq:test-as-etaR-phiE}, we get
\begin{equation}\label{eq:transform-eqn-phie-not}
\int_{\R^d}\nabla(\eta_Rv):\nabla\phi_E-\int_{\R^d}\nabla\eta_R\cdot v\Pi_E
\,=\,\sum_n\int_{I_n+B}\eta_Rv\cdot f_n(E).
\end{equation}
Expanding the gradient in the left-hand side, passing to the limit $R\uparrow\infty$, and using the stationarity of~$\nabla\phi_E,\Pi_E$ and of~$v$, the claim~\eqref{eq:reform-phiE-proba} follows. From there, we may then refer to the proof in~\cite[Proposition~2.1]{DG-21}.
\end{proof}

Next, as in Lemma~\ref{lem:Jzet}, we further need to define an associated extended flux and a flux corrector for the active suspension problem. The difficulty, however, is that even in the fluid domain $\R^d\setminus\Ic$ the flux $\sigma(\phi_E,\Pi_E)$ is not divergence-free. It thus needs to be first suitably compensated and we are led to defining the following auxiliary corrector~$\gamma_E$.
The proof of the upcoming lemma (which is a simplified version of Lemma~\ref{lem:cor-act}) is straightforward and  skipped for brevity.
 
\begin{lem}\label{lem:cor-gamma}
Let Hypotheses~\ref{hyp:part}, \ref{hyp:swim}, and~\ref{hyp:joint} hold.
For all $E\in\Md_0^\Sym$,
there exists a unique random field $\gamma_E\in \Ld^2(\Omega;H^1_\loc(\R^d)^d)$ such that:
\begin{enumerate}[---]
\item almost surely, the realizations of $\gamma_E$ satisfy
\begin{equation}\label{eq:gamma-cor}
-\triangle \gamma_E =  \sum_nf_n(E);
\end{equation}
\item the gradient field $\nabla\gamma_E$ is stationary with
\begin{gather*}
\expecm{\nabla\gamma_E}=0,\qquad
\expecm{|\nabla\gamma_E|^2}\,\lesssim\,\lambda\langle E\rangle^2,
\end{gather*}
and with the anchoring condition $\int_{B}\gamma_E=0$.
\end{enumerate}
In addition, the following properties hold.
\begin{enumerate}[(i)] 
\item \emph{Ergodic theorem:} almost surely,
\begin{equation*}
\begin{array}{rcll}
(\nabla\gamma_E)(\tfrac\cdot\e)&\rightharpoonup&\expec{\nabla\gamma_E}=0 \quad &\text{weakly in $\Ld^2_\loc(\R^d)^d$ as $\e\downarrow0$}.
\end{array}
\end{equation*}
\item \emph{Sublinearity:} almost surely, for all $q<\frac{2d}{d-2}$,
\[ \e\gamma_E(\tfrac\cdot\e)\,\to\,0\qquad \text{strongly in $\Ld^q_\loc(\R^d)^d$ as $\e\downarrow0$}.\qedhere\]
\end{enumerate}
\end{lem}
Using the above-defined $\gamma_E$ to compensate the divergence of the flux $\sigma(\phi_E,\Pi_E)$ in the fluid domain, and extending it similarly as in Lemma~\ref{lem:Jzet} inside the particles, we are now in the position to define a flux~$K_E$, a  divergence-free compensated flux~$L_E$, and the associated flux corrector~$\theta_E$.

\begin{lem}\label{lem:KE}
Under Hypotheses~\ref{hyp:part}, \ref{hyp:swim}, and~\ref{hyp:joint}, there exists a stationary random symmetric $2$-tensor field $K_E = (K_{E;ij})_{1 \le i,j \le d}$ with finite second moment 
\begin{equation}\label{eq:K_E_Om}
\expecm{|K_E|^2} \,\lesssim\, \lambda\langle E\rangle^2,
\end{equation}
such that, almost surely,
\begin{eqnarray}
K_E \mathds1_{\R^d \setminus \Ic} &=& \sigma(\phi_E , \Pi_E) \mathds1_{\R^d \setminus \Ic},\nonumber\\
\Div(K_E) &=&-\sum_nf_n(E).\label{eq:KE-def-prop}
\end{eqnarray}
Moreover, the expectation of $K_E$ takes the form
\begin{equation}\label{eq:prop-exp-KE}
\qquad\expec{K_E}\,=\, 2\Cc(E) +\cc(E)\Id,
\end{equation}
in terms of the effective tensors $\Cc(E)$ and $\cc(E)$ defined in~\eqref{eq:def-C}. Next, for the divergence-free compensated flux
\begin{equation}\label{eq:def-L}
L_E\,:=\,K_E-\expec{K_E}-\nabla\gamma_E,
\end{equation}
there exists a unique random $3$-tensor field $\theta_E = (\theta_{E;ijk})_{1 \le i,j,k \le d}$ such that:
\begin{enumerate}[---]
\item for all $i,j,k$, almost surely, realizations of $\theta_{E;ijk}$ belong to $H^1_{\loc}(\R^d)$ and satisfy
\begin{equation}
\label{eq:Delta_theta}
- \triangle \theta_{E;ijk} \,=\, \partial_j L_{E;ik} - \partial_k L_{E;ij};
\end{equation}
\item the random field $\nabla \theta_E$ is stationary, has vanishing expectation, has finite second moment, and satisfies the anchoring condition $\fint_B \theta_E = 0$. 
\end{enumerate}
\smallskip
In addition, the following properties hold.
\begin{enumerate}[(i)]
\item \emph{Skew-symmetry:} almost surely, $\theta_{E;ijk} = -\theta_{E;ikj}$ for all $i,j,k$.
\item \emph{Vector potential:} almost surely, for all $i$,
\[\qquad \Div(\theta_{E;i}) = L_{E;i}, \]
where we have set $\theta_{E;i} = (\theta_{E;ijk})_{1 \le j,k \le d}$ and $L_{E;i} = (L_{E;ij})_{1 \le j \le d}$.
\item \emph{Ergodic theorem:} almost surely,
\begin{equation*}
\begin{array}{rcll}
(\nabla\theta_E)(\tfrac\cdot\e)&\rightharpoonup&\expec{\nabla\theta_E}=0 \quad &\text{weakly in $\Ld^2_\loc(\R^d)^d$ as $\e\downarrow0$}.
\end{array}
\end{equation*}
\item \emph{Sublinearity:} almost surely, for all $q<\frac{2d}{d-2}$,
\[ \e\theta_E(\tfrac\cdot\e)\,\to\,0\qquad \text{strongly in $\Ld^q_\loc(\R^d)^d$ as $\e\downarrow0$}.\qedhere\]
\end{enumerate}
\end{lem}
\begin{proof}
We split the proof into three main steps.

\medskip
\step1 Construction and properties of the extended flux $K_E$.\\
For all $g\in C^{\infty}_c(\R^d)$, equation~\eqref{eq:phi-cor} yields by integration by parts,
\begin{eqnarray*}
\int_{\R^d \setminus \Ic}\nabla g:\sigma(\phi_E,\Pi_E)
&=&\sum_n \int_{(I_n+B)\setminus I_n}g\cdot f_n(E)
-\sum_n\int_{\partial I_n}g\cdot\sigma(\phi_E,\Pi_E)\nu,
\end{eqnarray*}
and thus, using boundary conditions for $(\phi_E,\Pi_E)$ in~\eqref{eq:phi-cor} and recalling the notation~\eqref{eq:extension-fn},
\begin{multline}\label{eq:div_varphi}
\int_{\R^d \setminus \Ic}\nabla g:\sigma(\phi_E,\Pi_E)+\sum_n\int_{\partial I_n}\bigg(g-\Big(\fint_{ I_n}g\Big)-\Big(\fint_{I_n}(\nabla g)^\Skew\Big)(x-x_n) \bigg)\cdot\sigma(\phi_E,\Pi_E)\nu\\
\,=\,\sum_n \int_{I_n+B}g\cdot f_n(E).
\end{multline}
For all $n$, we may then consider the following Neumann problem,
\begin{align}
\label{eq:Neumann_phi}
\left\{\begin{array}{ll}
-\triangle \phi_E^n+\nabla \Pi_E^n= f_n(E) ,&\text{in $I_n$},\\
\Div(\phi_E^n)=0,&\text{in $I_n$},\\
\sigma(\phi_E^n,\Pi_E^n)\nu=\sigma(\phi_E,\Pi_E)\nu, &\text{on $\partial I_n$}.
\end{array}\right.
\end{align}
Note that this only defines $\phi_E^n$ up to a rigid motion, which is fixed by choosing~$\phi_E^n$ with
$\fint_{I_n} \phi_E^n = 0$ and $\fint_{I_n} \nabla \phi_E^n \in \Md_0^\Sym$.
Assuming that this Neumann problem~\eqref{eq:Neumann_phi} is well-posed, and setting
\begin{eqnarray}
\tilde q_E&:=&\D(\phi_E)+\sum_n\D(\phi_E^n)\mathds1_{I_n},\nonumber\\
\tilde\Pi_E&:=&\Pi_E\mathds1_{\R^d\setminus\Ic}+\sum_n\Pi_E^n\mathds1_{I_n},\label{eq:def-tilde-qPi}
\end{eqnarray}
we easily deduce that the extended flux
\begin{equation}\label{eq:constr-KE-flux}
K_E\,:=\,2\tilde q_E-\tilde\Pi_E\Id
\end{equation}
satisfies the desired relations~\eqref{eq:KE-def-prop} (recall that $\D(\phi_E)|_\Ic=0$).
It remains to check the well-posedness of this Neumann problem and to establish the bound~\eqref{eq:K_E_Om}.
Note that this well-posedness, together with the uniqueness in our construction of the pressures $\{\Pi_E^n\}_n$ below, ensures that $\tilde q_E$ and $\tilde \Pi_E$ are stationary.
We split the proof into three further substeps.

\medskip
\substep{1.1}  Well-posedness of the Neumann problem~\eqref{eq:Neumann_phi} for $\phi_E^n$.\\
The weak formulation of \eqref{eq:Neumann_phi} reads as follows: for all divergence-free fields $v \in H^1(I_n)^d$, 
\begin{align}\label{eq:weak_formul_Neumann_phi}
2 \int_{I_n} \D(v):\D(\phi_E^n) = \mathcal{L}_E(v), 
\end{align} 
where $\mathcal{L}_E$ stands for the linear form
\begin{equation*} 
\mathcal{L}_E(v) \,:=\,\int_{I_n} v\cdot f_n(E)+\int_{\partial I_n} v \cdot  \sigma(\phi_E, \Pi_E) \nu.
\end{equation*}
Using the boundary conditions for $(\phi_E,\Pi_E)$ in~\eqref{eq:phi-cor} and recalling the notation~\eqref{eq:extension-fn}, the latter can be reformulated as
\begin{equation}\label{e.lin-form-reform}
\mathcal{L}_E(v) \,=\, \int_{\partial I_n} \bigg(v -\Big(\fint_{I_n}v\Big)-\Big(\fint_{I_n}(\nabla v)^\Skew\Big)(x-x_n)\bigg)\cdot  \sigma(\phi_E, \Pi_E) \nu.
\end{equation}
We shall show that~\eqref{eq:weak_formul_Neumann_phi} is well-posed in the following Hilbert subspace of $H^1(I_n)^d$,
\[\mathcal{T} := \bigg\{v \in H^1(I_n)^d~:~ \Div(v) = 0, ~\fint_{I_n} v = 0,~\text{and}~\fint_{I_n} \nabla v \in \Md_0^\Sym \bigg\}.\]
First note that Korn's inequality yields for all $v \in \mathcal{T}$,
\[ \|\nabla v\|_{\Ld^2(I_n)} \,\lesssim\, \|\!\D(v)\|_{\Ld^2(I_n)}, \]
which entails that the bilinear form $(v,\psi) \mapsto 2 \int_{I_n} \D(v) :\D(\psi)$ is continuous and coercive on $\mathcal T\times \mathcal T$. By the Lax--Milgram theorem, in order to prove the well-posedness of~\eqref{eq:weak_formul_Neumann_phi}, it remains to show that the linear form $\mathcal{L}_E$ is continuous on $\mathcal T$ as well.

\medskip\noindent
In order to deal with the Neumann condition, we consider an extension map
\begin{align*}
T_n ~:~ \big\{v \in H^1(I_n)^d : \Div(v) = 0 \big\}~\to~\big\{ v \in H^1_0(I_{n} +  B)^d : \Div(v) = 0\big\},
\end{align*}
such that $T_nv|_{I_n} = v|_{I_n}$ and 
\begin{equation}\label{eq:extension-L2-est}
\|\nabla T_nv\|_{\Ld^2(I_n +  B)} \,\lesssim\, \|\nabla v\|_{\Ld^2(I_n)}.
\end{equation}
Smuggling in $T_n$ in~\eqref{e.lin-form-reform}, integrating by parts, and inserting equation~\eqref{eq:phi-cor}, the linear form $\mathcal L_E$ can be rewritten as 
\begin{eqnarray*}
\mathcal{L}_E(v) &=& -\int_{(I_n+B)\setminus I_n}\!\!\Div\bigg[\sigma(\phi_E, \Pi_E)\, T_n\bigg(\!v -\Big(\fint_{I_n}v\Big)-\Big(\fint_{I_n}(\nabla v)^\Skew\Big)(x-x_n)\!\bigg)\bigg]\\
&=&\int_{(I_n+B)\setminus I_n}T_n\bigg(v -\Big(\fint_{I_n}v\Big)-\Big(\fint_{I_n}(\nabla v)^\Skew\Big)(x-x_n)\bigg)\cdot f_n(E)\\
&&-2\int_{(I_n+B)\setminus I_n}\D(\phi_E):\D\bigg[ T_n\bigg(\!v -\Big(\fint_{I_n}v\Big)-\Big(\fint_{I_n}(\nabla v)^\Skew\Big)(x-x_n)\!\bigg)\bigg].
\end{eqnarray*}
Hence, in view of~\eqref{eq:extension-L2-est} and of Korn's inequality,
\begin{equation*}
|\mathcal{L}_E(v)| \,\lesssim\,\|\!\D(v)\|_{\Ld^2(I_n)} \Big( \|\!\D(\phi_E)\|_{\Ld^2(I_n +  B)} +  \|f_n(E)\|_{\Ld^2(I_n +  B)} \Big),
\end{equation*}
which proves the continuity of $\mathcal L_E$ on $\calT$.

\medskip\noindent
By the Lax--Milgram theorem, we deduce that there exists a unique solution $\phi_E^n\in\calT$ of~\eqref{eq:weak_formul_Neumann_phi}, and that it satisfies
\[\|\!\D(\phi_E^n)\|_{\Ld^2(I_n)}\,\lesssim\,\|\!\D(\phi_E)\|_{\Ld^2(I_n+B)}+\|f_n(E)\|_{\Ld^2(I_n+B)}.\]
By Korn's inequality, this further yields
\begin{equation}\label{eq:bound-phin}
\|\nabla\phi_E^n\|_{\Ld^2(I_n)}\,\lesssim\,\|\!\D(\phi_E)\|_{\Ld^2(I_n+B)}+\|f_n(E)\|_{\Ld^2(I_n+B)}.
\end{equation}

\medskip
\substep{1.2} Construction of the pressure.\\
As~\eqref{eq:def-tilde-qPi} reads $\tilde q_E=\D(\phi_E)+\D(\phi_E^n)\mathds1_{I_n}$ in $I_n+B$,
combining equation~\eqref{eq:phi-cor} for $\phi_E$ and equation~\eqref{eq:weak_formul_Neumann_phi} for $\phi_E^n$, we find for all $v\in C^\infty_c(I_n+B)^d$ with $\Div(v)=0$,
\[2\int_{\R^d}\D(v):\tilde q_E=\int_{\R^d}v\cdot f_n(E).\]
Appealing e.g.\@ to~\cite[Proposition~12.10]{JKO94}, we deduce that there exists an associated pressure field $\Pi_E^n\in\Ld^2_\loc(I_n+B)$, which is unique up to an additive constant, such that for all test functions $v\in C^\infty_c(I_n+B)^d$,
\begin{equation}\label{eq:defin-SigmaEn-eqn}
\int_{\R^d}\D(v):(2\tilde q_E-\Pi_E^n\Id )=\int_{\R^d}v\cdot f_n(E).
\end{equation}
Since for all $v\in C^\infty_c((I_n+B)\setminus I_n)^d$ we get
\[\int_{\R^d}\D(v):\sigma(\phi_E,\Pi_E^n)\,=\,\int_{\R^d}\D(v):(2\tilde q_E-\Pi_E^n\Id)\,=\,\int_{\R^d}v\cdot f_n(E),\]
and comparing with equation~\eqref{eq:phi-cor},
we deduce that $\Pi_E^n$ can be chosen uniquely to coincide with $\Pi_E$ on $(I_n+B)\setminus I_n$.

\medskip\noindent
It remains to  estimate  the above-constructed pressure. Using that $\Pi_E^n$ coincides with $\Pi_E$ in $(I_n+B)\setminus I_n$, we can split 
\begin{eqnarray*}
\int_{I_n+B} \Pi_E^n&=& \int_{(I_n+B)\setminus I_n} \Pi_E+\int_{I_n}\Pi_E^n \\
&=&\int_{(I_n+B)\setminus I_n} \Pi_E+\int_{I_n}\Big(\Pi_E^n-\fint_{I_n+B} \Pi_E^n\Big)+|I_n| \fint_{I_n+B} \Pi_E,
\end{eqnarray*}
to the effect that
$$
(|I_n+B|-|I_n|) \Big|\fint_{I_n+B} \Pi_E^n\Big| \,\le \, \Big| \int_{(I_n+B)\setminus I_n} \Pi_E\Big|+ \Big|\int_{I_n}\Big(\Pi_E^n-\fint_{I_n+B} \Pi_E^n\Big)\Big|.
$$
Hence, 
\begin{eqnarray*}
\|\Pi_E^n\|_{\Ld^2(I_n)}&\lesssim&\Big\|\Pi_E^n-\fint_{I_n+B}\Pi_E^n\Big\|_{\Ld^2(I_n+B)}+\Big|\fint_{I_n+B}\Pi_E^n\Big|\\
&\lesssim&\Big\|\Pi_E^n-\fint_{I_n+B}\Pi_E^n\Big\|_{\Ld^2(I_n+B)}+\Big|\fint_{(I_n+B)\setminus I_n}\Pi_E\Big|.
\end{eqnarray*}
Starting from~\eqref{eq:defin-SigmaEn-eqn}, a standard argument based on the Bogovskii operator yields
\[\Big\|\Pi_E^n-\fint_{I_n+B}\Pi_E^n\Big\|_{\Ld^2(I_n+B)}\,\lesssim\,\|\tilde q_E\|_{\Ld^2(I_n+B)},\]
so that the above becomes
\[\|\Pi_E^n\|_{\Ld^2(I_n)}\,\lesssim\,\|\tilde q_E\|_{\Ld^2(I_n+B)}+\|\Pi_E\|_{\Ld^2((I_n+B)\setminus I_n)}.\]
Combining this with~\eqref{eq:bound-phin}, we get
\begin{equation}\label{eq:bounds-phinPin}
\|\nabla\phi_E^n\|_{\Ld^2(I_n)}+\|\Pi_E^n\|_{\Ld^2(I_n)}\,\lesssim\,\|\sigma(\phi_E,\Pi_E)\|_{\Ld^2((I_n+B)\setminus I_n)}+\|f_n(E)\|_{\Ld^2(I_n+B)}.
\end{equation}

\medskip
\substep{1.3} Proof of~\eqref{eq:K_E_Om}.\\
By definition~\eqref{eq:constr-KE-flux}, the bound~\eqref{eq:bounds-phinPin} yields for all $n$,
\begin{eqnarray}
\|K_E\|_{\Ld^2(I_n)}&\lesssim&\|\sigma(\phi_E^n,\Pi_E^n)\|_{\Ld^2(I_n)}\nonumber\\
&\lesssim&\|\sigma(\phi_E,\Pi_E)\|_{\Ld^2((I_n+B)\setminus I_n)}+\|f_n(E)\|_{\Ld^2(I_n+B)}.\label{eq:K_E_I_n}
\end{eqnarray}
Hence, for all $R>0$,
\begin{equation*}
\|K_E\|_{\Ld^2(B_R)}^2\,\lesssim\,\|\sigma(\phi_E,\Pi_E)\mathds1_{\R^d\setminus\Ic}\|_{\Ld^2(B_{R+5})}^2+\sum_{n:I_n\cap B_R\ne\varnothing}\int_{I_n+B}|f_n(E)|^2,
\end{equation*}
and thus, by stationarity, letting $R\uparrow\infty$,
\begin{equation*}
\|K_E\|_{\Ld^2(\Omega)}^2\,\lesssim\,\|\sigma(\phi_E,\Pi_E)\mathds1_{\R^d\setminus\Ic}\|_{\Ld^2(\Omega)}^2+\lambda\,\expecM{\int_{I_\circ+B}|f_\circ(E)|^2}.
\end{equation*}
Combined with~\eqref{eq:bnd-unif-E-phiE}, this yields~\eqref{eq:K_E_Om}.

\medskip
\step{2} Formula for $\expec{K_E}$ and definition of $\Cc(E)$ and $\cc(E)$.\\
We split the proof into two further substeps, separately proving formula~\eqref{eq:prop-exp-KE} for $\expec{K_E}$ and establishing the alternative formulas~\eqref{eq:def-C+} for $\Cc(E)$ and $\cc(E)$.

\medskip
\substep{2.1} Proof of~\eqref{eq:prop-exp-KE}.\\
The hardcore assumption allows to construct almost surely for all $R>0$ a cut-off function~$\eta_R$ such that
\[\eta_R|_{B_R}=1,\qquad\eta_R|_{\R^d\setminus B_{R+5}}=0,\qquad|\nabla\eta_R|\lesssim1,\]
and such that $\eta_R$ is constant in $I_n+B$ for all $n$.
By definition of $K_E$ and $(\phi_E,\Pi_E)$, we have
\[\expecm{K_E\mathds1_{\R^d\setminus\Ic}}\,=\,\expecm{\sigma(\phi_E,\Pi_E)\mathds1_{\R^d\setminus\Ic}}\,=\,2\expecm{\!\D(\phi_E)}-\expecm{\Pi_E\mathds1_{\R^d\setminus\Ic}}\Id\,=\,0,\]
and the ergodic theorem then yields almost surely,
\begin{equation}\label{eq:first-EK}
\expec{K_E} \,=\, \expecm{K_E\mathds1_{\Ic}}\,=\, \lim_{R \uparrow\infty}|B_R|^{-1} \int_\Ic \eta_R\,K_E.
\end{equation}
By definition of $K_E$ and the choice of $\eta_R$,  integration by parts and the equation~\eqref{eq:Neumann_phi} for~$(\phi_E^n,\Sigma_E^n)$ yield
\begin{eqnarray*}
\int_{\Ic} \eta_R\,K_E
&=& \sum_n \eta_R(x_n) \int_{I_n}\sigma(\phi_E^n, \Pi_E^n)\\
&=& \sum_n \eta_R(x_n) \int_{\partial I_n} \sigma(\phi_E^n, \Pi_E^n)\nu\otimes(x- x_n)\\
&&+\sum_n \eta_R(x_n) \int_{I_n} f_n(E)\otimes(x- x_n).
\end{eqnarray*}
By the boundary conditions for $(\phi_E^n,\Pi_E^n)$ and recalling the notation~\eqref{eq:extension-fn}, we deduce
\begin{eqnarray*}
\int_{\Ic} \eta_R\,K_E
&=&\int_{\R^d}\eta_R\sum_n\tfrac{\mathds1_{I_n}}{|I_n|}\Big(\tilde f_n(E)+\int_{\partial I_n} \sigma(\phi_E, \Pi_E)\nu\otimes(x-x_n)\Big).
\end{eqnarray*}
Letting $R\uparrow\infty$, the ergodic theorem then entails
\[\expec{K_E}\,=\,\expecM{\sum_n\tfrac{\mathds1_{I_n}}{|I_n|}\Big(\tilde f_n(E)+\int_{\partial I_n} \sigma(\phi_E, \Pi_E)\nu\otimes(\cdot-x_n)\Big)}.\]
Since $K_E$ is symmetric, taking the symmetric part of this identity yields~\eqref{eq:prop-exp-KE} in combination with the skew-symmetry of $\tilde f_n(E)$ and the definition~\eqref{eq:def-C} of $\Cc(E)$ and $\cc(E)$.

\medskip
\step{2.2} Proof of~\eqref{eq:def-C+}.\\
We start from~\eqref{eq:def-C}, projected in some direction $E'\in\Md_0^\Sym$,
\begin{equation*}
2E':\Cc(E)\,=\,\expecM{\sum_n\tfrac{\mathds1_{I_n}}{|I_n|}\int_{\partial I_n}E'(x-x_n)\cdot \sigma(\phi_E, \Pi_E)\nu},
\end{equation*}
which we reformulate using the ergodic theorem as 
\begin{equation}\label{eq:rewr-Cc-lim}
2E':\Cc(E)\,=\,\lim_{R\uparrow\infty}|B_R|^{-1}\sum_n\int_{\partial I_n}\eta_RE'(x-x_n)\cdot \sigma(\phi_E, \Pi_E)\nu,
\end{equation}
with $\eta_R$ as above.
We turn to a suitable reformulation of the right-hand side.
Adding and subtracting the passive corrector~$\psi_{E'}$, we can write
\begin{multline}\label{eq:reform-test-sigmaphi-againstE'}
\sum_n\int_{\partial I_n}\eta_RE'(x-x_n)\cdot \sigma(\phi_E, \Pi_E)\nu\\
\,=\,\sum_n\int_{\partial I_n}\eta_R(\psi_{E'}+E'(x-x_n))\cdot \sigma(\phi_E, \Pi_E)\nu
-\sum_n\int_{\partial I_n}\eta_R\psi_{E'}\cdot \sigma(\phi_E, \Pi_E)\nu.
\end{multline}
For the first right-hand side term, we use that $\psi_{E'}+E'(x-x_n)$ is a rigid motion in~$I_n$, cf.~\eqref{e.cor-1}, we appeal to boundary conditions for $(\phi_E,\Pi_E)$ in~\eqref{eq:phi-cor}, and we recall the notation~\eqref{eq:extension-fn},
which leads us to
\begin{equation*}
\sum_n\int_{\partial I_n}\eta_R(\psi_{E'}+E'(x-x_n))\cdot \sigma(\phi_E, \Pi_E)\nu
\,=\,-\sum_n\int_{I_n}\eta_R(\psi_{E'}+E'(x-x_n))\cdot f_n(E).
\end{equation*}
In order to reformulate the second right-hand side term of \eqref{eq:reform-test-sigmaphi-againstE'}, we appeal to the weak formulation of equation~\eqref{eq:phi-cor} for $\phi_E$: testing this equation with $\eta_R\psi_{E'}$ yields, as in~\eqref{eq:test-as-etaR-phiE},
\begin{multline*}
\int_{\R^d}\nabla(\eta_R\psi_{E'}):\nabla\phi_E - \int_{\R^d} \nabla \eta_R \cdot \psi_{E'} \Pi_E \,=\,\sum_n\int_{(I_n+B)\setminus I_n}\eta_R\psi_{E'}\cdot f_n(E) \\-\sum_n\int_{\partial I_n}\eta_R\psi_{E'}\cdot\sigma(\phi_E,\Pi_E)\nu.
\end{multline*}
Hence, \eqref{eq:reform-test-sigmaphi-againstE'} turns into
\begin{multline*}
\sum_n\int_{\partial I_n}\eta_RE'(x-x_n)\cdot \sigma(\phi_E, \Pi_E)\nu
\,=\,-\sum_n\int_{I_n}\eta_R(\psi_{E'}+E'(x-x_n))\cdot f_n(E)\\
- \int_{\R^d} \nabla \eta_R \cdot \psi_{E'} \Pi_E +\int_{\R^d}\nabla(\eta_R\psi_{E'}):\nabla\phi_E
-\sum_n\int_{(I_n+B)\setminus I_n}\eta_R\psi_{E'}\cdot f_n(E),
\end{multline*}
and thus, after reorganizing the terms, using that the skew-symmetry of $\tilde f_n(E)$ in~\eqref{eq:extension-fn} yields $\int_{I_n}E'(x-x_n)\cdot f_n(E)=0$,
\begin{multline*}
\sum_n\int_{\partial I_n}\eta_RE'(x-x_n)\cdot \sigma(\phi_E, \Pi_E)\nu
\,=\,\int_{\R^d}\nabla(\eta_R\psi_{E'}):\nabla\phi_E
\\
- \int_{\R^d} \nabla \eta_R \cdot \psi_{E'} \Pi_E -\sum_n\int_{I_n+B}\eta_R\psi_{E'}\cdot f_n(E).
\end{multline*}
Alternatively, expanding the first right-hand side term,
\begin{multline*}
\sum_n\int_{\partial I_n}\eta_RE'(x-x_n)\cdot \sigma(\phi_E, \Pi_E)\nu
\,=\,\int_{\R^d}\nabla\psi_{E'}:\nabla(\eta_R\phi_E) - \int_{\R^d} \nabla \eta_R \cdot \psi_{E'} \Pi_E\\
+\int_{\R^d}(\psi_{E'}\otimes\nabla\eta_R):\nabla\phi_E
-\int_{\R^d}(\phi_E\otimes\nabla\eta_R):\nabla\psi_{E'}
-\sum_n\int_{I_n+B}\eta_R\psi_{E'}\cdot f_n(E).
\end{multline*}
Now testing equation~\eqref{e.cor-1} for $\psi_{E'}$ with $\eta_R\phi_E$, and using that $\eta_R\phi_E$ is rigid in $\Ic$,
we find that the first right-hand side term only yields a term involving $\Sigma_E$
\begin{multline*}
\sum_n\int_{\partial I_n}\eta_RE'(x-x_n)\cdot \sigma(\phi_E, \Pi_E)\nu
\,=\, - \int_{\R^d} \nabla \eta_R \cdot (\psi_{E'} \Pi_E + \phi_E \Sigma_{E'}) \\
+ \int_{\R^d}(\psi_{E'}\otimes\nabla\eta_R):\nabla\phi_E
-\int_{\R^d}(\phi_E\otimes\nabla\eta_R):\nabla\psi_{E'}
-\sum_n\int_{I_n+B}\eta_R\psi_{E'}\cdot f_n(E).
\end{multline*}
Using the sublinearity of $\psi_{E'}$ and $\phi_E$ and the stationarity of $\nabla\psi_{E'},\nabla\phi_E,\Pi_E\mathds1_{\R^d\setminus\Ic}$, cf.~Lemmas~\ref{lem:pass-cor} and~\ref{lem:cor-act}, we can now pass to the limit $R\uparrow\infty$ in this identity, and we obtain
\begin{equation*}
\lim_{R\uparrow\infty}|B_R|^{-1}\sum_n\int_{\partial I_n}\eta_RE'(x-x_n)\cdot \sigma(\phi_E, \Pi_E)\nu
\,=\,-\expecM{\sum_n\tfrac{\mathds1_{I_n}}{|I_n|}\int_{I_n+B}\psi_{E'}\cdot f_n(E)}.
\end{equation*}
Inserting this into~\eqref{eq:rewr-Cc-lim},
the conclusion~\eqref{eq:def-C+} follows.
Finally, the formula for $\cc(E)$ simply follows by taking the trace in~\eqref{eq:def-C}.

\medskip
\step3 Construction of $\theta$. \\
This construction is standard: as $L_E$ is stationary, it follows from stationary calculus, e.g.~\cite[Chapter 7]{JKO94}, that there is a unique solution $\theta_E = (\theta_{E;ijk})_{ijk}\in\Ld^2(\Omega;H^1_\loc(\R^d))$ of
\begin{align}\label{eq:thetaE-Coul-gauge}
- \triangle \theta_{E;ijk} \,=\, \partial_j L_{E;ik} - \partial_k L_{E;ij},
\end{align}
such that $\nabla \theta_E$ is stationary, has vanishing expectation, has finite second moment,
\begin{align*}
\| \nabla \theta_E \|_{\Ld^2(\Omega)} \,\lesssim\, \|L_E\|_{\Ld^2(\Omega)} \,\lesssim\, \|K_E\|_{\Ld^2(\Omega)} + \|\nabla \gamma_E\|_{\Ld^2(\Omega)},
\end{align*} 
and satisfies the anchoring condition $\fint_B \theta_E = 0$.
By uniqueness, the skew-symmetry of~$\theta_E$ follows from the skew-symmetry of the right-hand side of~\eqref{eq:thetaE-Coul-gauge} with respect to indices~$j,k$.
The fact that $\theta_E$ is a vector potential for $L_E$ follows from this defining equation as e.g.\@ in~\cite[Section 3.1]{GNO-reg}. Finally, the sublinearity of $\theta_E$ is a standard property for random fields with stationary gradient and vanishing expectation; see e.g.~\cite[Chapter 7]{JKO94}.
\end{proof}

In view of the internal contribution to the effective viscosity, cf.~\eqref{eq:def-EF-0}, we also define the following stationary field,
\begin{equation}\label{eq:def-corF}
F_E\,:=\,-\sum_n\tfrac{\mathds1_{I_n}}{|I_n|}\int_{I_n+B}f_n(E)\otimes(x-x_n),
\end{equation}
which is such that
\begin{equation}\label{eq:def-EF}
\Ff_E\,=\,\expec{F_E}.
\end{equation}

When considering two-scale expansions, since correctors $\phi_E,\Pi_E,\gamma_E,\theta_E$, fluxes~$K_E,L_E$, and~$F_E$ are nonlinear with respect to $E$, we shall need to consider derivatives of these objects with respect to $E$.
We introduce a general notation for linearized quantities.

\begin{defin}\label{def:lin-corr}
Given $E,E',E''\in \Md_0^\Sym$, we use the following notation for directional derivatives of swimming forces $f_n(E)$ in directions $E',E''$,
\begin{eqnarray*}
\partial_{E'} f_n(E)&:=&\lim_{h\to0} \tfrac1h\big(f_n(E+hE')-f_n(E)\big),\\
\partial_{E',E''} f_n(E)&:=&
\lim_{h\to0} \tfrac1h\big(\partial_{E'}f_n(E+hE'')-\partial_{E'}f_n(E)\big).
\end{eqnarray*}

\begin{enumerate}[---]
\item \emph{First linearized correctors:}
We define $(\df \phi_{E,E'},\df \Pi_{E,E'})$ as in Lemma~\ref{lem:cor-act} with the swimming forces $f_n(E)$ replaced by $\partial_{E'} f_n(E)$ (and similarly for $\df F_{E,E'}$).
We then define~$\df \gamma_{E,E'}$ as in Lemma~\ref{lem:cor-gamma}
and~$\df K_{E,E'},\df L_{E,E'},\df \theta_{E,E'}$ as in Lemma~\ref{lem:KE} with $f_n(E)$ replaced 
by~$\partial_{E'} f_n(E)$ and with $(\phi_{E}, \Pi_{E})$ replaced by~$(\df \phi_{E,E'}, \df \Pi_{E,E'})$.
\smallskip\item \emph{Second linearized correctors:}
We define $(\df^2\phi_{E,E',E''}, \df^2 \Pi_{E,E',E''})$ as in Lemma~\ref{lem:cor-act} with swimming forces $f_n(E)$ replaced by $\partial_{E',E''} f_n(E)$ (and similarly for $\df^2 F_{E,E',E''}$).
We then define~$\df^2 \gamma_{E,E',E''}$ as in Lemma~\ref{lem:cor-gamma}
and~$\df^2 K_{E,E',E''},\df^2 L_{E,E',E''}, \df^2 \theta_{E,E',E''}$ as in Lemma~\ref{lem:KE} with $f_n(E)$ replaced 
by~$\partial_{E',E''} f_n(E)$ and with~$(\phi_{E}, \Pi_{E})$ replaced by~$(\df^2 \phi_{E,E',E''}$, $\df^2 \Pi_{E,E',E''})$.
\qedhere
\end{enumerate}
\end{defin}

Finally, we state that the effective maps $\Cc,\cc,\Ff$ are smooth,
with derivatives given in terms of linearized correctors.
The proof, based on Hypothesis~\ref{hyp:swim}, is straightforward and left to the reader.

\begin{lem}\label{lem:diffC}
The effective maps $\Cc,\cc,\Ff$ are smooth and 
\begin{gather}
|\Cc(E)|,|\cc(E)|,|\Ff(E)|\lesssim\lambda\langle E\rangle,\qquad
|\partial^k\Cc(E)|,|\partial^k\cc(E)|,|\partial^k\Ff(E)|\lesssim\lambda.
\label{eq:Cc-cc-reg}
\end{gather}
Moreover, we have for all~$E,F \in  \Md_0^\Sym$,
\begin{equation*} 
2\partial_F \Cc(E) +\partial_F \cc(E)\Id\,=\,\expecM{\sum_n\tfrac{\mathds1_{I_n}}{|I_n|}\int_{\partial I_n} \sigma(\df\phi_{E,F}, \df\Pi_{E,F})\nu\otimes_s(x-x_n)}.
\end{equation*}
In particular, in view of~\eqref{eq:prop-exp-KE}, this yields $\partial_F \expec{K_E}=\expec{\df K_{E,F}}$.
\end{lem}

\section{Proof of the homogenization result}\label{sec:hom}
This section is devoted to the proof of Theorem~\ref{th:homog}. We quickly establish the well-posedness result of Proposition~\ref{lem:well-posed} before turning to the analysis of the limit $\e \downarrow 0$.
For any map $V$ and domain $D$, we henceforth use the short-hand notation
$V_{\nmid D}\,:=\,\fint_D V$.

\subsection{Well-posedness of hydrodynamic model}
This section is devoted to the proof of Proposition~\ref{lem:well-posed}.
We start by recalling the following standard computation (see e.g.~\cite{DG-21}), which we already partly encountered when proving existence of correctors.
\begin{lem}[e.g.~\cite{DG-21}]
If vector fields $u,h$ and a scalar field $P$ satisfy the following relations,
\[\left\{\begin{array}{ll}
-\triangle u+\nabla P=h,&\text{in $\R^d\setminus\Ic$},\\
\Div(u)=0,&\text{in $\R^d\setminus\Ic$},\\
\D(u)=0,&\text{in $\Ic$},
\end{array}\right.\]
then the following  holds in $\R^d$
\begin{equation}\label{e.weak-sense}
-\triangle u+ \nabla(P\mathds1_{\R^d\setminus\Ic})=h\mathds1_{\R^d\setminus\Ic} - \sum_n\delta_{\partial I_n}\sigma(u,P)\nu.
\end{equation}
The same is true if~$\R^d$ and $\Ic$ are replaced by $U$ and $\Ic_\e(U)$, respectively.
\end{lem}

We turn to the proof of Proposition~\ref{lem:well-posed} and proceed by a fixed-point argument. Let~$\e>0$ be fixed. Given $v\in W^{1,1}(U)^d$, define $T_\e(v)\in H^1_0(U)^d$ as the unique solution of the following linear problem, with associated pressure $P_\e(v)\in\Ld^2(U\setminus\Ic_\e(U))$,
\begin{equation}\label{eq:approx-ueps-n}
\left\{\begin{array}{ll}
-\triangle T_\e(v)+\nabla P_\e(v)\\
 \qquad = h+\frac{\kappa}{\e} \sum_{n \in \mathcal{N} _\e(U)} f_{n,\e} (\chi_\delta \ast\D(v)_{\nmid \e I_n}), &\text{in $U\setminus\Ic_\e (U)$},\\
\Div(T_\e(v))=0,&\text{in $U\setminus\Ic_\e (U)$},\\
\D(T_\e(v))=0,&\text{in $\Ic_\e (U)$},\\
\int_{\e\partial I_n }\sigma(T_\e(v),P_\e(v))\nu+\tfrac \kappa \e \bar f_{n,\e} (\chi_\delta \ast\D(v)_{\nmid \e I_n} ) = 0,&\forall n \in \Nc_\e (U),\\
\int_{\e\partial I_n }\Theta(x-\e x_n )\cdot\sigma(T_\e(v),P_\e(v))\nu&\\
\hspace{2.5cm}+\kappa\,\Theta: \tilde f_{n}(\chi_\delta\ast\D(v)_{\nmid \e I_n})=0,&\forall \Theta \in \Md^\Skew,\forall n\in\Nc_\e(U).
\end{array}\right.
\end{equation}
We split the proof into three steps. We start by proving the result under the stronger smallness condition $\kappa\ell^{-d/2}\ll1$, before relaxing it to~\eqref{eq:cond-small-kappa}.

\medskip
\step1
Suboptimal contraction estimate: for all $v,w\in H^1_0(U)^d$ we have
\begin{equation}\label{eq:subopt-contract}
\int_U |\nabla (T_\e(v)-T_\e(w))|^2\,\lesssim_\chi\,(\kappa \ell^{-\frac d2})^2\int_U|\nabla(v-w)|^2.
\end{equation}
This proves that $T_\e$ is a contraction on $H^1_0(U)^d$ provided that $\kappa\ell^{-\frac d2}\ll1$ is small enough. Under this condition, we deduce the well-posedness of the hydrodynamic model~\eqref{eq:Stokes} in~$H^1_0(U)^d$.

\medskip\noindent
We turn to the proof of~\eqref{eq:subopt-contract}. For abbreviation, we set $T_\e(v,w):=T_\e(v)-T_\e(w)$ and $P_\e(v,w)=P_\e(v)-P_\e(w)$.
Testing the equations for $T_\e(v)$ and $T_\e(w)$ with $T_\e(v,w)$ in form of~\eqref{e.weak-sense}, we find
\begin{multline*}
\int_U |\nabla T_\e(v,w)|^2
\,=\,
-\sum_{n\in\Nc_\e(U)}\int_{\e\partial I_n}T_\e(v,w)\cdot\sigma(T_\e(v,w),P_\e(v,w))\nu\\
+\tfrac\kappa\e\sum_{n\in\Nc_\e(U)}\int_{\e(I_n+B)\setminus \e I_n}T_\e(v,w)\cdot\Big(f_{n,\e}(\chi_\delta\ast\D(v)_{\nmid\e I_n})-f_{n,\e}(\chi_\delta\ast\D(w)_{\nmid\e I_n})\Big),
\end{multline*}
and thus,
using the rigidity of $T_\e(v,w)$ in $\e I_n$ and using boundary conditions, recalling the notation~\eqref{eq:extension-fn},
\begin{equation*}
\int_U |\nabla T_\e(v,w)|^2
\,=\,\tfrac\kappa\e\sum_{n\in\Nc_\e(U)}\int_{\e(I_n+B)}T_\e(v,w)\cdot\Big(f_{n,\e}(\chi_\delta\ast\D(v)_{\nmid\e I_n})-f_{n,\e}(\chi_\delta\ast\D(w)_{\nmid\e I_n})\Big).
\end{equation*}
By the neutrality condition~\eqref{eq:neutrality}, appealing to Poincaré's inequality, using the hardcore assumption, and recalling that $f_{n,\e}$ is Lipschitz, we get
\begin{equation*}
\int_U |\nabla T_\e(v,w)|^2\,\lesssim\,\kappa\Big(\int_{U}|\nabla T_\e(v,w)|^2\Big)^\frac12\bigg(\sum_{n\in\Nc_\e(U)}\int_{\e I_n}|\chi_\delta\ast\D(v-w)|^2\bigg)^\frac12.
\end{equation*}
By Jensen's inequality, with $\int_{\R^d}\chi_\delta=1$, the second factor is bounded by
\[\sum_{n\in\Nc_\e(U)}\int_{\e I_n}|\chi_\delta\ast\D(v-w)|^2
\,\lesssim\,
\int_U\bigg(\sum_{n\in\Nc_\e(U)}\int_{\e I_n}\chi_\delta(x-y)\,dx\bigg)|\!\D(v-w)(y)|^2\,dy.\]
The expression into brackets can be estimated as follows,
\begin{eqnarray*}
\sup_{y\in\R^d}\sum_{n\in\Nc_\e(U)}\int_{\e I_n}\chi_\delta(x-y)\,dx
&\le&\sup_{y\in\R^d}\sum_{n}\int_{\frac\e\delta I_n}\chi(x-y)\,dx\\
&\lesssim&(\tfrac\e\delta)^d\sup_{y\in\R^d}\sum_{n}\sup_{y+\frac\e\delta I_n}\chi\\
&\le&(\tfrac\e\delta)^d\sup_{y\in\R^d}\sum_{n}\sup_{\frac\e\delta B_{\ell}(y+x_n)}\chi\\
&\lesssim&\ell^{-d}\sup_{y\in\R^d}\sum_{n}\int_{\frac\e\delta B_{\ell}(y+x_n)}\Big(\sup_{B_{2\ell\frac\e\delta}(z)}\chi\Big)\,dz,
\end{eqnarray*}
and thus, by the hardcore assumption, provided $\e\ell\le\delta$,
\begin{equation*}
\sup_{y\in\R^d}\sum_{n\in\Nc_\e(U)}\int_{\e I_n}\chi_\delta(x-y)\,dx
\,\lesssim\,\ell^{-d}\int_{\R^d}\Big(\sup_{B_2(z)}\chi\Big)\,dz.
\end{equation*}
Inserting this into the above,
we get
\begin{equation}\label{eq:bnd-Jensen-ellscal}
\sum_{n\in\Nc_\e(U)}\int_{\e I_n}|\chi_\delta\ast\D(v-w)|^2
\,\lesssim_\chi\,\ell^{-d}
\int_U|\nabla(v-w)|^2,
\end{equation}
and thus,
\begin{equation*}
\int_U |\nabla T_\e(v,w)|^2
\,\lesssim_\chi\,\kappa \ell^{-\frac d2}\Big(\int_{U}|\nabla T_\e(v,w)|^2\Big)^\frac12\Big(\int_U|\nabla(v-w)|^2\Big)^\frac12,
\end{equation*}
that is,~\eqref{eq:subopt-contract}.

\medskip
\step2 Improved contraction estimate: given $1<p\le2$ and given $\ell\gg_p1$ large enough, we have for all $v,w\in W^{1,p}_0(U)^d$,
\[\|\nabla(T_\e(v)-T_\e(w))\|_{\Ld^p(U)}\,\lesssim_{p,\chi}\,\kappa\ell^{-\frac dp}\|\nabla(v-w)\|_{\Ld^p(U)}.\]
This proves that $T_\e$ is a contraction on $W^{1,p}_0(U)^d$ provided that $\kappa\ell^{-\frac dp}\ll1$ is small enough, which then implies the well-posedness of the hydrodynamic model~\eqref{eq:Stokes} in~$W^{1,p}_0(U)^d$.

\medskip\noindent
We appeal to the dilute deterministic $\Ld^p$ regularity theory developed by Höfer in~\cite{Hofer-19}. Given $1<p\le2$, provided that $\ell\gg_p1$ is large enough (depending on $p$ and dimension $d$), it allows us to deduce almost surely,
\begin{equation*}
\|\nabla(T_\e(v)- T_\e(w))\|_{\Ld^p(U)}
\,\lesssim_p\,\kappa\Big\|\!\sum_{n\in\Nc_\e(U)}\big(f_{n,\e}(\chi_\delta\ast\D(v)_{\nmid \e I_n})-f_{n,\e}(\chi_\delta\ast\D(w)_{\nmid \e I_n})\big)\Big\|_{\Ld^p(U)},
\end{equation*}
and thus, using properties of $\{f_{n,\e}\}_n$,
\begin{eqnarray*}
\lefteqn{\|\nabla(T_\e(v)- T_\e(w))\|_{\Ld^p(U)}}\\
&\lesssim_p&\kappa\bigg(\sum_{n\in\Nc_\e(U)}\int_{\e(I_n+B)}\big|f_{n,\e}(\chi_\delta\ast\D(v)_{\nmid \e I_n})-f_{n,\e}(\chi_\delta\ast\D(w)_{\nmid \e I_n})\big|^p\bigg)^\frac1p\\
&\lesssim&\kappa\bigg(\sum_{n\in\Nc_\e(U)}\int_{\e I_n}|\chi_\delta\ast\D(v-w)|^p\bigg)^\frac1p.
\end{eqnarray*}
Repeating the argument in~\eqref{eq:bnd-Jensen-ellscal}, this proves the claim.

\medskip
\step3 Conclusion.\\
Testing equation~\eqref{eq:approx-ueps-n} with its solution $T_\e(v)$ itself, using boundary conditions and recalling the notation~\eqref{eq:extension-fn}, we find
\begin{equation*}
\int_U|\nabla T_\e(v)|^2\,=\,\int_{U\setminus\Ic_\e(U)}h\cdot T_\e(v)
+\tfrac\kappa\e\sum_{n\in\Nc_\e(U)}\int_{\e(I_n+B)}T_\e(v)\cdot f_{n,\e}(\chi_\delta\ast\D(v)_{\nmid\e I_n}),
\end{equation*}
and thus, by the neutrality condition~\eqref{eq:neutrality}, appealing to Poincaré's inequality and recalling that $f_{n,\e}$ is Lipschitz, arguing exactly as in Step~1,
\begin{eqnarray*}
\int_U|\nabla T_\e(v)|^2
\,\lesssim\,
\kappa^2\ell^{-d}+\int_{U\setminus\Ic_\e(U)}|h|^2
+\kappa^2\sum_{n\in\Nc_\e(U)}\int_{\e I_n}|\chi_\delta\ast\D(v)|^2.
\end{eqnarray*}
Now using the hardcore condition and Young's inequality, we deduce for all~$1<p\le2$,
\begin{eqnarray*}
\int_U|\nabla T_\e(v)|^2
\,\lesssim\,
\kappa^2\ell^{-d}+\int_{U\setminus\Ic_\e(U)}|h|^2
+\kappa^2\|\chi_\delta\|^2_{\Ld^1\cap\Ld^2(U)}\|\nabla v\|_{\Ld^p(U)}^2.
\end{eqnarray*}
This proves that $T_\e$ embeds $W^{1,p}_0(U)^d$ into $H^1_0(U)^d$, so that the solution $u_\e=T_\e(u_\e)$ in $W^{1,p}_0(U)^d$ constructed in Step~2 actually belongs to $H^1_0(U)^d$,
and the a priori estimate~\eqref{lem:well-posed-bd} follows by iteration.
\qed

\subsection{Qualitative homogenization at fixed $\delta$}\label{sec-qual-homog}
\begingroup\allowdisplaybreaks
This section is devoted to the proof of Theorem~\ref{th:homog}.
Before turning to the proof, we argue for the well-posedness of the homogenized equation~\eqref{eq:homog}.
Using $|\Cc(E)|, |\Ff(E)| \lesssim\lambda\langle E\rangle$, cf.~\eqref{eq:Cc-cc-reg}, a perturbative argument as in the proof of Proposition~\ref{lem:well-posed} yields the well-posedness of~\eqref{eq:homog} with $(\bar u_\delta,\bar P_\delta)\in H^1_0(U)^d\times\Ld^2(U)/\R$ provided that $\kappa\lambda\ll1$ is small enough. As $\lambda\lesssim\ell^{-d}$, this holds in particular under the smallness condition~\eqref{eq:cond-small-kappa}.

Interestingly, our proof of qualitative homogenization relies on a semi-quantitative two-scale analysis
and we do not believe that there exists a simpler and purely qualitative proof.
In terms of a suitable limiting profile $\bar u_\e$ (mildly depending on $\e$ and to be identified at the end of the proof), we consider the following two-scale expansions for the solution~$(u_\e,P_\e)$ of the hydrodynamic model~\eqref{eq:Stokes},
\begin{eqnarray}
u_\e&\leadsto&\bar u_\e+\e\psi_E(\tfrac\cdot\e)\partial_E\bar u_\e+\e\kappa \phi_{\chi_\delta\ast\D(u_\e)} (\tfrac\cdot\e),\label{eq:ueps-2scale}\\
P_\e\mathds1_{\R^d\setminus\e\Ic}&\leadsto&\bar P_\e+\bb:\D(\bar u_\e)+\kappa\cc(\chi_\delta\ast\D( u_\e))\nonumber\\
&&\hspace{2cm}+(\Sigma_E\mathds1_{\R^d\setminus\Ic})(\tfrac\cdot\e)\partial_E\bar u_\e+(\kappa\Pi_{\chi_\delta\ast\D(u_\e)(x)}\mathds1_{\R^d\setminus\Ic})(\tfrac x\e),\nonumber
\end{eqnarray}
where $u_\e$ is implicitely extended by $0$ on $\R^d\setminus U$ to ensure that $\chi_\delta\ast \D( u_\e)$ is well-defined,
and where we implicitly sum over $E$ in an orthonormal basis of $\Md_0^\Sym$.
We start with two comments on the form of these two-scale expansions:
\begin{enumerate}[---]
\item The most unusual feature is that correctors $\phi,\Pi$ are evaluated at a background fluid deformation \mbox{$\chi_\delta\ast\D(u_\e)$} depending on the microscopic solution $u_\e$.
This non-standard choice is taken as an intermediate step
and happens to be providential in our proof, where the limit of this background deformation $\chi_\delta\ast\D(u_\e)\sim\chi_\delta\ast\D(\bar u_\e)$ can only be identified at the very end. To our knowledge, the necessity of such a two-step homogenization argument is new to the literature and gives the present problem an interest of its own.
\smallskip\item The choice of the non-oscillating part $\bar P_\e+\bb:\D(\bar u_\e)+\kappa\cc(\chi_\delta\ast\D(u_\e))$ in the two-scale expansion of the pressure is dictated by the proof and is similar to the case of passive suspensions~\cite{DG-21}: the pressure for~\eqref{eq:Stokes} does not converge to the pressure of the homogenized problem in its naïve form.
\end{enumerate}

\smallskip
As $\chi_\delta\ast\D(u_\e)$ is smooth, the maps $(x,y)\mapsto (\phi_{\chi_\delta\ast \D( u_\e)(x)})(y),(\Pi_{\chi_\delta\ast\D(u_\e)(x)}\mathds1_{\R^d\setminus\Ic})(y)$ are  Carath\'eodory functions, which ensures that $x\mapsto \phi_{\chi_\delta\ast \D( u_\e)(x)}(\frac x\e),(\Pi_{\chi_\delta\ast\D(u_\e)(x)}\mathds1_{\R^d\setminus\Ic})(\frac x\e)$ in~\eqref{eq:ueps-2scale} are measurable.
When differentiating such composed functions, some care is needed in the notation.
Given a smooth field $V:\R^d\to \Md_0^\Sym$,
we denote by $\phi_{V}$ the function of two variables $(x,y)\mapsto \phi_{V(x)}(y)$,
and we use the short-hand notation $\phi_V(\frac x\e):=\phi_{V(x)}(\frac x\e)$. We use the following notation for the derivative of $\phi_V$ with $V$ frozen,
\[(\nabla\phi_{|V})(\tfrac x\e)\,:=\, \nabla_y \phi_{V(x)}|_{y=\frac x\e},\]
so that the total derivative is then given by
\begin{equation}\label{e.deriv-not}
\nabla(\e \phi_V)(\tfrac x\e)\,=\,(\nabla \phi_{|V})(\tfrac x\e)+ \e {\df\phi_{V(x),E}(\tfrac x\e)}\otimes {\nabla  (E:V(x))} ,
\end{equation}
where the linearized corrector $\df\phi$ is given in Definition~\ref{def:lin-corr} and where we implicitly sum over $E$ in an orthonormal basis of $\Md_0^\Sym$.
In addition, if $V$ is a smooth random field, we use the following notation for the expectation of $\phi_V$ with $V$ frozen: for any random field~$a$,
\[\E[a(x)\phi_{|V}(\tfrac x\e)]\,:=\,\expec{a(x)\phi_E(\tfrac x\e)}\!|_{E=V(x)}.\]
The same notation is used for other correctors and fluxes.

As usual, correctors are not capturing the relevant behavior close to the boundary of the domain $U$: in particular, the two-scale expansion in~\eqref{eq:ueps-2scale} does not vanish at the boundary. To circumvent this, we proceed by truncating correctors in a neighborhood of the boundary.
We set for abbreviation
\[\partial_{r}U:=\{x\in U:\dist(x,\partial U)<r\},\]
and, given $r_\e\ge4\e$ (which we shall optimize later on), we choose a smooth cut-off function $\eta_\e\in C^\infty_c(U;[0,1])$ such that
\[\eta_\e|_{U\setminus\partial_{2r_\e}U}=1,\qquad \eta_\e|_{\partial_{r_\e}U}=0,\qquad|\nabla \eta_\e|\,\lesssim\,\tfrac1{r_\e},\]
and such that $\eta_\e$ is constant inside each of the fattened particles $\{\e(I_n+\frac12 B)\}_n$.
Note that by definition the set $\Ic_\e(U)$ coincides with $\e\Ic$ on the support of $\eta_\e$,
\begin{equation}\label{eq:eta-eps-coincid/Ic}
\eta_\e\mathds1_{\Ic_\e(U)}=\eta_\e\mathds1_{\e\Ic}.
\end{equation}
In these terms, truncating the two-scale expansions~\eqref{eq:ueps-2scale}, we are led to considering the following truncated two-scale expansion errors,
\[w_\e\,:=\,u_\e-u_\e^{2s},\qquad Q_\e\,:=\,P_\e \mathds1_{U\setminus\Ic_\e(U)}-P_\e^{2s},\]
where the two-scale expansion $(u_\e^{2s},P_\e^{2s})$ of $(u_\e, P_\e)$ is given by
\begin{eqnarray}
u_\e^{2s}(x)&:=& \bar u_\e(x)+\e\eta_\e(x)\psi_E(\tfrac x\e)\partial_E\bar u_\e(x)+\e\kappa\eta_\e(x)\phi_{\chi_\delta\ast \D( u_\e)}(\tfrac x\e),\label{eq:def-2sc-error}\\
P_\e^{2s}(x)&:=&-P_\e^*(x)+\bar P_\e(x)+\eta_\e(x)\bb:\D(\bar u_\e)(x)+\kappa\eta_\e(x)\cc(\chi_\delta\ast \D(u_\e)(x))\nonumber\\
&&\hspace{.5cm}+\eta_\e(x)(\Sigma_E\mathds1_{\R^d\setminus\Ic})(\tfrac x\e)(\partial_E\bar u_\e)(x)+\kappa\eta_\e(x)( \Pi_{\chi_\delta\ast\D(u_\e)}\mathds1_{\R^d\setminus\Ic})(\tfrac x\e),\nonumber
\end{eqnarray}
where we have further added a locally constant pressure field
\begin{equation}\label{eq:def-P*}
P_\e^*\,:=\,P'_{\e}\mathds1_{U\setminus\Ic_\e(U)}+\sum_{n\in\Nc_\e(U)}P''_{\e,n}\mathds1_{\e I_n},
\end{equation}
in terms of some constants $P_\e'$ and $\{P_{\e,n}''\}_n$ to be suitably chosen later on, cf.~\eqref{eq:choice-csts-P}, and
where the limiting profile $(\bar u_\e,\bar P_\e)\in H^1_0(U)^d\times\Ld^2(U)/\R$ is chosen as the unique solution of
\begin{equation}\label{eq:def-barueps}
-\Div(2\Bb_\pas\D(\bar u_\e))+\nabla\bar P_\e\,=\,(1-\lambda)h+\Div(2\kappa \Bb_\act(\chi_\delta\ast \D(u_\e))).
\end{equation}
Again, this equation is viewed as a convenient intermediate step towards the relevant homogenized equation~\eqref{eq:homog}: as in the two-scale expansion~\eqref{eq:def-2sc-error}, the background fluid deformation $\chi_\delta\ast\D(u_\e)$ is expressed here in terms of the microscopic solution $u_\e$ itself and its limit will only be identified at the very end of the proof.
We split the proof into three main steps.

\medskip
\step1 Equation for the two-scale expansion error $(w_\e,Q_\e)$: in the weak sense in $U$,
\begin{multline}\label{eq:eqn-wQ}
-\triangle w_\e +\nabla Q_\e\,=\,
-\Div\Big((J_E\mathds1_{\Ic})(\tfrac\cdot\e)\eta_\e\partial_E\bar u_\e+\kappa(K_{\chi_\delta\ast\D(u_\e)}\mathds1_{\Ic})(\tfrac\cdot\e)\,\eta_\e\Big)\\
-\sum_{n\in\Nc_\e(U)}\Big(\delta_{\e\partial I_n}\sigma(u_\e,P_\e+P_\e'-P_{\e,n}'')\nu+\tfrac\kappa\e f_{n,\e}\big(\chi_\delta\ast\D(u_\e)_{\nmid \e I_n}\big)\mathds1_{\e I_n}\Big)\\
-\tfrac\kappa\e\!\!\!\sum_{n\in\Nc_\e(U)}\Big(\eta_\e f_{n,\e}(\chi_\delta\ast\D(u_\e))-f_{n,\e} \big({\chi_\delta\ast\D(u_\e)}_{\nmid \e I_n}\big)\Big)-\kappa\df F_{\chi_\delta\ast\D(u_\e),E}(\tfrac\cdot\e)\eta_\e\nabla(\chi_\delta\ast\partial_E u_\e)\\
+(1-\eta_\e)(\lambda-\mathds1_{\Ic_\e(U)})h-\kappa\Ff(\chi_\delta\ast\D(u_\e))\nabla\eta_\e\\
-\Div\Big((1-\eta_\e)\big(2(\Bb_\pas-\Id)\D(\bar u_\e)+2\kappa\Bb_\act(\chi_\delta\ast\D(u_\e))\big)\Big)\\
\hspace{-2cm}+\e\Div\Big(
\big(2\psi_{E}\otimes_s-\Id\otimes\psi_E-\Ups_E\big)(\tfrac\cdot\e)\nabla(\eta_\e\partial_E\bar u_\e)
-\eta_\e h\otimes(\triangle^{-1}\nabla\mathds1_{\Ic})(\tfrac\cdot\e)\\
\hspace{2cm}+\kappa\big(2\phi_{\chi_\delta\ast\D(u_\e)}\otimes_s-\Id\otimes\phi_{\chi_\delta\ast\D(u_\e)}-\theta_{\chi_\delta\ast\D(u_\e)}+\gamma_{\chi_\delta\ast\D(u_\e)}\otimes\Id\big)(\tfrac\cdot\e)\nabla\eta_\e\\
+\kappa\big(2\df \phi_{\chi_\delta\ast\D(u_\e),E}\otimes_s-\Id\otimes\df \phi_{\chi_\delta\ast\D(u_\e),E}-\df \theta_{\chi_\delta\ast\D(u_\e),E}\\
\hspace{7cm}+\df\gamma_{\chi_\delta\ast\D(u_\e),E}\otimes\Id\big)(\tfrac\cdot\e)\,\eta_\e\nabla (\chi_\delta\ast \partial_E u_\e)
\Big)\\
+\kappa\e\nabla_i\Big((\triangle^{-1}\nabla_i\df F_{|\chi_\delta\ast\D(u_\e),E})(\tfrac\cdot\e)\eta_\e\nabla(\chi_\delta\ast\partial_Eu_\e)\Big)\\
+\e(\triangle^{-1}\nabla_j\mathds1_{\Ic})(\tfrac\cdot\e) \nabla_j(\eta_\e h)-\kappa\e\gamma_{\chi_\delta\ast\D(u_\e)}(\tfrac\cdot\e)\triangle\eta_\e
-\kappa\e\df\gamma_{\chi_\delta\ast\D(u_\e),E}(\tfrac\cdot\e):\eta_\e\triangle(\chi_\delta\ast \partial_E u_\e)\\
-\kappa\e(\triangle^{-1}\nabla_i\df F_{|\chi_\delta\ast\D(u_\e),E})(\tfrac\cdot\e)\nabla_i(\eta_\e\nabla(\chi_\delta\ast\partial_Eu_\e))\\
+2\kappa\e\big(\df\theta_{\chi_\delta\ast\D(u_\e),E}-\df\gamma_{\chi_\delta\ast\D(u_\e),E}\otimes\Id\big)(\tfrac\cdot\e):\nabla(\chi_\delta\ast\partial_Eu_\e)\otimes\nabla\eta_\e\\
+\kappa\e\big(\df^2 \theta_{\chi_\delta\ast\D(u_\e),E,E'}-\df^2\gamma_{\chi_\delta\ast\D(u_\e),E,E'}\otimes\Id-\triangle^{-1}\nabla\df^2 F_{|\chi_\delta\ast\D(u_\e),E,E'})\big)(\tfrac\cdot\e)\\
:\eta_\e\nabla (\chi_\delta\ast \partial_{E'} u_\e)\otimes\nabla (\chi_\delta\ast \partial_E u_\e).
\end{multline}
We split the proof into six  substeps.

\medskip
\substep{1.1} Reformulation of the equation for $u_\e$:
\begin{multline}\label{eq:ueps-press-ref}
-\triangle u_\e+\nabla(P_\e\mathds1_{U\setminus\Ic_\e(U)}+P_\e^*)
\,=\,h\mathds1_{U\setminus\Ic_\e(U)}
+\tfrac\kappa\e\!\!\sum_{n\in\Nc_\e(U)}f_{n,\e} \big({\chi_\delta\ast\D(u_\e)}_{\nmid \e I_n}\big)\\
-\sum_{n\in\Nc_\e(U)}\Big(\delta_{\e\partial I_n}\sigma(u_\e,P_\e+P_\e'-P_{\e,n}'')\nu
+\tfrac\kappa\e f_{n,\e}\big(\chi_\delta\ast\D(u_\e)_{\nmid \e I_n}\big)\mathds1_{\e I_n}\Big).
\end{multline}
Starting from equation~\eqref{eq:Stokes} in form of~\eqref{e.weak-sense}, we find
\begin{multline*}
-\triangle u_\e+\nabla(P_\e\mathds1_{U\setminus\Ic_\e(U)})\\
\,=\,h\mathds1_{U\setminus\Ic_\e(U)}
+\tfrac\kappa\e\!\!\sum_{n\in\Nc_\e(U)}f_{n,\e} \big({\chi_\delta\ast\D(u_\e)}_{\nmid \e I_n}\big)\mathds1_{\e(I_n+B)\setminus\e I_n}
-\sum_{n\in\Nc_\e(U)}\delta_{\e\partial I_n}\sigma(u_\e,P_\e)\nu.
\end{multline*}
Adding and subtracting the contribution of swimming forces on the particles, this becomes
\begin{multline*}
-\triangle u_\e+\nabla(P_\e\mathds1_{U\setminus\Ic_\e(U)})\,=\,h\mathds1_{U\setminus\Ic_\e(U)}+\tfrac\kappa\e\!\!\sum_{n\in\Nc_\e(U)}f_{n,\e} \big({\chi_\delta\ast\D(u_\e)}_{\nmid \e I_n}\big)\\
-\sum_{n\in\Nc_\e(U)}\Big(\delta_{\e\partial I_n}\sigma(u_\e,P_\e)\nu
+\tfrac\kappa\e f_{n,\e} \big({\chi_\delta\ast\D(u_\e)}_{\nmid \e I_n}\big)\mathds1_{\e I_n}\Big).
\end{multline*}
Adding $\nabla P^*_\e$ to both sides, cf.~\eqref{eq:def-P*}, the claim~\eqref{eq:ueps-press-ref} follows.

{\medskip
\substep{1.2} Equation for the two-scale expansion for the two-scale expansion error $u^{2s}_\e$:
\begin{multline}\label{e.homNL-1re}
-\triangle u^{2s}_\e +\nabla (P^{2s}_\e+P_\e^*)\,=\,
\nabla\Big(\bar P_\e+\eta_\e\bb:\D(\bar u_\e)+\kappa\eta_\e\cc(\chi_\delta\ast \D(u_\e))\Big)\\
-T_\e^1-\kappa T_\e^2-\Div\big(2(1-\eta_\e)\D(\bar u_\e)\big)\\
-\e\Div\Big(2 \psi_{E} (\tfrac\cdot\e)\otimes_s \nabla( \eta_\e\partial_E\bar u_\e)
+2 \kappa\phi_{\chi_\delta\ast\D(u_\e)} (\tfrac\cdot\e)\otimes_s \nabla \eta_\e\\
\hspace{4cm}+2 \kappa\df \phi_{\chi_\delta\ast\D(u_\e),E_\alpha}(\tfrac\cdot\e)\otimes_s\eta_\e \nabla(\chi_\delta\ast\partial_{E_\alpha}u_\e)\Big)
\\
+\e\nabla\Big( \psi_{E} (\tfrac\cdot\e)\cdot \nabla( \eta_\e\partial_E\bar u_\e)
+ \kappa\phi_{\chi_\delta\ast\D(u_\e)}(\tfrac\cdot\e) \cdot \nabla \eta_\e\\
+\kappa \df \phi_{\chi_\delta\ast\D(u_\e),E_\alpha}(\tfrac\cdot\e)\cdot \eta_\e\nabla(\chi_\delta\ast \partial_{E_\alpha}u_\e)\Big),
\end{multline}
where
\begin{eqnarray*}
T_\e^1&:=&\Div\Big(\big(2(\D (\psi_{E})+E)-\Sigma_E\mathds1_{\R^d\setminus\Ic}\Id\big)(\tfrac\cdot\e)\,\eta_\e\partial_E\bar u_\e\Big),\\
T_\e^2&:=&\Div\Big(\big(2\D (\phi_{|\chi_\delta\ast\D(u_\e)})-\Pi_{\chi_\delta\ast\D(u_\e)}\mathds1_{\R^d\setminus\Ic}\Id\big)(\tfrac\cdot\e)\,\eta_\e\Big)
\end{eqnarray*}
are two terms that we shall further reformulate in the upcoming substeps.

\medskip\noindent
For the two-scale expansion error $(u_\e^{2s},P_\e^{2s})$ defined in~\eqref{eq:def-2sc-error}, we have
\begin{multline}\label{e.homNL-1}
-\triangle u^{2s}_\e +\nabla (P^{2s}_\e+P_\e^*)\,=\,
-\triangle\bar u_\e+\nabla\Big(\bar P_\e+\eta_\e\bb:\D(\bar u_\e)+\kappa\eta_\e\cc(\chi_\delta\ast \D(u_\e))\Big)\\
-\triangle\Big(\e\psi_E(\tfrac\cdot\e)\eta_\e\partial_E\bar u_\e+\e\kappa\phi_{\chi_\delta\ast \D( u_\e)}(\tfrac \cdot\e)\eta_\e\Big)\\
+\nabla\Big((\Sigma_E\mathds1_{\R^d\setminus\Ic})(\tfrac\cdot\e)\eta_\e\partial_E\bar u_\e+\kappa(\Pi_{\chi_\delta\ast\D(u_\e)}\mathds1_{\R^d\setminus\Ic})(\tfrac\cdot\e)\eta_\e\Big).
\end{multline}
It remains to reformulate the penultimate right-hand side term.
By the identity
$$
\triangle h = \Div (2 \D(h))-\nabla \Div (h),
$$
we can write
\begin{eqnarray*}
\triangle\big(\e\psi_{E}(\tfrac\cdot\e)\eta_\e\partial_E\bar u_\e\big)&=& \Div \big(2\D (\e\psi_{E}(\tfrac\cdot\e)\eta_\e\partial_E\bar u_\e)\big)-\nabla \Div\big(\e\psi_{E}(\tfrac\cdot\e)\eta_\e\partial_E\bar u_\e\big),\\
\triangle \big(\e\phi_{\chi_\delta\ast\D(u_\e)}(\tfrac\cdot\e)\eta_\e\big)&=& \Div \big(2\D (\e\phi_{\chi_\delta\ast\D(u_\e)}(\tfrac\cdot\e)\eta_\e)\big)-\nabla\Div \big(\e\phi_{\chi_\delta\ast\D(u_\e)}(\tfrac\cdot\e)\eta_\e\big).
\end{eqnarray*}
First, as $\Div(\psi_E)=0$, we find
\begin{eqnarray}
\D (\e\psi_E(\tfrac\cdot\e)\eta_\e\partial_E\bar u_\e)&=&\D (\psi_{E})(\tfrac\cdot\e)\eta_\e\partial_E\bar u_\e+\e \psi_{E} (\tfrac\cdot\e)\otimes_s \nabla( \eta_\e\partial_E\bar u_\e),
\nonumber\\
\Div (\e\psi_E(\tfrac\cdot\e)\eta_\e\partial_E\bar u_\e)&=&\e \psi_{E} (\tfrac\cdot\e)\cdot \nabla( \eta_\e\partial_E\bar u_\e) .\label{e.div-2s2}
\end{eqnarray}
Second, as $\Div(\phi_{E})=0$, further recalling \eqref{e.deriv-not}, we find
\begin{eqnarray}
\D (\e\phi_{\chi_\delta\ast\D(u_\e)}(\tfrac\cdot\e)\eta_\e)&=&\D(\phi_{|\chi_\delta\ast\D(u_\e)})(\tfrac\cdot\e)\eta_\e+\e\df\phi_{\chi_\delta\ast\D(u_\e),E}(\tfrac\cdot\e) \otimes_s \eta_\e\nabla(\chi_\delta\ast\partial_{E}u_\e)\nonumber\\
&&\qquad+\e \phi_{\chi_\delta\ast\D(u_\e)} (\tfrac\cdot\e)\otimes_s \nabla \eta_\e,\nonumber\\
\Div (\e\phi_{\chi_\delta\ast\D(u_\e)}(\tfrac\cdot\e)\eta_\e)&=&\e \phi_{\chi_\delta\ast\D(u_\e)}(\tfrac\cdot\e) \cdot \nabla \eta_\e\nonumber\\
&&\qquad+\e\df\phi_{\chi_\delta\ast\D(u_\e),E}(\tfrac\cdot\e)\cdot  \eta_\e\nabla(\chi_\delta\ast\partial_Eu_\e),\label{e.div-2s1}
\end{eqnarray}
where we recall that we implicitly sum over $E$ in an orthonormal basis of $\Md^\Sym_0$, and where the linearized corrector $\df \phi$ is given in Definition~\ref{def:lin-corr}.
In these terms, the penultimate right-hand side term in~\eqref{e.homNL-1} takes the form
\begin{eqnarray*}
\lefteqn{\triangle\Big(\e\psi_E(\tfrac\cdot\e)\eta_\e\partial_E\bar u_\e+\e\kappa\phi_{\chi_\delta\ast\D(u_\e)}(\tfrac\cdot\e)\eta_\e\Big)}\\
&=&\Div\Big(2\D (\psi_{E})(\tfrac\cdot\e)\eta_\e\partial_E\bar u_\e+2\kappa\D (\phi_{|\chi_\delta\ast\D(u_\e)})(\tfrac\cdot\e)\eta_\e
+2\e \psi_{E} (\tfrac\cdot\e)\otimes_s \nabla( \eta_\e\partial_E\bar u_\e)\\
&&\hspace{1cm}+\,2\e \kappa\phi_{\chi_\delta\ast\D(u_\e)} (\tfrac\cdot\e)\otimes_s \nabla \eta_\e
+2\e \kappa\df \phi_{\chi_\delta\ast\D(u_\e),E_\alpha}(\tfrac\cdot\e)\otimes_s\eta_\e \nabla(\chi_\delta\ast\partial_{E_\alpha}u_\e)\Big)
\\
&-&\!\!\!\nabla\Big(\e \psi_{E} (\tfrac\cdot\e)\cdot \nabla( \eta_\e\partial_E\bar u_\e)
+\e \kappa\phi_{\chi_\delta\ast\D(u_\e)}(\tfrac\cdot\e) \cdot \nabla \eta_\e\\
&&\hspace{6cm}+\e\kappa \df \phi_{\chi_\delta\ast\D(u_\e),E_\alpha}(\tfrac\cdot\e)\cdot \eta_\e\nabla(\chi_\delta\ast \partial_{E_\alpha}u_\e)\Big).
\end{eqnarray*}
Inserting this identity into~\eqref{e.homNL-1}, using that $\Div (\bar u_\e) =0$, and decomposing
\[\triangle\bar u_\e\,=\,\Div(2\D(\bar u_\e))\,=\,\Div\big(2(1-\eta_\e)\D(\bar u_\e)\big)+\Div(2E\eta_\e \partial_{E}\bar u_\e),\]
the claim~\eqref{e.homNL-1re} follows.
}

\medskip
\substep{1.3} Proof of
\begin{multline}\label{eq:reform-Dpsi-Jups}
T_\e^1=\Div\Big(\big(2(\D (\psi_{E})+E)-\Sigma_E\mathds1_{\R^d\setminus\Ic}\Id\big)(\tfrac\cdot\e)\,\eta_\e\partial_E\bar u_\e\Big)\\
\,=\,\Div\big(2\eta_\e\Bb_\pas\D(\bar u_\e)\big)+\nabla\big(\eta_\e\bb:\D(\bar u_\e)\big)
-\e\Div\big(\Ups_E(\tfrac\cdot\e)\nabla(\eta_\e\partial_E\bar u_\e)\big)\\
-\Div\big((J_E\mathds1_{\Ic})(\tfrac\cdot\e)\eta_\e\partial_E\bar u_\e\big).
\end{multline}
In terms of the extended flux $J_E$, cf.~Lemma~\ref{lem:Jzet}, we have
\begin{equation*}
\Div\Big(\big(2(\D (\psi_{E})+E)-\Sigma_E\mathds1_{\R^d\setminus\Ic}\Id\big)(\tfrac\cdot\e)\,\eta_\e\partial_E\bar u_\e\Big)
\,=\,\Div\big((J_E\mathds1_{\R^d\setminus\Ic})(\tfrac\cdot\e)\,\eta_\e\partial_E\bar u_\e\big).
\end{equation*}
Recalling that $\Div(J_E)=0$, we can decompose
\begin{multline*}
\Div\Big(\big(2(\D (\psi_{E})+E)-\Sigma_E\mathds1_{\R^d\setminus\Ic}\Id\big)(\tfrac\cdot\e)\,\eta_\e\partial_E\bar u_\e\Big)\\
\,=\,J_E(\tfrac\cdot\e)\nabla(\eta_\e\partial_E\bar u_\e)-\Div\big((J_E\mathds1_{\Ic})(\tfrac\cdot\e)\,\eta_\e\partial_E\bar u_\e\big).
\end{multline*}
As Lemma~\ref{lem:Jzet} further yields $J_{E}-\expec{J_E}=\Div(\Ups_{E})$ and $\expec{J_E}=2\Bb_\pas E+(\bb:E)\Id$, where the flux corrector~$\Ups_E$ is skew-symmetric in its last two indices, the claim~\eqref{eq:reform-Dpsi-Jups}  follows.
For completeness, we recall the standard argument based on skew-symmetry of $\Ups$ that leads to this identity:
for any smooth scalar field $\zeta$, we have
\begin{eqnarray}
(J_E-\expec{J_E})\nabla\zeta &=&\Div(\Ups_{E})\nabla \zeta \nonumber \\
&=&\ee_i(\nabla_k\Ups_{E;ijk})\nabla_j \zeta \nonumber\\
&=&\ee_i\nabla_k\big(\!\underbrace{\Ups_{E;ijk}}_{=-\Ups_{E;ikj}}\nabla_j \zeta\big)-\underbrace{\ee_i\Ups_{E;ijk}\nabla_{jk}^2 \zeta}_{=0}\nonumber \\
&=&-\Div(\Ups_{E}\nabla \zeta).\label{antoine.8}
\end{eqnarray}

\medskip
\substep{1.4} Proof of
\begin{multline}\label{eq:reform-Dphi-Ktheta}
T_\e^2=\Div\Big(\big(2\D (\phi_{|\chi_\delta\ast\D(u_\e)})-\Pi_{\chi_\delta\ast\D(u_\e)}\mathds1_{\R^d\setminus\Ic}\Id\big)(\tfrac\cdot\e)\,\eta_\e\Big)\\
\,=\,\Div(2\eta_\e\Cc(\chi_\delta\ast\D(u_\e)))+\nabla(\eta_\e\cc(\chi_\delta\ast\D(u_\e)))\\
-\tfrac1\e\eta_\e\!\!\sum_{n\in\Nc_\e(U)}f_{n,\e}(\chi_\delta\ast\D(u_\e))
-\Div\big((K_{\chi_\delta\ast\D(u_\e)}\mathds1_{\Ic})(\tfrac\cdot\e)\,\eta_\e\big)\\
-\e\Div\Big(\big(\theta_{\chi_\delta\ast\D(u_\e)}-\gamma_{\chi_\delta\ast\D(u_\e)}\otimes\Id\big)(\tfrac\cdot\e)\nabla\eta_\e\Big)\\
-\e\Div\Big(\big(\df \theta_{\chi_\delta\ast\D(u_\e),E}-\df\gamma_{\chi_\delta\ast\D(u_\e),E}\otimes\Id\big)(\tfrac\cdot\e)\,\eta_\e\nabla (\chi_\delta\ast \partial_E u_\e)\Big)\\
-\e\gamma_{\chi_\delta\ast\D(u_\e)}(\tfrac\cdot\e)\triangle\eta_\e
-\e\df\gamma_{\chi_\delta\ast\D(u_\e),E}(\tfrac\cdot\e):\eta_\e\triangle(\chi_\delta\ast \partial_E u_\e)\\
+2\e\big(\df\theta_{\chi_\delta\ast\D(u_\e),E}-\df\gamma_{\chi_\delta\ast\D(u_\e),E}\otimes\Id\big)(\tfrac\cdot\e):\nabla(\chi_\delta\ast\partial_Eu_\e)\otimes\nabla\eta_\e\\
+\e\big(\df^2 \theta_{\chi_\delta\ast\D(u_\e),E,E'}-\df^2\gamma_{\chi_\delta\ast\D(u_\e),E,E'}\otimes\Id\big)(\tfrac\cdot\e)\\
:\eta_\e\nabla (\chi_\delta\ast \partial_{E'} u_\e)\otimes\nabla (\chi_\delta\ast \partial_E u_\e).
\end{multline}
In terms of the extended flux $K_E$, cf.~Lemma~\ref{lem:KE}, we have
\begin{equation*}
\Div\Big(\big(2\D (\phi_{|\chi_\delta\ast\D(u_\e)})-\Pi_{\chi_\delta\ast\D(u_\e)}\mathds1_{\R^d\setminus\Ic}\Id\big)(\tfrac\cdot\e)\,\eta_\e\Big)
\,=\,\Div\big((K_{\chi_\delta\ast\D(u_\e)}\mathds1_{\R^d\setminus\Ic})(\tfrac\cdot\e)\,\eta_\e\big).
\end{equation*}
As Lemma~\ref{lem:KE} further yields $\expec{K_E}=2\Cc(E)+\cc(E)\Id$ and $\Div(K_E)=-\sum_n f_n(E)$,
and appealing to~\eqref{e.deriv-not} and~\eqref{eq:eta-eps-coincid/Ic}, we find
\begin{multline}\label{eq:decomp-Dphi-Pi-div}
\Div\Big(\big(2\D (\phi_{|\chi_\delta\ast\D(u_\e)})-\Pi_{\chi_\delta\ast\D(u_\e)}\mathds1_{\R^d\setminus\Ic}\Id\big)(\tfrac\cdot\e)\,\eta_\e\Big)\\
\,=\,\Div(2\eta_\e\Cc(\chi_\delta\ast\D(u_\e)))+\nabla(\eta_\e\cc(\chi_\delta\ast\D(u_\e)))
-\tfrac1\e\eta_\e\!\!\sum_{n\in\Nc_\e(U)}f_{n,\e}(\chi_\delta\ast\D(u_\e))\\
-\Div\big((K_{\chi_\delta\ast\D(u_\e)}\mathds1_{\Ic})(\tfrac\cdot\e)\,\eta_\e\big)
+\big(K_{\chi_\delta\ast\D(u_\e)}(\tfrac\cdot\e)-\expec{K_{|\chi_\delta\ast\D(u_\e)}}\!\big)\nabla\eta_\e\\
+\big(\df K_{\chi_\delta\ast\D(u_\e),E}(\tfrac\cdot\e)-\expec{\df K_{|\chi_\delta\ast\D(u_\e),E}(\tfrac\cdot\e)}\!\big)\,\eta_\e\nabla(\chi_\delta\ast\partial_Eu_\e).
\end{multline}
It remains to reformulate the last two right-hand side terms.
As Lemma~\ref{lem:KE} yields
\[K_{E}-\expec{K_{E}}\,=\,\Div(\theta_{E})+\nabla\gamma_{E},\]
we get
\begin{equation*}
\big( K_{\chi_\delta\ast\D(u_\e)}(\tfrac\cdot\e)-\expec{K_{|\chi_\delta\ast\D(u_\e)}}\!\big)\nabla\eta_\e
\,=\,\big(\Div(\theta_{|\chi_\delta\ast\D(u_\e)})
+\nabla\gamma_{|\chi_\delta\ast\D(u_\e)}\big)(\tfrac\cdot\e)\nabla\eta_\e,
\end{equation*}
and thus, by Leibniz' rule, using~\eqref{e.deriv-not} and the skew-symmetry of $\theta$ (whence the minus sign of the first right-hand side term, cf.~\eqref{antoine.8}),
\begin{multline*}
\big( K_{\chi_\delta\ast\D(u_\e)}(\tfrac\cdot\e)-\expec{K_{|\chi_\delta\ast\D(u_\e)}}\!\big)\nabla\eta_\e\\
\,=\,-\e\Div\Big(\big(\theta_{\chi_\delta\ast\D(u_\e)}-\gamma_{\chi_\delta\ast\D(u_\e)}\otimes\Id\big)(\tfrac\cdot\e)\nabla\eta_\e\Big)
-\e\gamma_{\chi_\delta\ast\D(u_\e)}(\tfrac\cdot\e)\triangle\eta_\e\\
+\e\big(\df\theta_{\chi_\delta\ast\D(u_\e),E}-\df\gamma_{\chi_\delta\ast\D(u_\e),E}\otimes\Id\big)(\tfrac\cdot\e):\nabla(\chi_\delta\ast\partial_Eu_\e)\otimes\nabla\eta_\e.
\end{multline*}
Similarly, with the notation of Definition~\ref{def:lin-corr}, we have
\begin{multline*}
\big( \df K_{\chi_\delta\ast\D(u_\e),E}(\tfrac\cdot\e)-\expec{\df K_{|\chi_\delta\ast\D(u_\e),E}}\!\big)\,\eta_\e\nabla(\chi_\delta\ast \partial_E u_\e)\\
\,=\,\big(\Div(\df \theta_{|\chi_\delta\ast\D(u_\e),E})+\nabla\df\gamma_{|\chi_\delta\ast\D(u_\e),E}\big)(\tfrac\cdot\e)\,\eta_\e\nabla (\chi_\delta\ast \partial_E u_\e),
\end{multline*}
and thus, by Leibniz' rule, using~\eqref{e.deriv-not} and the skew-symmetry of $\df\theta$,
\begin{multline*}
\big( \df K_{\chi_\delta\ast\D(u_\e),E}(\tfrac\cdot\e)-\expec{\df K_{|\chi_\delta\ast\D(u_\e),E}}\!\big)\,\eta_\e\nabla(\chi_\delta\ast \partial_E u_\e)\\
\,=\,-\e\Div\Big(\big(\df \theta_{\chi_\delta\ast\D(u_\e),E}-\df\gamma_{\chi_\delta\ast\D(u_\e),E}\otimes\Id\big)(\tfrac\cdot\e)\,\eta_\e\nabla (\chi_\delta\ast \partial_E u_\e)\Big)\\
+\e\big(\df \theta_{\chi_\delta\ast\D(u_\e),E}-\df\gamma_{\chi_\delta\ast\D(u_\e),E}\otimes\Id\big)(\tfrac\cdot\e):\nabla(\eta_\e\nabla (\chi_\delta\ast \partial_E u_\e))\\
+\e\big(\df^2 \theta_{\chi_\delta\ast\D(u_\e),E,E'}-\df^2\gamma_{\chi_\delta\ast\D(u_\e),E,E'}\otimes\Id\big)(\tfrac\cdot\e):\eta_\e\nabla (\chi_\delta\ast \partial_{E'} u_\e)\otimes\nabla (\chi_\delta\ast \partial_E u_\e).
\end{multline*}
Inserting this into~\eqref{eq:decomp-Dphi-Pi-div}, the claim~\eqref{eq:reform-Dphi-Ktheta} follows.

\medskip
\substep{1.5} In order to reconstruct $\Bb_\act(\chi_\delta\ast \D(u_\e))$, we shall need
the following identity:
\begin{multline}\label{eq:decomp-barF}
\Div(\eta_\e\Ff(\chi_\delta\ast\D(u_\e)))
\,=\,\df F_{\chi_\delta\ast\D(u_\e),E}(\tfrac\cdot\e)\eta_\e\nabla(\chi_\delta\ast\partial_E u_\e)+\Ff(\chi_\delta\ast\D(u_\e))\nabla\eta_\e\\
-\e\nabla_i\Big((\triangle^{-1}\nabla_i\df F_{|\chi_\delta\ast\D(u_\e),E})(\tfrac\cdot\e)\eta_\e\nabla(\chi_\delta\ast\partial_Eu_\e)\Big)\\
+\e(\triangle^{-1}\nabla_i\df F_{|\chi_\delta\ast\D(u_\e),E})(\tfrac\cdot\e)\nabla_i(\eta_\e\nabla(\chi_\delta\ast\partial_Eu_\e))\\
+\e(\triangle^{-1}\nabla_i\df^2 F_{|\chi_\delta\ast\D(u_\e),E,E'})(\tfrac\cdot\e)\eta_\e\nabla(\chi_\delta\ast\partial_Eu_\e)\nabla_i(\chi_\delta\ast\partial_{E'}(u_\e)).
\end{multline}
As $\expec{F_E}=\Ff(E)$, cf.~\eqref{eq:def-EF}, and using the notation in Definition~\ref{def:lin-corr}, we can decompose
\begin{eqnarray*}
\df F_{E,E'}&=&\partial_{E'}\Ff(E)+\big(\df F_{E,E'}-\expec{\df F_{E,E'}}\!\big)\\
&=&\partial_{E'}\Ff(E)+\nabla_i\triangle^{-1}\nabla_i\df F_{E,E'}.
\end{eqnarray*}
Using this together with Leibniz' rule, we find
\begin{multline*}
\df F_{\chi_\delta\ast\D(u_\e),E}(\tfrac\cdot\e)\eta_\e\nabla(\chi_\delta\ast\partial_E u_\e)
\,=\,\partial_E\Ff(\chi_\delta\ast\D(u_\e))\eta_\e\nabla(\chi_\delta\ast\partial_E u_\e)\\
+\e\nabla_i\Big((\triangle^{-1}\nabla_i\df F_{|\chi_\delta\ast\D(u_\e),E})(\tfrac\cdot\e)\Big)\eta_\e\nabla(\chi_\delta\ast\partial_Eu_\e)\\
-\e(\triangle^{-1}\nabla_i\df^2 F_{|\chi_\delta\ast\D(u_\e),E,E'})(\tfrac\cdot\e)\eta_\e\nabla(\chi_\delta\ast\partial_Eu_\e)\nabla_i(\chi_\delta\ast\partial_{E'}(u_\e)).
\end{multline*}
Further reformulating the first right-hand side term, noting that
\[\partial_E\Ff(\chi_\delta\ast\D(u_\e))\nabla(\chi_\delta\ast\partial_E u_\e)\,=\,\Div(\Ff(\chi_\delta\ast\D(u_\e))),\]
the claim~\eqref{eq:decomp-barF} follows.

\medskip
\substep{1.6} Proof of~\eqref{eq:eqn-wQ}.\\
Subtracting~\eqref{eq:ueps-press-ref} and~\eqref{e.homNL-1re}, inserting identities~\eqref{eq:reform-Dpsi-Jups}, \eqref{eq:reform-Dphi-Ktheta}, and~\eqref{eq:decomp-barF}, using equation~\eqref{eq:def-barueps} for~$\bar u_\e$,
and decomposing
\begin{eqnarray*}
h\mathds1_{U\setminus\Ic_\e(U)}-(1-\lambda)h&=&(\lambda-\mathds1_{\Ic_\e(U)})(1-\eta_\e)h+(\lambda-\mathds1_{\e\Ic})\eta_\e h\\
&=&(\lambda-\mathds1_{\Ic_\e(U)})(1-\eta_\e)h\\
&&\hspace{0.5cm}-\e\nabla_j\Big((\triangle^{-1}\nabla_j\mathds1_{\Ic})(\tfrac\cdot\e)\eta_\e h\Big)+\e(\triangle^{-1}\nabla_j\mathds1_{\Ic})(\tfrac\cdot\e) \nabla_j(\eta_\e h),
\end{eqnarray*}
the claim~\eqref{eq:eqn-wQ} follows after straightforward simplifications.

\medskip

\noindent In the rest of the proof, for notational convenience, we do not make explicit the dependence of estimates wrt $\kappa$ and $\ell$.

\medskip
\step2 Energy estimate for~\eqref{eq:eqn-wQ}: for all $K\ge1$,  
\begin{multline}\label{eq:est-main-simpl}
\int_U|\nabla w_\e|^2
\,\lesssim\,
\tfrac1{K}\int_U(Q_\e)^2
+\Big(1+\|(h,\nabla\bar u_\e)\|_{W^{1,\infty}(U)}^2+\|\chi_\delta\ast\D(u_\e)\|_{W^{2,\infty}(U)}^6\Big)\\
\times\bigg(
\int_{\partial_{3r_\e}U}|(1,\nabla\psi,\Sigma\mathds1_{\R^d\setminus\Ic},\nabla\phi,\Pi\mathds1_{\R^d\setminus\Ic})(\tfrac\cdot\e)|^2\\
+\e^2K\int_U\big(1+r_\e^{-2}\mathds1_{\partial_{3r_\e}U}\big)|(1,\psi,\nabla\psi,\Ups,\phi,\theta,\gamma,\df \phi,\nabla\df \phi,\df \theta,\df\gamma,\df^2\theta,\df^2\gamma,\\
\triangle^{-1}\nabla\mathds1_{\Ic},\triangle^{-1}\nabla\df F,\triangle^{-1}\nabla\df^2F)(\tfrac\cdot\e)|^2\bigg),
\end{multline}
where we use the following notation for correctors,
\begin{equation}\label{eq:notation-|psi|}
|\psi|:=\sup_{E}\,|E|^{-1}|\psi_E|,
\end{equation}
and similarly for $|\Ups|$. For the active corrector $\phi$, which depends nonlinearly on the direction~$E$, as well as for $\theta,\gamma$, we rather set
\begin{equation}\label{eq:notation-|phi|}
|\phi|\,:=\,\sup_{|E|\le C_\delta(h)}\langle E\rangle^{-1}|\phi_E|,
\end{equation}
where the deterministic constant $C_\delta(h)>0$ is chosen such that almost surely
\[\|\chi_\delta\ast\D(u_\e)\|_{\Ld^\infty(U)}\,\le\,\|\chi_\delta\|_{\Ld^2(\R^d)}\|\nabla u_\e\|_{\Ld^2(U)}\,\le\,C_\delta(h).\]
Such a constant can be chosen in view of~\eqref{lem:well-posed-bd}.
For the linearized correctors $\df\phi$ and $\df^2\phi$, as well as for $\df\theta,\df\gamma,\df^2\theta,\df^2\gamma,\triangle^{-1}\nabla\df F,\triangle^{-1}\nabla\df^2F$, we similarly set
\begin{eqnarray*}
|\df\phi|&:=&\sup_{|E|\le C_\delta(h)}\sup_{E'}\,|E'|^{-1}|\df\phi_{E,E'}|,\\
|\df^2\phi|&:=&\sup_{|E|\le C_\delta(h)}\sup_{E',E''}\,|E'|^{-1}|E''|^{-1}|\df^2\phi_{E,E',E''}|.
\end{eqnarray*}
By Poincar\'e's inequality in form of 
\begin{eqnarray*}
\sup_{\e(I_n+B)} \Big|\nabla \bar u_\e -\fint_{\e(I_n+B)} \nabla \bar u_\e\Big|&\lesssim&\e \|\nabla^2 \bar u_\e\|_{\Ld^\infty(U)},\\
\sup_{\e(I_n+B)} \Big|\chi_\delta \ast \D(u_\e)-\fint_{\e(I_n+B)} \chi_\delta*\D(u_\e) \Big|&\lesssim&\e \|\chi_\delta \ast\D(u_\e)\|_{W^{1,\infty}(U)},
\end{eqnarray*}
the claim~\eqref{eq:est-main-simpl} follows from 
\begin{multline}\label{eq:est-main}
\int_U|\nabla w_\e|^2\,\lesssim\,\tfrac1{K}\int_U(Q_\e)^2\\
+\int_{\partial_{3r_\e}U}|(1,\nabla\psi,\Sigma\mathds1_{\R^d\setminus \Ic},\nabla\phi,\Pi\mathds1_{\R^d\setminus\Ic})(\tfrac\cdot\e)|^2\big(1+|h|^2+|\nabla\bar u_\e|^2+[\nabla\bar u_\e]_{4\e}^2+|\chi_\delta\ast\D(u_\e)|^2\big)\\
+\e^2r_\e^{-2}K\int_{\partial_{3r_\e}U}|(\psi,\Ups,\phi,\theta,\gamma,\df \phi)(\tfrac\cdot\e)|^2\big(1+|\nabla\bar u_\e|^2+|\chi_\delta\ast \D(u_\e)|^2\big)\\
+\e^2K\int_U|(1,\psi,\Ups,\phi,\df \phi,\df\theta,\df\gamma,\df^2\theta,\df^2\gamma,\triangle^{-1}\nabla\mathds1_{\Ic},\triangle^{-1}\nabla\df F,\triangle^{-1}\nabla\df^2 F)(\tfrac\cdot\e)|^2\\
\hspace{1cm}\times\Big(|\langle\nabla\rangle h|^2+|\nabla^2\bar u_\e|^2+|\nabla\langle\nabla\rangle \chi_\delta\ast \D(u_\e)|^2+|\nabla\chi_\delta\ast \D(u_\e)|^4\\
\hspace{7cm}+|\nabla(\chi_\delta\ast\partial_E u_\e)|^2\big(1+[\chi_\delta\ast\D(u_\e)]_{4\e}^4\big)\Big)\\
+K\sum_n\int_{\e(I_n+B)}|(\psi,\nabla\psi)(\tfrac\cdot\e)|^2\Big|\nabla\bar u_\e-\fint_{\e(I_n+B)}\nabla\bar u_\e\Big|^2
\\
+K\sum_n\int_{\e(I_n+B)}|(\df\phi,\nabla\df\phi)(\tfrac\cdot\e)|^2\Big|\chi_\delta\ast\D( u_\e)-\fint_{\e(I_n+B)}\chi_\delta\ast\D( u_\e)\Big|^2,
\end{multline}
where we use the short-hand notation $[g]_{4 \e}(x) := (\fint_{B_{4\e}(x)} |g|^2)^{\frac12}$.
We split the proof of~\eqref{eq:est-main}  into seven  substeps.

\medskip
\substep{2.1} Preliminary.\\
In order to obtain~\eqref{eq:est-main}, we may wish to test equation~\eqref{eq:eqn-wQ} with $w_\e$ itself. However, $w_\e$ is not rigid inside particles, which prevents us from taking advantage of the boundary conditions. To circumvent this issue, we make use of the following truncation maps $T_0^\e,T_1^\e$: for all $g\in C^\infty_b(\R^d)^d$,
\begin{eqnarray*}
T_0^\e[g]&:=&(1-\rho^\e)g+\sum_n\rho_n^\e\Big(\fint_{\e(I_n+  B)}g\Big),\\
T_1^\e[g]&:=&(1-\rho^\e)g+\sum_n\rho_n^\e\bigg(\Big(\fint_{\e(I_n+  B)}g\Big)+(\cdot-\e x_n)_j\Big(\fint_{\e(I_n+ B)}\partial_j g\Big)\bigg),
\end{eqnarray*}
where for all $n$ we have chosen a cut-off function $\rho_n^\e\in C^\infty_c(\R^d;[0,1])$ with
\[\rho_n^\e|_{\e (I_n+\frac14B)}=1,\qquad\rho_n^\e|_{\R^d\setminus\e(I_n+\frac12B)}=0,\qquad|\nabla\rho_n^\e|\lesssim\tfrac1\e,\]
and where we have set for abbreviation $\rho^\e:=\sum_n\rho_n^\e$.
In these terms, we shall test~\eqref{eq:eqn-wQ} with  the following modification of the two-scale expansion error $w_\e$, cf.~\eqref{eq:def-2sc-error},
\begin{equation*}
\tilde w_\e\,:=\,u_\e-T_1^\e[\bar u_\e]-\e\eta_\e\psi_E(\tfrac\cdot\e)T_0^\e[\partial_E\bar u_\e]-\e\kappa \eta_\e\phi_{T_0^\e[\chi_\delta\ast \D( u_\e)]}(\tfrac \cdot\e).
\end{equation*}
Testing equation~\eqref{eq:eqn-wQ} with $\tilde w_\e$,
we obtain
\begin{equation}\label{eq:decomp-I0-3}
I_0^\e\,=\,I_1^\e+I_2^\e+I_3^\e+I_4^\e+I_5^\e,
\end{equation}
in terms of
\begin{eqnarray*}
I_0^\e&:=&\int_U\nabla\tilde w_\e:\nabla w_\e-\int_U Q_\e\,\Div(\tilde w_\e),
\\
I_1^\e&:=&\int_U\eta_\e\nabla \tilde w_\e:\Big((J_E\mathds1_{\Ic})(\tfrac\cdot\e)\partial_E\bar u_\e+\kappa (K_{\chi_\delta\ast \D(u_\e)}\mathds1_{\Ic})(\tfrac\cdot\e)\Big)\\
I_2^\e&:=&-\sum_{n\in\Nc_\e(U)}\bigg(\int_{\e\partial I_n}\tilde w_\e\cdot\sigma(u_\e,P_\e+P_\e'-P_{\e,n}'')\nu+\tfrac\kappa\e \int_{\e I_n} \tilde w_\e\cdot f_{n,\e}\big(\chi_\delta\ast\D(u_\e)_{\nmid \e I_n}\big)\bigg)\\
I_3^\e&:=&-\tfrac\kappa\e\sum_{n\in\Nc_\e(U)}\int_{\e(I_n+B)}\eta_\e\tilde w_\e\cdot\Big(f_{n,\e}(\chi_\delta\ast\D(u_\e))-f_{n,\e} \big({\chi_\delta\ast\D(u_\e)}_{\nmid \e I_n}\big)\Big)\\
&&-\kappa\int_{U}\eta_\e\tilde w_\e\otimes\nabla(\chi_\delta\ast\partial_E u_\e):\df F_{\chi_\delta\ast\D(u_\e),E}(\tfrac\cdot\e),\\
I_4^\e&:=&
\int_U(1-\eta_\e)(\lambda-\mathds1_{\Ic_\e(U)})\tilde w_\e\cdot h
+\tfrac\kappa\e\!\!\!\sum_{n\in\Nc_\e(U)}\int_{\e(I_n+B)}(1-\eta_\e)\tilde w_\e\cdot f_{n,\e} \big({\chi_\delta\ast\D(u_\e)}_{\nmid \e I_n}\big)\\
&&+2\int_U(1-\eta_\e)\D(\tilde w_\e):\big((\Bb_\pas - \Id)\D(\bar u_\e)+\kappa\Bb_\act(\chi_\delta\ast \D(u_\e))\big)\\
&&-\kappa\int_U\tilde w_\e\otimes\nabla\eta_\e:\Ff(\chi_\delta\ast\D(u_\e))
\end{eqnarray*}
and
\begin{eqnarray*}
I_5^\e
&:=&-\e\int_U\nabla\tilde w_\e:\Big(
\big(2\psi_{E}\otimes_s-\Id\otimes\psi_{E}-\Ups_E\big)(\tfrac\cdot\e)\nabla(\eta_\e\partial_E\bar u_\e)
-\eta_\e h\otimes(\triangle^{-1}\nabla\mathds1_{\Ic})(\tfrac\cdot\e)\\
&&\hspace{1.4cm}+\kappa\big(2\phi_{\chi_\delta\ast\D(u_\e)}\otimes_s-\Id\otimes\phi_{\chi_\delta\ast\D(u_\e)}-\theta_{\chi_\delta\ast\D(u_\e)}+\gamma_{\chi_\delta\ast\D(u_\e)}\otimes\Id\big)(\tfrac\cdot\e)\nabla\eta_\e\\
&&\hspace{1.4cm}+\kappa\big(2\df \phi_{\chi_\delta\ast\D(u_\e),E}\otimes_s-\Id\otimes\df \phi_{\chi_\delta\ast\D(u_\e),E}-\df \theta_{\chi_\delta\ast\D(u_\e),E}\\
&&\hspace{6.5cm}+\df\gamma_{\chi_\delta\ast\D(u_\e),E}\otimes\Id\big)(\tfrac\cdot\e)\,\eta_\e\nabla (\chi_\delta\ast \partial_E u_\e)
\Big)\\
&&-\kappa\e\int_U\eta_\e\nabla_i\tilde w_\e\otimes\nabla(\chi_\delta\ast\partial_Eu_\e) :(\triangle^{-1}\nabla_i\df F_{|\chi_\delta\ast\D(u_\e)})(\tfrac\cdot\e)\\
&&+\e\int_U\tilde w_\e\cdot\Big((\triangle^{-1}\nabla_j\mathds1_{\Ic})(\tfrac\cdot\e) \nabla_j(\eta_\e h)-\kappa\gamma_{\chi_\delta\ast\D(u_\e)}(\tfrac\cdot\e)\triangle\eta_\e\\
&&\hspace{1cm}-\kappa\df\gamma_{\chi_\delta\ast\D(u_\e),E}(\tfrac\cdot\e):\eta_\e\triangle(\chi_\delta\ast \partial_E u_\e)\\
&&\hspace{1cm}-\kappa(\triangle^{-1}\nabla_i\df F_{|\chi_\delta\ast\D(u_\e),E})(\tfrac\cdot\e)\nabla_i(\eta_\e\nabla(\chi_\delta\ast\partial_Eu_\e))\\
&&\hspace{1cm}+2\kappa\big(\df\theta_{\chi_\delta\ast\D(u_\e),E}-\df\gamma_{\chi_\delta\ast\D(u_\e),E}\otimes\Id\big)(\tfrac\cdot\e):\nabla(\chi_\delta\ast\partial_Eu_\e)\otimes\nabla\eta_\e\\
&&\hspace{1cm}+\kappa\big(\df^2 \theta_{\chi_\delta\ast\D(u_\e),E,E'}-\df^2\gamma_{\chi_\delta\ast\D(u_\e),E,E'}\otimes\Id-\triangle^{-1}\nabla_i\df^2 F_{|\chi_\delta\ast\D(u_\e),E,E'}\big)(\tfrac\cdot\e)\\
&&\hspace{7cm}:\eta_\e\nabla (\chi_\delta\ast \partial_{E'} u_\e)\otimes\nabla (\chi_\delta\ast \partial_E u_\e)\Big).
\end{eqnarray*}
We analyze the different terms separately in the upcoming six substeps.

\medskip
\substep{2.2} Proof that for all $K\ge1$,
\begin{multline}\label{eq:I0eps-re}
I_0^\e\,\ge\,(1-\tfrac1{2K})\int_U|\nabla w_\e|^2-\tfrac1{K}\int_U(Q_\e)^2-K\int_U|\nabla(\tilde w_\e-w_\e)|^2\\
-\e^2r_\e^{-2}CK\int_{\partial_{2r_\e}U}|(\psi,\phi,\df \phi)(\tfrac\cdot\e)|^2\big(1+|\nabla\bar u_\e|^2+|\chi_\delta\ast \D(u_\e)|^2\big)\\
-\e^2CK\int_U|(\psi,\phi,\df \phi)(\tfrac\cdot\e)|^2\big(|\nabla^2\bar u_\e|^2+|\nabla \chi_\delta \ast \D(u_\e)|^2\big).
\end{multline}
Adding and subtracting $w_\e$ to $\tilde w_\e$, we find from Young's inequality, for all $K\ge1$,
\[I_0^\e\,\ge\,(1-\tfrac1{2K})\int_U|\nabla w_\e|^2-\tfrac1{K}\int_U (Q_\e)^2-K\int_U|\nabla(\tilde w_\e-w_\e)|^2-K\int_U|\Div(w_\e)|^2.\]
Since $\Div(u_\e)=\Div(\bar u_\e)=0$, we get from~\eqref{e.div-2s2} and~\eqref{e.div-2s1},
\begin{multline}\label{eq:div-weps}
\Div(w_\e)\,=\,-\e \psi_{E} (\tfrac\cdot\e)\cdot \nabla( \eta_\e\partial_E\bar u_\e)
-\e\kappa\phi_{\chi_\delta\ast\D(u_\e)}(\tfrac\cdot\e) \cdot \nabla \eta_\e\\
-\e\kappa\df\phi_{\chi_\delta\ast\D(u_\e),E}(\tfrac\cdot\e)\cdot \eta_\e \nabla  (\chi_\delta\ast\partial_Eu_\e),
\end{multline}
and
the claim~\eqref{eq:I0eps-re} follows using the properties of $\eta_\e$.

\medskip
\substep{2.3} Proof that
\begin{multline}\label{eq:I1eps-re}
I_1^\e\,\lesssim\,\Big(\int_{\partial_{2r_\e}U}\big(1+|\nabla\bar u_\e|^2+|\chi_\delta\ast \D(u_\e)|^2\big)\Big)^\frac12\\
\times\Big(\int_{\partial_{3r_\e}U}|(1,\nabla\psi,\Sigma\mathds1_{\R^d\setminus \Ic},\nabla\phi,\Pi\mathds1_{\R^d\setminus\Ic})(\tfrac\cdot\e)|^2 [\nabla\bar u_\e]_{4\e}^2\Big)^\frac12.
\end{multline}

\medskip\noindent
First note that for all $n$ we have in $\e I_n$, since $\D(\psi_E)|_{I_n}=-E$,
\begin{eqnarray}
\D(\tilde w_\e)|_{\e I_n}
&=&-\Big(\!\D(T_1^\e[\bar u_\e])+\eta_\e\D(\psi_E)(\tfrac\cdot\e)T_0^\e[\partial_E\bar u_\e]\Big)\Big|_{\e I_n}\nonumber\\
&=&-(1-\eta_\e)|_{\e I_n}\fint_{\e(I_n+  B)}\D(\bar u_\e).\label{e.Dtilwe}
\end{eqnarray}
As by construction $J$ and $K$ are symmetric matrix fields, we may replace $\nabla \tilde w_\e$ by $\D(\tilde w_\e)$ in the definition of $I_1^\e$, and we thus find
\begin{equation*}
I_1^\e\,=\,-\sum_{n\in\Nc_\e(U)}(\eta_\e(1-\eta_\e))(\e x_n)\fint_{\e(I_n+B)}\D(\bar u_\e):\int_{\e I_n}\Big(J_E(\tfrac\cdot\e)\partial_E\bar u_\e+\kappa (K_{\chi_\delta\ast \D(u_\e)})(\tfrac\cdot\e)\Big).
\end{equation*}
By the hardcore assumption, using the properties of $\eta_\e$, and using~\eqref{eq:J_E_I_n},~\eqref{eq:K_E_I_n}, and the Lipschitz continuity of $E\mapsto f_n(E)$ (cf. Hypothesis~\ref{hyp:swim}), the claim~\eqref{eq:I1eps-re}  follows.

\medskip
\substep{2.4} Proof that
\begin{multline}\label{eq:I2eps}
I_2^\e\,\lesssim\,\Big(\int_{U}|\nabla w_\e|^2\Big)^\frac12\Big(\int_{\partial_{3r_\e}U}|\nabla\bar u_\e|^2\Big)^\frac12\\
+\Big(\int_{\partial_{3r_\e}U}|\nabla\bar u_\e|^2\Big)^\frac12\Big(\int_{\partial_{3r_\e}U}|(1,\nabla\psi,\nabla\phi)(\tfrac\cdot\e)|^2\big(1+|\nabla\bar u_\e|^2+|\chi_\delta\ast\D(u_\e)|^2\big)\Big)^\frac12\\
+\e\Big(\int_{\partial_{3r_\e}U}|\nabla\bar u_\e|^2\Big)^\frac12\Big(r_\e^{-2}\int_{\partial_{3r_\e}U}|(\psi,\phi)(\tfrac\cdot\e)|^2\big(1+|\nabla\bar u_\e|^2+|\chi_\delta\ast\D(u_\e)|^2\big)\Big)^\frac12\\
+\e \Big(\int_{\partial_{3r_\e}U}|\nabla\bar u_\e|^2\Big)^\frac12\Big(\int_{\partial_{3r_\e}U}|(\psi,\df\phi)(\tfrac\cdot\e)|^2\big(|\nabla^2\bar u_\e|^2+|\nabla\chi_\delta\ast\D(u_\e)|^2\big)\Big)^\frac12.
\end{multline}
Appealing to the boundary conditions in~\eqref{eq:Stokes}, recalling the notation~\eqref{eq:extension-fn}, and using $\int_{\e\partial I_n}\nu=0$ and $\int_{\e\partial I_n}\nu\otimes(x-\e x_n)=|\e I_n|\Id$, we have
\begin{equation*}
I_2^\e\,=\,-\sum_{n\in\Nc_\e(U)}\Big(\fint_{\e I_n}\D(\tilde w_\e)\Big):
\int_{\e\partial I_n}\sigma(u_\e,P_\e+P_\e'-P_{\e,n}'')\nu\otimes(x-\e x_n),
\end{equation*}
and thus, using~\eqref{e.Dtilwe} again,
and noting that the identities $\int_{\e\partial I_n}\nu\otimes(\cdot-\e x_n)=|\e I_n|\Id$ and $\Div(\bar u_\e)=0$ allow to remove any constant from the pressure field $P_\e$ in this expression,
\begin{equation}\label{eq:pre-I1eps-2nd}
I_2^\e\,=\,\sum_{n\in\Nc_\e(U)}(1-\eta_\e)(\e x_n)\Big(\fint_{\e (I_n+B)}\D(\bar u_\e)\Big):
\int_{\e\partial I_n}\sigma(u_\e,P_\e-P_{\e,n}''')\nu\otimes(x-\e x_n),
\end{equation}
where we have chosen $P_{\e,n}''':=\fint_{\e(I_n+B) \setminus \e I_n} P_\e$.
In order to estimate the right-hand side, we shall turn surface integrals into volume integrals, proceeding as for the construction of $K_E$ in Lemma~\ref{lem:KE}. 
More precisely, for all~$n$, we consider the following Neumann problem, 
\begin{align*}
\left\{\begin{array}{ll}
-\triangle u_\e^n+\nabla P_\e^n=f_{n,\e} \big({\chi_\delta\ast\D(u_\e)}_{\nmid \e I_n}\big) ,&\text{in $\e I_n$},\\
\Div(u_\e^n)=0,&\text{in $\e I_n$},\\
\sigma(u_\e^n,P_\e^n)\nu=\sigma(u_\e,P_\e-P_{\e,n}'''
)\nu, &\text{on $\e\partial I_n$}.
\end{array}\right.
\end{align*}
As for~\eqref{eq:Neumann_phi}, we can show that there is a unique solution $u_\e^n\in H^1_0(\e I_n)^d$ with $\fint_{\e I_n}u_\e^n=0$ and~$\fint_{\e I_n}\nabla u_\e^n\in\Md_0^\Sym$, and a unique pressure $P_\e^n\in\Ld^2(\e I_n)/\R$, such that
\begin{equation}\label{eq:bnd-ueps-n-extension}
\|(\nabla u_\e^n,P_\e^n)\|_{\Ld^2(\e I_n)}\lesssim\|\!\D(u_\e)\|_{\Ld^2(\e(I_n+B))}+\big\|f_{n,\e} \big({\chi_\delta\ast\D(u_\e)}_{\nmid \e I_n}\big)\big\|_{\Ld^2(\e(I_n+B))}.
\end{equation}
In these terms, we can reformulate the surface integral in~\eqref{eq:pre-I1eps-2nd} as follows,
\begin{multline*}
\int_{\e\partial I_n}\sigma(u_\e,P_\e)\nu\otimes(x-\e x_n)
\,=\,\int_{\e I_n}\nabla_j\Big(\sigma(u_\e^n,P_\e^n)\ee_j\otimes(x-\e x_n)\Big)\\
\,=\,-\int_{\e I_n}f_{n,\e}\big(\chi_\delta\ast\D(u_\e)_{\nmid \e I_n}\big)\otimes(x-\e x_n)
+\int_{\e I_n}\sigma(u_\e^n,P_\e^n),
\end{multline*}
and thus, in view of~\eqref{eq:bnd-ueps-n-extension},
\begin{multline*}
\bigg|\int_{\e\partial I_n}\sigma(u_\e,P_\e)\nu\otimes(x-\e x_n)\bigg|\\
\,\lesssim\,
|\e I_n|^{\frac12}\Big(|\e I_n|^{\frac12}
+\|\!\D(u_\e)\|_{\Ld^2(\e(I_n+B))}
+\|\chi_\delta\ast\D(u_\e)\|_{\Ld^2(\e I_n)}\Big).
\end{multline*}
Inserting this into~\eqref{eq:pre-I1eps-2nd} and using the properties of $\eta_\e$, we are led to
\begin{equation*}
I_2^\e\,\lesssim\,\Big(\int_{\partial_{3r_\e}U}|\nabla\bar u_\e|^2\Big)^\frac12
\Big(\int_{\partial_{3r_\e}U}\big(1+|\nabla u_\e|^2+|\chi_\delta\ast\D(u_\e)|^2\big)\Big)^\frac12.
\end{equation*}
Inserting then the two-scale expansion to replace the norm of~$\nabla u_\e$ in the right-hand side,
\[u_\e\,=\,w_\e+\bar u_\e+\e\psi_E(\tfrac\cdot\e)\eta_\e\partial_E\bar u_\e+\e\kappa\phi_{\chi_\delta\ast\D( u_\e)}(\tfrac\cdot\e)\eta_\e,\]
the claim~\eqref{eq:I2eps} follows.

\medskip
\substep{2.5} Proof that
\begin{multline}\label{eq:I3eps}
|I_3^\e|\,\lesssim\,\e\Big(\int_{U}|\nabla\tilde w_\e|^2\Big)^\frac12\Big(\int_{U}|\nabla(\chi_\delta\ast\D(u_\e))|^4+|\nabla^2(\chi_\delta\ast\D(u_\e))|^2\\
+|\nabla(\chi_\delta\ast \D(u_\e))|^2\big(1+[\chi_\delta\ast\D(u_\e)]_{4\e}^4\big)\Big)^\frac12.
\end{multline}
We start by decomposing
\begin{multline}\label{eq:decomp-I3eps}
I_3^\e\,=\,-\tfrac\kappa\e\sum_{n\in\Nc_\e(U)}\int_{\e(I_n+B)}\Big(\eta_\e\tilde w_\e-\fint_{\e I_n}\eta_\e\tilde w_\e\Big)\cdot\Big(f_{n,\e}(\chi_\delta\ast\D(u_\e))-f_{n,\e} \big({\chi_\delta\ast\D(u_\e)}_{\nmid \e I_n}\big)\Big)\\
-\tfrac\kappa\e\sum_{n\in\Nc_\e(U)}\Big(\fint_{\e I_n}\eta_\e\tilde w_\e\Big)\cdot\int_{\e(I_n+B)}\Big(f_{n,\e}(\chi_\delta\ast\D(u_\e))-f_{n,\e} \big({\chi_\delta\ast\D(u_\e)}_{\nmid \e I_n}\big)\Big)\\
-\kappa\int_{U}\eta_\e\tilde w_\e\otimes\nabla(\chi_\delta\ast\partial_E u_\e):\df F_{\chi_\delta\ast\D(u_\e),E}(\tfrac\cdot\e).
\end{multline}
We now  estimate the first right-hand side term.
Using the Lipschitz regularity of~$f_{n,\e}$, applying Poincaré's inequality, and using the hardcore assumption, we get
\begin{multline*}
\bigg|\tfrac1\e\sum_{n\in\Nc_\e(U)}\int_{\e(I_n+B)}\Big(\eta_\e\tilde w_\e-\fint_{\e I_n}\eta_\e\tilde w_\e\Big)\cdot\Big(f_{n,\e}(\chi_\delta\ast\D(u_\e))-f_{n,\e} \big({\chi_\delta\ast\D(u_\e)}_{\nmid \e I_n}\big)\Big)\bigg|\\
\,\lesssim\,\e\Big(\int_{U}|\nabla(\eta_\e\tilde w_\e)|^2\Big)^\frac12\Big(\int_U|\nabla(\chi_\delta\ast\D(u_\e))|^2\Big)^\frac12.
\end{multline*}
We further appeal to the following version of Poincaré's inequality for~$\tilde w_\e\in H^1_0(U)^d$ in~$\partial_{2r_\e}U$ (this will be used several times in the proof): by the properties of $\eta_\e$,
\begin{equation}\label{e.good-Poinca-bord}
\int_U |\nabla \eta_\e|^2 \tilde w_\e^2 \,\lesssim\, r_\e^{-2} \int_{\partial_{2r_\e} U} \tilde w_\e^2 \,\lesssim \int_U |\nabla\tilde w_\e|^2.
\end{equation}
The above then becomes
\begin{multline}\label{eq:decomp-I3eps-first}
\bigg|\tfrac1\e\sum_{n\in\Nc_\e(U)}\int_{\e(I_n+B)}\Big(\eta_\e\tilde w_\e-\fint_{\e I_n}\eta_\e\tilde w_\e\Big)\cdot\Big(f_{n,\e}(\chi_\delta\ast\D(u_\e))-f_{n,\e} \big({\chi_\delta\ast\D(u_\e)}_{\nmid \e I_n}\big)\Big)\bigg|\\
\,\lesssim\,\e\Big(\int_{U}|\nabla\tilde w_\e|^2\Big)^\frac12\Big(\int_U|\nabla(\chi_\delta\ast\D(u_\e))|^2\Big)^\frac12.
\end{multline}
We turn to the analysis of the last two terms in~\eqref{eq:decomp-I3eps}. By a Taylor expansion (of the form $|u(a+b)-u(a)-cu'(a)|\lesssim b^2 \sup |u''|+|b-c||u'(a)|$), using the $C^2$ regularity of $f_{n,\e}$, we can estimate
\begin{multline*}
\bigg|\int_{\e(I_n+B)}\Big(f_{n,\e}(\chi_\delta\ast\D(u_\e))-f_{n,\e} \big({\chi_\delta\ast\D(u_\e)}_{\nmid \e I_n}\big)\Big)\\
-\Big(\fint_{\e I_n}\nabla_i(\chi_\delta\ast\partial_E u_\e)\Big)\int_{\e(I_n+B)}(x-\e x_n)_i~\partial_E f_{n,\e}({\chi_\delta\ast\D(u_\e)}_{\nmid \e I_n})\bigg|\\
\,\lesssim\,\e^2\int_{\e(I_n+B)}|\nabla(\chi_\delta\ast\D(u_\e))|^2+\e^2\int_{\e(I_n+B)}|\nabla^2(\chi_\delta\ast\D(u_\e))|.
\end{multline*}
Summing over $n$ and recognizing the definition of $\df F$, cf.~\eqref{eq:def-corF},
\[\df F_{E,E'}\,:=\,-\sum_n\tfrac{\mathds1_{I_n}}{|I_n|}\int_{I_n+B}\partial_{E'} f_n(E)\otimes(x-x_n),\]
we are led to 
\begin{multline*}
\bigg|\tfrac1\e\sum_{n\in\Nc_\e(U)}\Big(\fint_{\e I_n}\eta_\e\tilde w_\e\Big)\cdot\int_{\e(I_n+B)}\Big(f_{n,\e}(\chi_\delta\ast\D(u_\e))-f_{n,\e} \big({\chi_\delta\ast\D(u_\e)}_{\nmid \e I_n}\big)\Big)\\
+\int_{U}\eta_\e\tilde w_\e\otimes\nabla(\chi_\delta\ast\partial_E u_\e):\df F_{\chi_\delta\ast\D(u_\e), E}(\tfrac\cdot\e)\bigg|\\
\,\lesssim\,\e\Big(\int_{U}|\eta_\e\tilde w_\e|^2\Big)^\frac12\Big(\int_{U}|\nabla(\chi_\delta\ast\D(u_\e))|^4+|\nabla^2(\chi_\delta\ast\D(u_\e))|^2\Big)^\frac12\\
+\e\Big(\int_{U}|\nabla(\eta_\e\tilde w_\e)|^2\Big)^\frac12\Big(\int_{U}|\nabla(\chi_\delta\ast\D(u_\e))|^2\big(1+[\chi_\delta\ast\D(u_\e)]_{4\e}^4\big)\Big)^\frac12.
\end{multline*}
Further appealing to Poincaré's inequality and to~\eqref{e.good-Poinca-bord}, this becomes
\begin{multline*}
\bigg|\tfrac1\e\sum_{n\in\Nc_\e(U)}\Big(\fint_{\e I_n}\eta_\e\tilde w_\e\Big)\cdot\int_{\e(I_n+B)}\Big(f_{n,\e}(\chi_\delta\ast\D(u_\e))-f_{n,\e} \big({\chi_\delta\ast\D(u_\e)}_{\nmid \e I_n}\big)\Big)\\
+\int_{U}\eta_\e\tilde w_\e\otimes\nabla(\chi_\delta\ast\partial_E u_\e):\df F_{\chi_\delta\ast\D(u_\e), E}(\tfrac\cdot\e)\bigg|\\
\,\lesssim\,\e\Big(\int_{U}|\nabla\tilde w_\e|^2\Big)^\frac12\Big(\int_{U}|\nabla(\chi_\delta\ast\D(u_\e))|^4+|\nabla^2(\chi_\delta\ast\D(u_\e))|^2\\
+|\nabla(\chi_\delta\ast\D(u_\e))|^2\big(1+[\chi_\delta\ast\D(u_\e)]_{4\e}^4\big)\Big)^\frac12.
\end{multline*}
Combining this with~\eqref{eq:decomp-I3eps} and~\eqref{eq:decomp-I3eps-first}, the claim~\eqref{eq:I3eps} follows.

\medskip
\substep{2.6}
Proof that
\begin{eqnarray}
I_4^\e&\lesssim&
\Big(\int_{U}|\nabla \tilde w_\e|^2\Big)^\frac12\Big(\int_{\partial_{3r_\e}U}\big(1+|h|^2+|\nabla\bar u_\e|^2+|\chi_\delta\ast\D(u_\e)|^2\big)\Big)^\frac12,
\label{eq:I4eps}\\
I_5^\e
&\lesssim&
\e\Big(\int_U|\nabla\tilde w_\e|^2\Big)^\frac12\Big(
r_\e^{-2}\int_{\partial_{2r_\e}U}|(\psi,\Ups,\phi,\theta,\gamma)|^2\big(1+|\nabla\bar u_\e|^2+|\chi_\delta\ast\D(u_\e)|^2\big)\Big)^\frac12\nonumber\\
&&+\,\e\Big(\int_U|\nabla\tilde w_\e|^2\Big)^\frac12\Big(\int_U|(\psi,\Ups,\df\phi,\df\theta,\df\gamma,\triangle^{-1}\nabla\mathds1_{\Ic},\triangle^{-1}\nabla\df F)|^2\nonumber\\
&&\hspace{4,5cm}\times\big(|\langle\nabla\rangle h|^2+|\nabla^2\bar u_\e|^2+|\nabla\langle\nabla\rangle \chi_\delta\ast \D(u_\e)|^2\big)\Big)^\frac12\nonumber\\
&&+\,\e\Big(\int_U|\nabla\tilde w_\e|^2\Big)^\frac12\Big(\int_{U}|(\df^2\theta,\df^2\gamma,\triangle^{-1}\nabla\df^2F)(\tfrac\cdot\e)|^2|\nabla\chi_\delta\ast \D(u_\e)|^4\Big)^\frac12.\label{eq:I5eps}
\end{eqnarray}
Using properties of $\eta_\e$, the neutrality condition~\eqref{eq:neutrality}, and Poincaré's inequality, a direct estimate yields
\begin{equation*}
I_4^\e\,\lesssim\,
\Big(\int_{\partial_{3r_\e}U}\big(1+|h|^2+|\nabla\bar u_\e|^2+|\chi_\delta\ast\D(u_\e)|^2\big)\Big)^\frac12\Big(\int_{U}|\nabla \tilde w_\e|^2+\int_{U}|\nabla\eta_\e|^2|\tilde w_\e|^2\Big)^\frac12,
\end{equation*}
and the claimed estimate~\eqref{eq:I4eps} follows by applying~\eqref{e.good-Poinca-bord} again.
The bound~\eqref{eq:I5eps} on $I_5^\e$ is obtained by similar straightforward computations.

\medskip
\substep{2.7} Proof of~\eqref{eq:est-main}.\\
Starting from~\eqref{eq:decomp-I0-3}, combining estimates~\eqref{eq:I0eps-re}, \eqref{eq:I1eps-re}, \eqref{eq:I2eps}, \eqref{eq:I3eps}, \eqref{eq:I4eps}, and~\eqref{eq:I5eps}, adding and subtracting~$w_\e$ to~$\tilde w_\e$ in the right-hand side,
and applying Young's inequality, we get after straightforward simplifications, for all~$K\ge1$,
\begin{multline}\label{eq:pre-bndweps}
\int_U|\nabla w_\e|^2\,\lesssim\,\tfrac1{K}\int_U(Q_\e)^2+K\int_U|\nabla(\tilde w_\e-w_\e)|^2\\
+\int_{\partial_{3r_\e}U}|(1,\nabla\psi,\Sigma\mathds1_{\R^d\setminus \Ic},\nabla\phi,\Pi\mathds1_{\R^d\setminus\Ic})(\tfrac\cdot\e)|^2\big(1+|h|^2+|\nabla\bar u_\e|^2+[\nabla\bar u_\e]_{4\e}^2+|\chi_\delta\ast\D(u_\e)|^2\big)\\
+\e^2r_\e^{-2}K\int_{\partial_{3r_\e}U}|(\psi,\Ups,\phi,\theta,\gamma,\df \phi)(\tfrac\cdot\e)|^2\big(1+|\nabla\bar u_\e|^2+|\chi_\delta\ast \D(u_\e)|^2\big)\\
+\e^2K\int_U|(1,\psi,\Ups,\phi,\df \phi,\df\theta,\df\gamma,\df^2\theta,\df^2\gamma,\triangle^{-1}\nabla\mathds1_{\Ic},\triangle^{-1}\nabla\df F,\triangle^{-1}\nabla\df^2 F)(\tfrac\cdot\e)|^2\\
\times\Big(|\langle\nabla\rangle h|^2+|\nabla^2\bar u_\e|^2+|\nabla\langle\nabla\rangle \chi_\delta\ast \D(u_\e)|^2+|\nabla\chi_\delta\ast \D(u_\e)|^4\\
+|\nabla(\chi_\delta\ast\D (u_\e))|^2\big(1+[\chi_\delta\ast\D(u_\e)]_{4\e}^4\big)\Big).
\end{multline}
It remains to evaluate the norm of $\nabla(\tilde w_\e-w_\e)$. By definition of $\tilde w_\e$, we have
\begin{multline*}
\nabla(\tilde w_\e-w_\e)\,=\,\nabla(\bar u_\e-T_1^\e[\bar u_\e])+\nabla\big(\e\psi_E(\tfrac\cdot\e)\eta_\e(\partial_E\bar u_\e-T_0^\e[\partial_E\bar u_\e])\big)\\
+\nabla\big(\e\kappa\phi_{\chi_\delta\ast\D(u_\e)}(\tfrac\cdot\e)\eta_\e-\e\kappa\phi_{T_0^\e[\chi_\delta\ast \D( u_\e)]}(\tfrac\cdot\e)\eta_\e\big),
\end{multline*}
and thus, inserting the definition of the truncation operators $T_0^\e,T_1^\e$, and noting that $\eta_\e$ is constant in the support of the cut-off functions  $\{\rho_\e^n\}_n$ by definition, 
\begin{multline*}
\nabla(\tilde w_\e-w_\e)\,=\,\sum_n\rho_n^\e\Big(\nabla\bar u_\e-\fint_{\e(I_n+ B)}\nabla\bar u_\e\Big)
+\sum_n\nabla\psi_E(\tfrac\cdot\e)\eta_\e\rho_n^\e\Big(\partial_E\bar u_\e-\fint_{\e(I_n+ B)}\partial_E\bar u_\e\Big)
\\
+\kappa\sum_n\big(\nabla\phi_{|\chi_\delta\ast\D( u_\e)}(\tfrac\cdot\e)-\nabla\phi_{\chi_\delta\ast\D( u_\e)_{\nmid \e(I_n+B)}}(\tfrac\cdot\e)\big) \eta_\e\rho_n^\e
\\
+\sum_n\e\psi_E(\tfrac\cdot\e)\eta_\e\rho_n^\e\nabla\partial_E\bar u_\e
+\sum_n\e \kappa \df\phi_{\chi_\delta\ast\D(u_\e),E}(\tfrac\cdot\e) \otimes \eta_\e\rho_n^\e\nabla (\chi_\delta\ast\partial_Eu_\e) 
\\
+\sum_n\bigg(\bar u_\e-\Big(\fint_{\e(I_n+ B)}\bar u_\e\Big)-\Big(\fint_{\e(I_n+ B)}\nabla\bar u_\e\Big)(x-\e x_n)\bigg)\otimes\nabla\rho_n^\e\\
+\sum_n\e\psi_E(\tfrac\cdot\e)\otimes\eta_\e\nabla\rho_n^\e\Big(\partial_E\bar u_\e-\fint_{\e(I_n+ B)}\partial_E\bar u_\e\Big)
\\
+\sum_n\e \kappa\big(\phi_{\chi_\delta\ast\D( u_\e)}(\tfrac\cdot\e)-\phi_{\chi_\delta\ast\D( u_\e)_{\nmid \e(I_n+B)}}(\tfrac\cdot\e)\big)\otimes\eta_\e\nabla\rho_n^\e.
\end{multline*}
Using that
\begin{multline*}
{\int_{\e(I_n+B)}\big|\phi_{\chi_\delta\ast\D( u_\e)}(\tfrac\cdot\e)-\phi_{\chi_\delta\ast\D( u_\e)_{\nmid \e(I_n+B)}}(\tfrac\cdot\e)\big|^2}\\
\,\lesssim\,\e^2\int_{\e(I_n+B)} |\df \phi(\tfrac\cdot\e)|^2\Big|\chi_\delta\ast\D(u_\e)-\fint_{\e(I_n+B)}\chi_\delta\ast\D(u_\e)\Big|^2,
\end{multline*}
and
\begin{multline*}
{\int_{\e(I_n+B)}\big|\nabla\phi_{|\chi_\delta\ast\D( u_\e)}(\tfrac\cdot\e)-\nabla \phi_{\chi_\delta\ast\D( u_\e)_{\nmid \e(I_n+B)}}(\tfrac\cdot\e)\big|^2}\\
\,\lesssim\,\e^2\int_{\e(I_n+B)} |(\nabla \df \phi)(\tfrac\cdot\e)|^2\Big|\chi_\delta\ast\D( u_\e)-\fint_{\e(I_n+B)}\chi_\delta\ast\D( u_\e)\Big|^2,
\end{multline*}
a direct computation then leads us to
\begin{multline}\label{eq:weps-tildeweps-estim}
\int_U|\nabla(\tilde w_\e-w_\e)|^2\,\lesssim\,
\e^2\int_{U}|(1,\psi)(\tfrac\cdot\e)|^2|\nabla^2\bar u_\e|^2
+\e^2 \int_{U}|\df\phi(\tfrac\cdot\e)|^2|\nabla\chi_\delta\ast\D(u_\e)|^2
\\
+\sum_n\int_{\e(I_n+B)}|(\psi,\nabla\psi)(\tfrac\cdot\e)|^2\Big|\nabla\bar u_\e-\fint_{\e(I_n+B)}\nabla\bar u_\e\Big|^2
\\
+\sum_n\int_{\e(I_n+B)}|(\df\phi,\nabla\df\phi)(\tfrac\cdot\e)|^2\Big|\chi_\delta\ast\D( u_\e)-\fint_{\e(I_n+B)}\chi_\delta\ast\D( u_\e)\Big|^2.
\end{multline}
Inserting this into~\eqref{eq:pre-bndweps}, the conclusion~\eqref{eq:est-main} follows.

\medskip
\step3 Pressure estimate for~\eqref{eq:eqn-wQ}:
\begin{multline}\label{eq:est-main-pres-simpl}
\int_U(Q_\e)^2\,\lesssim\,\int_U|\nabla w_\e|^2
+r_\e\Big(1+\|(h,\nabla\bar u_\e,\chi_\delta\ast\D(u_\e))\|_{\Ld^\infty(U)}^2\Big)\\
+\e^2\Big(1+\|(h,\nabla\bar u_\e,\bar P_\e)\|_{W^{1,\infty}(U)}^2+\|\chi_\delta\ast\D(u_\e)\|_{W^{2,\infty}(U)}^6\Big)\\
\times\int_U\big(1+r_\e^{-2}\mathds1_{\partial_{3r_\e}U}\big)
\big|\big(1,\psi,\Ups,\Sigma\mathds1_{\R^d\setminus\Ic},\phi,\theta,\gamma,\df\phi,\df\theta,\df\Pi\mathds1_{\R^d\setminus\Ic},\df\gamma,\df^2\theta,\df^2\gamma,\\
\triangle^{-1}\nabla\mathds1_\Ic,\triangle^{-1}\nabla\df F,\triangle^{-1}\nabla\df^2F\big)(\tfrac\cdot\e)\big|^2,
\end{multline}
which follows from 
\begin{multline}
\int_U(Q_\e)^2\,\lesssim\,
\int_U|\nabla w_\e|^2
+\int_{\partial_{3r_\e}U}\big(1+|h|^2+|\nabla\bar u_\e|^2+|\chi_\delta\ast\D(u_\e)|^2\big)\\
+\e^2r_\e^{-2}\int_{\partial_{3r_\e}U}|(\psi,\Ups,\phi,\theta,\gamma,\df\theta,\df\gamma)(\tfrac\cdot\e)|^2
\big(1+|\nabla\bar u_\e|^2+|\chi_\delta\ast\D(u_\e)|^2\big)\\
+\e^2\int_U|(1,\psi,\Ups,\df\phi,\df\theta,\df\gamma,\df^2\theta,\df^2\gamma,\triangle^{-1}\nabla\mathds1_\Ic,\triangle^{-1}\nabla\df F,\triangle^{-1}\nabla\df^2F)(\tfrac\cdot\e)|^2\\
\hspace{1cm}\times\Big(|\langle\nabla\rangle h|^2+|\nabla^2\bar u_\e|^2+|\nabla\bar P_\e|^2+|\nabla\langle\nabla\rangle\chi_\delta\ast\D(u_\e)|^2+|\nabla\chi_\delta\ast\D(u_\e)|^4\\
\hspace{7cm}+|\nabla(\chi_\delta\ast\D(u_\e))|^2\big(1+[\chi_\delta\ast\D(u_\e)]_{4\e}^4\big)\Big)\\
+\sum_n\int_{\e(I_n+B)}|(\Sigma\mathds1_{\R^d\setminus\Ic})(\tfrac\cdot\e)|^2\Big|\nabla\bar u_\e-\fint_{\e(I_n+B)}\nabla\bar u_\e\Big|^2\\
+\sum_n\int_{\e(I_n+B)}|(\df\Pi\mathds1_{\R^d\setminus\Ic})(\tfrac\cdot\e)|^2\Big|\chi_\delta\ast\D(u_\e)-\fint_{\e(I_n+B)}\chi_\delta\ast\D(u_\e)\Big|^2,
\label{eq:bnd-Qeps}
\end{multline}
after taking uniform norms of $h,\bar u_\e,\chi_\delta\ast\D(u_\e)$ and using Poincar\'e's inequality in the last two summands.

\medskip\noindent
Similarly as in Step~2, we shall appeal to a truncated version of $Q_\e$,
\begin{multline*}
\tilde Q_\e\,:=\,P_\e\mathds1_{\R^d\setminus\e\Ic}+P_\e^*-T_0^\e[\bar P_\e]-\eta_\e\bb:T_0^\e[\D(\bar u_\e)]-\kappa\eta_\e \cc(T_0^\e[\chi_\delta\ast\D( u_\e)])
\\
-(\Sigma_E\mathds1_{\R^d\setminus\Ic})(\tfrac\cdot\e)\eta_\e T_0^\e[\partial_E\bar u_\e]-\kappa(\Pi_{T_0^\e[\chi_\delta\ast\D(u_\e)]}\mathds1_{\R^d\setminus\Ic})(\tfrac\cdot\e)\eta_\e,
\end{multline*}
where we recall that $P_\e^*$ stands for some locally constant pressure field, cf.~\eqref{eq:def-P*},
\[P_\e^*\,:=\,P'_{\e}\mathds1_{U\setminus\Ic_\e(U)}+\sum_{n\in\Nc_\e(U)}P''_{\e,n}\mathds1_{\e I_n},\]
and where we now choose the constants $P_\e'$ and $\{P_{\e,n}''\}_n$ in such a way that
\begin{equation}\label{eq:choice-csts-P}
\tilde Q_\e|_{\Ic_\e(U)}=0\qquad\text{and}\qquad\int_U\tilde Q_\e=0.
\end{equation}
Using the Bogovskii operator as in~\cite{DG-21a}, we can construct a vector field $S_\e\in H^1_0(U)^d$ such that $S_\e|_{\e I_n}$ is a constant for all $n\in\Nc_\e(U)$ and such that
\begin{gather}
\Div(S_\e)\,=\,-\tilde Q_\e\qquad\text{in $U$},\nonumber
\\
\int_U|\nabla S_\e|^2\,\lesssim\,\int_U|\tilde Q_\e|^2.\label{eq:def-Seps}
\end{gather}
Testing equation~\eqref{eq:eqn-wQ} with $S_\e$, using the property that $S_\e|_{\e I_n}$ is a constant for all~$n\in\Nc_\e(U)$, and using the boundary conditions in~\eqref{eq:Stokes-re}, we find
\begin{eqnarray*}
\lefteqn{\int_U\tilde Q_\e Q_\e~= ~-\int_U\nabla S_\e:\nabla w_\e
-\kappa\int_U\eta_\e S_\e\otimes\nabla(\chi_\delta\ast\partial_E u_\e):\df F_{\chi_\delta\ast\D(u_\e),E}(\tfrac\cdot\e)}\\
&&-\tfrac\kappa\e\sum_{n\in\Nc_\e(U)}\int_{\e(I_n+B)}\eta_\e S_\e\cdot\Big(f_{n,\e}(\chi_\delta\ast\D(u_\e))-f_{n,\e}\big(\chi_\delta\ast\D(u_\e)_{\nmid \e I_n}\big)\Big)\\
&&+\int_U(1-\eta_\e)(\lambda-\mathds1_{\Ic_\e(U)})S_\e\cdot h
-\kappa\int_US_\e\otimes\nabla\eta_\e:\Ff(\chi_\delta\ast\D(u_\e))\\
&&+\tfrac\kappa\e\sum_{n\in\Nc_\e(U)}\int_{\e(I_n+B)}(1-\eta_\e)S_\e\cdot f_{n,\e}\big(\chi_\delta\ast\D(u_\e)_{\nmid \e I_n}\big)\\
&&+\int_U(1-\eta_\e)\nabla S_\e:\Big(2(\Bb_\pas-\Id)\D(\bar u_\e)+2\kappa\Cc(\chi_\delta\ast\D(u_\e))+\kappa\Ff(\chi_\delta\ast\D(u_\e))\Big)\\
&&-\e\int_U\nabla S_\e:\Big(\big(2\psi_{E}\otimes_s-\Id\otimes\psi_E-\Ups_E\big)(\tfrac\cdot\e)\nabla(\eta_\e\partial_E\bar u_\e)
-\eta_\e h\otimes(\triangle^{-1}\nabla\mathds1_{\Ic})(\tfrac\cdot\e)\\
&&\hspace{1.7cm}+\kappa\big(2\phi_{\chi_\delta\ast\D(u_\e)}\otimes_s-\Id\otimes\phi_{\chi_\delta\ast\D(u_\e)}-\theta_{\chi_\delta\ast\D(u_\e)}+\gamma_{\chi_\delta\ast\D(u_\e)}\otimes\Id\big)(\tfrac\cdot\e)\nabla\eta_\e\\
&&\hspace{1.7cm}+\kappa\big(2\df \phi_{\chi_\delta\ast\D(u_\e),E}\otimes_s-\Id\otimes\df \phi_{\chi_\delta\ast\D(u_\e),E}-\df \theta_{\chi_\delta\ast\D(u_\e),E}\\
&&\hspace{7cm}+\df\gamma_{\chi_\delta\ast\D(u_\e),E}\otimes\Id\big)(\tfrac\cdot\e)\eta_\e\nabla (\chi_\delta\ast \partial_E u_\e)\Big)\\
&&-\e\int_U\eta_\e\nabla_iS_\e\otimes\nabla(\chi_\delta\ast\partial_Eu_\e):(\triangle^{-1}\nabla_i\df F_{|\chi_\delta\ast\D(u_\e),E})(\tfrac\cdot\e)\\
&&+\e\int_US_\e\cdot\Big((\triangle^{-1}\nabla_j\mathds1_\Ic)(\tfrac\cdot\e)\nabla_j(\eta_\e h)
-\kappa\gamma_{\chi_\delta\ast\D(u_\e)}(\tfrac\cdot\e)\triangle\eta_\e\\
&&\hspace{2cm}-\kappa\df\gamma_{\chi_\delta\ast\D(u_\e),E}(\tfrac\cdot\e):\eta_\e\triangle(\chi_\delta\ast\partial_Eu_\e)\\
&&\hspace{2cm}-\kappa(\triangle^{-1}\nabla_i\df F_{|\chi_\delta\ast\D(u_\e),E})(\tfrac\cdot\e)\nabla_i(\eta_\e\nabla(\chi_\delta\ast\partial_Eu_\e))\\
&&\hspace{2cm}+2\kappa\big(\df\theta_{\chi_\delta\ast\D(u_\e),E}-\df\gamma_{\chi_\delta\ast\D(u_\e),E}\otimes\Id\big)(\tfrac\cdot\e):\nabla(\chi_\delta\ast\partial_Eu_\e)\otimes\nabla\eta_\e\\
&&\hspace{2cm}+\kappa\big(\df^2 \theta_{\chi_\delta\ast\D(u_\e),E,E'}-\df^2\gamma_{\chi_\delta\ast\D(u_\e),E,E'}\otimes\Id-\triangle^{-1}\nabla\df^2F_{|\chi_\delta\ast\D(u_\e),E,E'}\big)(\tfrac\cdot\e)\\
&&\hspace{7cm}:\eta_\e\nabla (\chi_\delta\ast \partial_{E'} u_\e)\otimes\nabla (\chi_\delta\ast \partial_E u_\e)\Big).
\end{eqnarray*}
Adding and subtracting $Q_\e$ to $\tilde Q_\e$ in the left-hand side, and proceeding as in Step~2 to estimate the different contributions, we deduce
\begin{eqnarray*}
\lefteqn{\int_U(Q_\e)^2~\lesssim~\int_U(\tilde Q_\e-Q_\e)^2+\Big(\int_U|\nabla S_\e|^2\Big)^\frac12\Big(\int_U|\nabla w_\e|^2\Big)^\frac12}\\
&&+\Big(\int_U|\nabla S_\e|^2\Big)^\frac12\Big(\int_{\partial_{3r_\e}U}\big(1+|h|^2+|\nabla\bar u_\e|^2+|\chi_\delta\ast\D(u_\e)|^2\big)\Big)^\frac12\\
&&+\e\Big(\int_U|\nabla S_\e|^2\Big)^\frac12\Big(r_\e^{-2}\int_{\partial_{3r_\e}U}|(\psi,\Ups,\phi,\theta,\gamma,\df\theta,\df\gamma)|^2\big(1+|\nabla\bar u_\e|^2+|\chi_\delta\ast\D(u_\e)|^2\big)\Big)^\frac12\\
&&+\e\Big(\int_U|\nabla S_\e|^2\Big)^\frac12\bigg(\int_U|(1,\psi,\Ups,\df\phi,\df\theta,\df\gamma,\df^2\theta,\df^2\gamma,\triangle^{-1}\nabla\mathds1_\Ic,\triangle^{-1}\nabla\df F,\triangle^{-1}\nabla\df^2 F)|^2\\
&&\hspace{3cm}\times\Big(|\langle\nabla\rangle h|^2+|\nabla^2\bar u_\e|^2+|\nabla\langle\nabla\rangle\chi_\delta\ast\D(u_\e)|^2+|\nabla\chi_\delta\ast\D(u_\e)|^4\\
&&\hspace{7cm}+|\nabla(\chi_\delta\ast\D(u_\e))|^2\big(1+[\chi_\delta\ast\D(u_\e)]_{4\e}^4\big)\Big)\bigg)^\frac12.
\end{eqnarray*}
Hence, by~\eqref{eq:def-Seps} and Young's inequality,
\begin{multline}\label{eq:bnd-Qeps-pre}
{\int_U(Q_\e)^2~\lesssim~\int_U(\tilde Q_\e-Q_\e)^2+\int_U|\nabla w_\e|^2
+\int_{\partial_{3r_\e}U}\big(1+|h|^2+|\nabla\bar u_\e|^2+|\chi_\delta\ast\D(u_\e)|^2\big)}\\
+\e^2r_\e^{-2}\int_{\partial_{3r_\e}U}|(\psi,\Ups,\phi,\theta,\gamma,\df\theta,\df\gamma)|^2
\big(1+|\nabla\bar u_\e|^2+|\chi_\delta\ast\D(u_\e)|^2\big)\\
+\e^2\int_U|(1,\psi,\Ups,\df\phi,\df\theta,\df\gamma,\df^2\theta,\df^2\gamma,\triangle^{-1}\nabla\mathds1_\Ic,\triangle^{-1}\nabla\df F,\triangle^{-1}\nabla\df^2 F)|^2\\
\times\Big(|\langle\nabla\rangle h|^2+|\nabla^2\bar u_\e|^2+|\nabla\langle\nabla\rangle\chi_\delta\ast\D(u_\e)|^2+|\nabla\chi_\delta\ast\D(u_\e)|^4\\
+|\nabla(\chi_\delta\ast\D(u_\e))|^2\big(1+[\chi_\delta\ast\D(u_\e)]_{4\e}^4\big)\Big).
\end{multline}
It remains to estimate the norm of $\tilde Q_\e-Q_\e$ in the right-hand side. By definition of $\tilde Q_\e$, we have
\begin{multline*}
\tilde Q_\e-Q_\e\,=\,\bar P_\e-T_0^\e[\bar P_\e]
+\eta_\e\bb:\big(\D(\bar u_\e)-T_0^\e[\D(\bar u_\e)]\big)\\
+\kappa\eta_\e \big(\cc(\chi_\delta\ast \D(u_\e))-\cc(T_0^\e[\chi_\delta\ast\D u_\e)])\big)
+(\Sigma_E\mathds1_{\R^d\setminus\Ic})(\tfrac\cdot\e)\eta_\e(\partial_E\bar u_\e-T_0^\e[\partial_E\bar u_\e])\\
+\kappa\big(\Pi_{\chi_\delta\ast\D(u_\e)}\mathds1_{\R^d\setminus\Ic}-\Pi_{T_0^\e[\chi_\delta\ast\D(u_\e)]}\mathds1_{\R^d\setminus\Ic}\big)(\tfrac\cdot\e)\,\eta_\e,
\end{multline*}
and thus, using~\eqref{eq:Cc-cc-reg} and proceeding as for~\eqref{eq:weps-tildeweps-estim}, we get
\begin{multline}
\int_U(\tilde Q_\e-Q_\e)^2\,\lesssim\,
\e^2\int_U\big(|\nabla^2\bar u_\e|^2+|\nabla\bar P_\e|^2+|\nabla\chi_\delta\ast\D(u_\e)|^2\big)\\
+\sum_n\int_{\e(I_n+B)}|(\Sigma\mathds1_{\R^d\setminus\Ic})(\tfrac\cdot\e)|^2\Big|\nabla\bar u_\e-\fint_{\e(I_n+B)}\nabla\bar u_\e\Big|^2\\
+\sum_n\int_{\e(I_n+B)}|(\df\Pi\mathds1_{\R^d\setminus\Ic})(\tfrac\cdot\e)|^2\Big|\chi_\delta\ast\D(u_\e)-\fint_{\e(I_n+B)}\chi_\delta\ast\D(u_\e)\Big|^2.
\end{multline}
Inserting this into~\eqref{eq:bnd-Qeps-pre}, the conclusion~\eqref{eq:bnd-Qeps} follows.

\medskip
\step4 Conclusion.\\
Choosing $K\ge1$ large enough to absorb
part of the pressure into the left-hand side, the combination of~\eqref{eq:est-main-simpl} and~\eqref{eq:est-main-pres-simpl} yields
\begin{multline}\label{eq:est-main+}
\int_U|\nabla w_\e|^2+\int_U(Q_\e)^2
\,\lesssim\,
\Big(1+\|(h,\nabla\bar u_\e,\bar P_\e)\|_{W^{1,\infty}(U)}^2+\|\chi_\delta\ast\D(u_\e)\|_{W^{2,\infty}(U)}^6\Big)\\
\times\bigg(
\int_{\partial_{3r_\e}U}|(1,\nabla\psi,\Sigma\mathds1_{\R^d\setminus\Ic},\nabla\phi,\Pi\mathds1_{\R^d\setminus\Ic})(\tfrac\cdot\e)|^2\\
+\e^2\int_U\big(1+r_\e^{-2}\mathds1_{\partial_{3r_\e}U}\big)|(1,\psi,\nabla\psi,\Ups,\Sigma\mathds1_{\R^d\setminus\Ic},\phi,\theta,\gamma,\df \phi,\nabla\df \phi,\df \theta,\df\Pi\mathds1_{\R^d\setminus\Ic},\df\gamma,\\
\df^2\theta,\df^2\gamma,\triangle^{-1}\nabla\mathds1_{\Ic},\triangle^{-1}\nabla\df F,\triangle^{-1}\nabla\df^2F)(\tfrac\cdot\e)|^2\bigg).
\end{multline}
By Proposition~\ref{lem:well-posed}, we have $\|\nabla u_\e\|_{\Ld^2(U)}\lesssim1+\|h\|_{\Ld^2(U)}$, and thus for all $s>0$,
\[\|\chi_\delta\ast \D(u_\e)\|_{W^{1+s,\infty}(U)}\,\lesssim\,\|\chi_\delta\|_{H^{1+s}(\R^d)}(1+\|h\|_{\Ld^2(U)}).\]
By the regularity theory for the Stokes equation, using~\eqref{eq:Cc-cc-reg}, we then deduce that the solution $(\bar u_\e,\bar P_\e)$ of~\eqref{eq:def-barueps} satisfies for all $s>0$,
\begin{eqnarray}
\|(\nabla\bar u_\e,\bar P_\e)\|_{W^{1,\infty}(U)}&\lesssim&
\|h\|_{W^{s,\infty}(U)}+  \|\Bb_\act(\chi_\delta\ast\D(u_\e))\|_{W^{1+s,\infty}(U)}
\nonumber\\
&\lesssim&
\|h\|_{W^{s,\infty}(U)}+\|\chi_\delta\|_{H^{1+s}(\R^d)}(1+\|h\|_{\Ld^2(U)}).\label{eq:apriori-barueps}
\end{eqnarray}
Inserting these estimates into~\eqref{eq:est-main+}, we get
\begin{multline}\label{eq:est-wQ-appl}
\int_U|\nabla w_\e|^2+\int_U(Q_\e)^2\\
\,\lesssim_{\chi_\delta}\,
\big(1+\|h\|_{W^{1,\infty}(U)}^6\big)\bigg(
\int_{\partial_{3r_\e}U}\big|\big(1,\nabla\psi,\Sigma\mathds1_{\R^d\setminus\Ic},\nabla\phi,\Pi\mathds1_{\R^d\setminus\Ic}\big)(\tfrac\cdot\e)\big|^2\\
+\e^2\int_U\big(1+r_\e^{-2}\mathds1_{\partial_{3r_\e}U}\big)|(1,\psi,\nabla\psi,\Ups,\Sigma\mathds1_{\R^d\setminus\Ic},\phi,\theta,\gamma,\df \phi,\nabla\df \phi,\df \theta,\df\Pi\mathds1_{\R^d\setminus\Ic},\df\gamma,\\
\df^2\theta,\df^2\gamma,\triangle^{-1}\nabla\mathds1_{\Ic},\triangle^{-1}\nabla\df F,\triangle^{-1}\nabla\df^2F)(\tfrac\cdot\e)|^2\bigg).
\end{multline}
By the ergodic theorem and by the sublinearity of correctors, cf.~Lemmas~\ref{lem:pass-cor}, \ref{lem:Jzet}, \ref{lem:cor-act}, \ref{lem:cor-gamma}, and~\ref{lem:KE}, we have for any fixed $r>0$, almost surely,
\begin{equation}\label{eq:conv-erg-1}
\lim_{\e\downarrow0}\int_{\partial_{3r}U}\big|\big(1,\nabla\psi,\Sigma\mathds1_{\R^d\setminus\Ic},\nabla\phi,\Pi\mathds1_{\R^d\setminus\Ic}\big)(\tfrac\cdot\e)\big|^2~\lesssim~Cr,
\end{equation}
and
\begin{multline}\label{eq:conv-erg-2}
\lim_{\e\downarrow0}\e^2\int_U\big(1+r^{-2}\mathds1_{\partial_{3r}U}\big)|(1,\psi,\nabla\psi,\Ups,\Sigma\mathds1_{\R^d\setminus\Ic},\phi,\theta,\gamma,\df \phi,\nabla\df \phi,\df \theta,\df\Pi\mathds1_{\R^d\setminus\Ic},\df\gamma,\\
\df^2\theta,\df^2\gamma,\triangle^{-1}\nabla\mathds1_{\Ic},\triangle^{-1}\nabla\df F,\triangle^{-1}\nabla\df^2F)(\tfrac\cdot\e)|^2\,=\,0.
\end{multline}
Some care is however needed to prove these convergence results as we take suprema in the notation~\eqref{eq:notation-|psi|}--\eqref{eq:notation-|phi|}: while the linear dependence of $\nabla\psi_E$ on $E$ makes the suprema $|\nabla\psi|=\sup_E|E|^{-1}|\nabla\psi_E|$ trivial, the same is not true for $\nabla\phi_E$. In that case, we use the Sobolev embedding in form of
\begin{equation}\label{eq:sobolev-delta}
|(\nabla\phi)(\tfrac\cdot\e)|^2\,\le\,\sup_{|E|\le C_\delta(h)}|\nabla\phi_E(\tfrac \cdot \e)|^2\,\lesssim\,\sum_{l=0}^k\int_{|E|\le C_\delta(h)}|(\nabla\df^l\phi_E)(\tfrac\cdot\e)|^2\,dE,
\end{equation}
for some $k>\frac12\dim \Md_0^\Sym$, and thus
\[\int_{\partial_{3r}U}|(\nabla\phi)(\tfrac\cdot\e)|^2\,\lesssim\,\sum_{l=0}^k\int_{|E|\le C_\delta(h)}\Big(\int_{\partial_{3r}U}|(\nabla\df^l\phi_E)(\tfrac\cdot\e)|^2\Big)\,dE,\]
to which the ergodic theorem for (linearized) correctors $\nabla\df^l\phi$ in Lemma~\ref{lem:cor-act} can now be applied, leading to the claim~\eqref{eq:conv-erg-1}.
Similarly, writing
\begin{eqnarray*}
\e^2\int_U\big(1+r^{-2}\mathds1_{\partial_{3r}U}\big)|\phi(\tfrac\cdot\e)|^2&\le&\e^2\int_U\big(1+r^{-2}\mathds1_{\partial_{3r}U}\big)\sup_{|E|\le C_\delta(h)}|\phi_E(\tfrac\cdot\e)|^2\\
&\lesssim&\sum_{l=0}^k\int_{|E|\le C_\delta(h)}\bigg(\e^2\int_U\big(1+r^{-2}\mathds1_{\partial_{3r}U}\big)|\df^l\phi_E(\tfrac\cdot\e)|^2\bigg)\,dE,
\end{eqnarray*}
the claim~\eqref{eq:conv-erg-2} indeed follows from the sublinearity of the (linearized) correctors $\df^l\phi$ in Lemma~\ref{lem:cor-act}.

\medskip\noindent
Next, inserting~\eqref{eq:conv-erg-1}--\eqref{eq:conv-erg-2} into~\eqref{eq:est-wQ-appl} and appealing to a diagonalization argument, we conclude that there exists a (random) sequence~$r_\e\downarrow0$ such that for this choice we have, almost surely,
\[\lim_{\e\downarrow0}\Big(\int_U|\nabla w_\e|^2+\int_U(Q_\e)^2\Big)\,=\,0,\]
that is, $w_\e\to0$ in $H^1_0(U)$ and $Q_\e\to0$ in $\Ld^2(U)$.
On the other hand, note that a priori estimates~\eqref{lem:well-posed-bd} and~\eqref{eq:apriori-barueps} entail that up to an extraction we have $u_\e\cvf u_0$ and $\bar u_\e\cvf\bar u_0$ in $H^1_0(U)$, for some $u_0,\bar u_0\in H^1_0(U)^d$. Passing to the weak limit in equation~\eqref{eq:def-barueps} along this subsequence, we find that $\bar u_0$ satisfies
\begin{equation}\label{eq:pre-eqn-baru0}
-\Div(2\Bb_\pas\D(\bar u_0))+\nabla\bar P_0=(1-\lambda)h+\Div(2\kappa\Bb_\act(\chi_\delta\ast\D( u_0))).
\end{equation}
Now, by definition of the two-scale expansion error $w_\e$, cf.~\eqref{eq:def-2sc-error}, together with the sublinearity of correctors, cf.~Lemmas~\ref{lem:pass-cor} and~\ref{lem:cor-act}, the convergence $w_\e\to0$ in $H^1_0(U)$ implies $u_\e-\bar u_\e\to0$ in $\Ld^2(U)$, and thus $u_0=\bar u_0$. From~\eqref{eq:pre-eqn-baru0}, we deduce that $\bar u:=u_0=\bar u_0$ actually satisfies the homogenized equation~\eqref{eq:homog}. In view of the well-posedness for the latter, we conclude $u_\e\cvf\bar u$ in $H^1_0(U)$ independently of extractions.

\medskip\noindent
We turn to the convergence of the pressure field.
Recall that we have shown $Q_\e\to0$ in~$\Ld^2(U)$.
The a priori estimate~\eqref{eq:apriori-barueps} ensures $\bar P_\e\cvf\bar P$ in $\Ld^2(U)$, where $\bar P$ is the unique pressure field in $\Ld^2(U)/\R$ for the homogenized equation~\eqref{eq:pre-eqn-baru0}. By definition of $Q_\e$, cf.~\eqref{eq:def-2sc-error}, together with the ergodic theorem for corrector pressures, cf.~Lemmas~\ref{lem:pass-cor} and~\ref{lem:cor-act}, and with the choice~\eqref{eq:choice-csts-P} of $P_\e'$, the convergence of the pressure follows.
\qed

\subsection{Quantitative homogenization and limit $\delta\downarrow0$}
This section is devoted to the proof of Theorem~\ref{theor:loc-nonlin-diag}.
As the above proof is semi-quantitative, one can infer convergence rates provided that quantitative mixing assumptions such as Hypothesis~\ref{hyp:part-strong} are further made on the statistical ensemble of inclusions.
Quantitative rates then allow in particular to let the parameter $\delta$ tend to $0$ in a nontrivial regime.
We split the proof into three steps.

\medskip
\step1 Convergence of the homogenized equation~\eqref{eq:homog} as $\delta\downarrow0$.\\
Writing equation~\eqref{eq:homog} as
\[-\Div(2\Bb_\pas\D(\bar u_\delta))+\nabla\bar P_\delta=(1-\lambda)h+\Div(2\kappa\Bb_\act(\chi_\delta\ast\D(\bar u_\delta))),\]
and appealing to the regularity theory for the Stokes equation, the unique solution $(\bar u_\delta,\bar P_\delta)\in H^1_0(U)^d\times\Ld^2(U)/\R$ satisfies for all $0<\eta<1$,
\[\|(\nabla\bar u_\delta,\bar P_\delta)\|_{W^{1-\eta,\infty}(U)}\,\lesssim_\eta\,\|h\|_{\Ld^\infty(U)}+  \kappa\|\Bb_\act(\chi_\delta\ast\D(\bar u_\delta))\|_{W^{1-\eta,\infty}(U)}.\]
Using~\eqref{eq:Cc-cc-reg}, we deduce
\begin{eqnarray*}
\|(\nabla\bar u_\delta,\bar P_\delta)\|_{W^{1-\eta,\infty}(U)}
&\lesssim_\eta&\|h\|_{\Ld^\infty(U)}+\kappa\lambda\|\chi_\delta\ast\D(\bar u_\delta)\|_{W^{1-\eta,\infty}(U)}\\
&\lesssim_\eta&\|h\|_{\Ld^\infty(U)}+\kappa\lambda\|\nabla \bar u_\delta\|_{W^{1-\eta,\infty}(U)}.
\end{eqnarray*}
The smallness condition~\eqref{eq:cond-small-kappa} yields $\kappa\lambda\lesssim\kappa\ell^{-d}\ll1$, and we thus infer
\begin{equation*}
\|(\nabla\bar u_\delta,\bar P_\delta)\|_{W^{1-\eta,\infty}(U)}
\,\lesssim_\eta\,\|h\|_{\Ld^\infty(U)}.
\end{equation*}
Up to an extraction, this implies $(\nabla\bar u_\delta,\bar P_\delta)\to(\nabla\bar u_0,\bar P_0)$ in $\Ld^\infty(U)$,
for some limit $(\bar u_0,\bar P_0)\in H^1_0(U)^d\times\Ld^2(U)/\R$.
Passing to the limit in equation~\eqref{eq:homog}, we find that $(\bar u_0,\bar P_0)$ satisfies equation~\eqref{eq:homog-0}.
Provided that $\kappa\lambda\ll1$ is small enough, which is ensured by the smallness condition~\eqref{eq:cond-small-kappa}, the well-posedness of~\eqref{eq:homog-0} follows from the same argument as for~\eqref{eq:homog}. We conclude that $(\bar u_0,\bar P_0)=(\bar u,\bar P)$ is the unique solution of~\eqref{eq:homog-0} and that $\bar u_\delta\to\bar u$ in~$W^{1,\infty}(U)$ and $\bar P_\delta\to\bar P$ in $\Ld^\infty(U)$.

\medskip\noindent
Combining this with Theorem~\ref{th:homog}, by a diagonalization argument, we deduce that there is a (random) sequence $\delta_\e^\circ\downarrow0$ such that, for any sequence $0<\delta_\e\le\delta_\e^\circ$, the solution $(u_\e,P_\e)$ of~\eqref{eq:Stokes-re} with $\delta=\delta_\e$ satisfies, as $\e\downarrow0$,
\begin{equation}\label{eq:diagonal}
\begin{array}{ccll}
u_\e&\cvf& \bar u,&\text{in $H^1_0(U)^d$},\\
P_\e\mathds1_{U\setminus\Ic_\e(U)}&\cvf&(1-\lambda)\bar P+(1-\lambda)\bb:\D(\bar u)&\\
&&\hspace{1.2cm}+(1-\lambda)\kappa\big( \cc(\D(\bar u))-\fint_U \cc(\D(\bar u))\big),\quad&\text{in $\Ld^2(U)$}.
\end{array}
\end{equation}
To improve on such a diagonal result, we need to prove a quantitative version of Theorem~\ref{th:homog} and capture the precise dependence on $\delta$. This is the purpose of the next two steps.

\medskip
\step2 Corrector estimates. \\
As we proved in~\cite{DG-21b} for passive correctors, under a quantitative mixing assumption such as Hypothesis~\ref{hyp:part-strong},
we have for all $s<\infty$,
\begin{gather}
\label{eq:def_mu_d}
\expec{|(\nabla\psi,\Sigma\mathds1_{\R^d\setminus\Ic},\nabla\Ups)(x)|^s}^\frac1s \,\lesssim_s\,1, \nonumber \\
\expec{|(\psi,\Ups)(x)|^s}^\frac1s \,\lesssim_s\, \mu_d(x):=\left\{
\begin{array}{lcl}
\log^\frac12 (2+|x|)&:&d=2,\\
1&:&d>2.
\end{array}
\right.
\end{gather}
which optimally quantifies the sublinearity of passive correctors.
The method in~\cite{DG-21b} applies mutadis mutandis  to active correctors, and yields the following: for all $E,E',E''\in\Md_0^\Sym$ and $s<\infty$, we get
\begin{align}
&\E\Big[{|(\nabla\phi_E,\Pi_E\mathds1_{\R^d\setminus\Ic},\nabla\theta_E,\nabla\df\phi_E,\nabla\df\theta_{E,E'},\nabla\df^2\phi_{E,E',E''},\nabla\df^2\theta_{E,E',E''})(x)|^s}\Big]^\frac1s \,\lesssim_{s,E,E',E''}\, 1,\nonumber\\
&\E\Big[|(\phi_E,\gamma_E,\theta_E,\df\phi_{E,E'},\df\gamma_{E,E'},\df\theta_{E,E'},\nonumber\\
&\hspace{3cm}\df^2\phi_{E,E',E''},\df^2\gamma_{E,E',E''},\df^2\theta_{E,E',E''})(x)|^s\Big]^\frac1s \,\lesssim_{s,E,E',E''}\, \mu_d(x).
\label{eq:moment-estimates}
\end{align}
Yet, for our purposes, we further need corresponding estimates on suprema such as
\[|\nabla\phi|=\sup_{|E|\le C_\delta(h)}\langle E\rangle^{-1}|\nabla\phi_E|,\]
which is not trivial due to the nonlinear dependence on $E$. As we aim at capturing the best dependence on $\delta$, we cannot appeal to brutal Sobolev estimates as in~\eqref{eq:sobolev-delta}. Instead, we shall take advantage of the above moment estimates~\eqref{eq:moment-estimates} together with Hypothesis~\ref{hyp:swim-large}.
More precisely, we decompose
\begin{equation}\label{eq:est-decomp-nabphi}
|\nabla\phi|\,\le\,\sup_{E}\,\langle E\rangle^{-1}|\nabla\phi_E^\infty|+\sup_{E}\,\langle E\rangle^{-1}|\nabla(\phi_E-\phi_E^\infty)|,
\end{equation}
where we compare $\phi_E$ to the random field $\phi_E^\infty$ that is defined via the same corrector problem~\eqref{eq:phi-cor} \&~\eqref{eq:bnd-unif-E-phiE} with the swimming forces $\{f_n(E)\}_n$ replaced by their large-$E$ approximations $\{f^\infty(E)\xi_n\}_n$, cf.~Hypothesis~\ref{hyp:swim-large}. On the one hand, as $f^\infty(E)$ is a deterministic function of $E$, the supremum of $\nabla\phi_E^\infty$ over $E$ becomes trivial and moment estimates can be established in the following form, for all~$s<\infty$,
\[\E\Big[\big(\textstyle\sup_E\,\langle E\rangle^{-1}|\nabla\phi_E^\infty|\big)^s\Big]^\frac1s\,\lesssim_s\,1.\]
On the other hand, by the Sobolev embedding, we can bound for all $s<\infty$, provided $s>\dim\Md_0^\Sym$,
\[\sup_{E}\,\langle E\rangle^{-1}|\nabla(\phi_E-\phi_E^\infty)|\,\lesssim\,\bigg(\int_{\Md_0^\Sym}\langle E\rangle^{-s}\big(|\nabla(\phi_E-\phi_E^\infty)|+|\nabla(\df\phi_E-\df\phi_E^\infty)|\big)^s\,dE\bigg)^\frac1s\]
and thus, using the proof of moment estimates~\eqref{eq:moment-estimates} together with Hypothesis~\ref{hyp:swim-large} in form of
\[\E\Big[{\big(|\nabla(\phi_E-\phi_E^\infty)|+|\nabla(\df\phi_E-\df\phi_E^\infty)|\big)^s}\Big]^{\frac1s}\,\lesssim_s\,\langle E\rangle^{1-\gamma},\]
we deduce for all $s<\infty$ with $s>(1\vee\frac1\gamma)\dim\Md_0^\Sym$,
\[\E\Big[\big({\textstyle\sup_{E}}\,\langle E\rangle^{-1}|\nabla(\phi_E-\phi_E^\infty)|\big)^s\Big]^\frac1s\,\lesssim\,\bigg(\int_{\Md_0^\Sym}\langle E\rangle^{-\gamma s}\,dE\bigg)^\frac1s\,\lesssim_s\,1.\]
Combining these bounds with~\eqref{eq:est-decomp-nabphi}, we deduce for all $s<\infty$,
\[\expec{|\nabla\phi|^s}^\frac1s\,\lesssim_s\,1,\]
uniformly with respect to $\delta>0$. This string of arguments allows us to post-process~\eqref{eq:moment-estimates} into
\begin{eqnarray*}
\E\Big[{|(\nabla\phi,\Pi\mathds1_{\R^d\setminus\Ic},\nabla\theta,\nabla\df\phi,\nabla\df\theta,\nabla\df^2\phi,\nabla\df^2\theta)(x)|^s}\Big]^\frac1s&\lesssim_s&1,\\
\E\Big[|(\phi,\gamma,\theta,\df\phi,\df\gamma,\df\theta,\df^2\phi,\df^2\gamma,\df^2\theta)(x)|^s\Big]^\frac1s&\lesssim_s&\mu_d(x).
\end{eqnarray*}

\medskip
\step3 Conclusion.\\
In order to capture the best dependence on $\delta$, we need a version of~\eqref{eq:est-main+} where the dependence on $\bar u_\e,\chi_\delta\ast\D(u_\e)$ does not deteriorate in terms of uniform norms.
Rather combining~\eqref{eq:est-main} and~\eqref{eq:bnd-Qeps}, and choosing $K\ge1$ large enough, we get
\begin{multline*}
\int_U|\nabla w_\e|^2+\int_U(Q_\e)^2\\
\,\lesssim\,\int_{\partial_{3r_\e}U}|(1,\nabla\psi,\Sigma\mathds1_{\R^d\setminus \Ic},\nabla\phi,\Pi\mathds1_{\R^d\setminus\Ic})(\tfrac\cdot\e)|^2\big(1+|h|^2+|\nabla\bar u_\e|^2+[\nabla\bar u_\e]_{4\e}^2+|\chi_\delta\ast\D(u_\e)|^2\big)\\
+\e^2r_\e^{-2}\int_{\partial_{3r_\e}U}|(\psi,\Ups,\phi,\theta,\gamma,\df \phi,\df\theta,\df\gamma)(\tfrac\cdot\e)|^2\big(1+|\nabla\bar u_\e|^2+|\chi_\delta\ast \D(u_\e)|^2\big)\\
+\e^2\int_U|(1,\psi,\Ups,\phi,\df \phi,\df\theta,\df\gamma,\df^2\theta,\df^2\gamma,\triangle^{-1}\nabla\mathds1_{\Ic},\triangle^{-1}\nabla\df F,\triangle^{-1}\nabla\df^2 F)(\tfrac\cdot\e)|^2\\
\hspace{1cm}\times\Big(|\langle\nabla\rangle h|^2+|\nabla^2\bar u_\e|^2+|\nabla\bar P_\e|^2+|\nabla\langle\nabla\rangle \chi_\delta\ast \D(u_\e)|^2+|\nabla\chi_\delta\ast \D(u_\e)|^4\\
\hspace{7cm}+|\nabla(\chi_\delta\ast D( u_\e))|^2\big(1+[\chi_\delta\ast\D(u_\e)]_{4\e}^4\big)\Big)\\
+\sum_n\int_{\e(I_n+B)}|(\psi,\nabla\psi,\Sigma\mathds1_{\R^d\setminus\Ic})(\tfrac\cdot\e)|^2\Big|\nabla\bar u_\e-\fint_{\e(I_n+B)}\nabla\bar u_\e\Big|^2\\
+\sum_n\int_{\e(I_n+B)}|(\df\phi,\nabla\df\phi,\df\Pi\mathds1_{\R^d\setminus\Ic})(\tfrac\cdot\e)|^2\Big|\chi_\delta\ast\D( u_\e)-\fint_{\e(I_n+B)}\chi_\delta\ast\D( u_\e)\Big|^2.
\end{multline*}
By Proposition~\ref{lem:well-posed}, we have $\|\nabla u_\e\|_{\Ld^2(U)}\lesssim1+\|h\|_{\Ld^2(U)}$ and thus, for all $k\ge0$ and $s\ge2$,
\begin{eqnarray*}
\|\chi_\delta\ast\D(u_\e)\|_{W^{k,s}(U)}&\lesssim&\delta^{-k-d(\frac12-\frac1s)}\big(1+\|h\|_{\Ld^2(U)}\big),\\
\|\chi_\delta\ast\D(u_\e)\|_{W^{k,s}(\partial_{r_\e}U)}&\lesssim&\delta^{-k-\frac d2}\big(\delta^d\wedge r_\e\big)^\frac1s\big(1+\|h\|_{\Ld^2(U)}\big).
\end{eqnarray*}
By the regularity theory for the Stokes equation, using~\eqref{eq:Cc-cc-reg}, we then deduce that the solution $(\bar u_\e,\bar P_\e)$ of~\eqref{eq:def-barueps} satisfies for all $s\ge2$ and $\eta>0$,
\begin{eqnarray*}
\|(\nabla\bar u_\e,\bar P_\e)\|_{W^{1,s}(U)}&\lesssim&\delta^{-1-d(\frac12-\frac1s)}\big(1+\|h\|_{\Ld^{2\vee s}(U)}\big),\\
\|(\nabla\bar u_\e,\bar P_\e)\|_{\Ld^{s}(\partial_{3r_\e}U)}&\lesssim&\delta^{-\frac d2}\big(\delta^d\wedge(r_\e\delta^{-\eta})\big)^{\frac1s}\big(1+\|h\|_{\Ld^{\infty}(U)}\big).
\end{eqnarray*}
Inserting these estimates into the above, we get for all $s\ge1$,
\begin{multline*}
\int_U|\nabla w_\e|^2+\int_U(Q_\e)^2\\
\,\lesssim\,\Big(r_\e+(r_\e\delta^{-d-\eta})\wedge(r_\e\delta^{-d})^\frac1s\Big)\Big(\fint_{\partial_{3r_\e}U}|(1,\nabla\psi,\Sigma\mathds1_{\R^d\setminus \Ic},\nabla\phi,\Pi\mathds1_{\R^d\setminus\Ic})(\tfrac\cdot\e)|^{2s}\Big)^\frac1s\\
+\e^2r_\e^{-2}\Big(r_\e+(r_\e\delta^{-d-\eta})\wedge(r_\e\delta^{-d})^\frac1s\Big)\Big(\fint_{\partial_{3r_\e}U}|(\psi,\Ups,\phi,\theta,\gamma,\df \phi,\df\theta,\df\gamma)(\tfrac\cdot\e)|^{2s}\Big)^\frac1s\\
+\e^2\delta^{-2-d(2+\frac1s)}\Big(\int_U|(1,\psi,\Ups,\nabla\psi,\Sigma\mathds1_{\R^d\setminus\Ic},\phi,\df \phi,\nabla\df\phi,\df\theta,\df\Pi\mathds1_{\R^d\setminus\Ic},\df\gamma,\\
\df^2\theta,\df^2\gamma,\triangle^{-1}\nabla\mathds1_{\Ic},\triangle^{-1}\nabla\df F,\triangle^{-1}\nabla\df^2 F)(\tfrac\cdot\e)|^{2s}\Big)^\frac1s,
\end{multline*}
where the multiplicative constant depends on the $W^{1,\infty}(U)$ norm of $h$.
Hence, taking the expectation and using the corrector estimates of Step~2, we get for all $s<\infty$,
\begin{multline*}
\expecM{\Big(\int_U|\nabla w_\e|^2\Big)^s}^\frac1s+\expecM{\Big(\int_U(Q_\e)^2\Big)^s}^\frac1s\\
\,\lesssim_h\,
\big(1+\e^2\mu_d(\tfrac1\e)^2r_\e^{-2}\big)\Big(r_\e+(r_\e\delta^{-d-\eta})\wedge (r_\e\delta^{-d})^\frac1s\Big)
+\e^2\mu_d(\tfrac1\e)^2\delta^{-2-d(2+\frac1s)}
\end{multline*}
Choosing $r_\e:=\e\mu_d(\frac1\e)$, this becomes
\begin{multline*}
\expecM{\Big(\int_U|\nabla w_\e|^2\Big)^s}^\frac1s+\expecM{\Big(\int_U(Q_\e)^2\Big)^s}^\frac1s\\
\,\lesssim_h\,
\e\mu_d(\tfrac1\e)+\e\mu_d(\tfrac1\e)\delta^{-d-\eta}+(\e\mu_d(\tfrac1\e))^2\delta^{-2-d(2+\frac1s)}.
\end{multline*}
As $\e,\delta\downarrow0$ in the regime~\eqref{eq:cond-delta-small}, we thus get $w_\e\to0$ in $\Ld^s(\Omega;H^1(U))$ and $Q_\e\to0$ in $\Ld^s(\Omega;\Ld^2(U))$. Arguing as for Theorem~\ref{th:homog}, and further using the result of Step~1, the conclusion follows.
\qed

\section{Dilute expansion of the effective viscosity}\label{sec:dil}
This section is devoted to the proof of Theorem~\ref{theor:dilute}, that is, the first-order dilute expansion of the effective viscosity $\Bb_\tot$. We recall that Einstein's formula for the passive contribution was already established in~\cite{DG-20} (see also~\cite{GVH,GV-20,GVM-20,GV-Hofer-20}), in form of
\[\big|\Bb_\pas -\Id-\lambda_1\Bb_\pas^{(1)}\big|\,\lesssim\,\lambda_2|\!\log\lambda_1|,\]
and it remains to prove
\begin{equation}\label{eq:einstein-CF}
\big|\Bb_\act(E)-\lambda_1\Bb_\act^{(1)}(E)\big|\,\lesssim\,\langle E\rangle\Big(\lambda_2|\!\log\lambda_1|+(\lambda_1\lambda_2+\lambda_1\lambda_3|\!\log\lambda_1|)^\frac12\Big).
\end{equation}
We split the proof into four steps.

\medskip
\step1 Periodic approximation.\\
Define a periodized version $\Pc_L$ of the point process $\Pc=\{x_n\}_n$ on the cube $Q_L:=[-\frac L2,\frac L2)^d$,
\[\Pc_L:=\{x_n:n\in N_L\},\qquad N_L:=\{n:x_n\in Q_{L-4}\},\]
and consider the corresponding random set
\[\Ic_L:=\bigcup_{n\in N_L}I_n,\qquad I_n:=x_n+I_n^\circ.\]
For notational convenience, we choose an enumeration $\Pc_L:=\{x_{n,L}\}_n$ and we set $I_{n,L}:=x_{n,L}+I_{n,L}^\circ$. By definition, under Hypothesis~\ref{hyp:part}, for all $L$, the periodized random set \mbox{$\Ic_L+L\Z^d$} satisfies the same regularity and hardcore conditions as in Hypothesis~\ref{hyp:part}. Moreover, we emphasize the stabilization property $\Pc_L|_{Q_{L-4}}=\Pc|_{Q_{L-4}}$. Next, we define $\psi_{E;L}\in\Ld^2(\Omega;H^1_\per(Q_L)^d)$ as the unique almost sure solution of the periodic version of~\eqref{e.cor-1},
\begin{equation}\label{e.cor-1-per}
\quad\left\{\begin{array}{ll}
-\triangle\psi_{E;L}+\nabla \Sigma_{E;L}=0,&\text{in $Q_L \setminus\Ic_L$},\\
\Div(\psi_{E;L})=0,&\text{in $Q_L \setminus\Ic_L$},\\
\D(\psi_{E;L}+Ex)=0,&\text{in $\Ic_L$},\\
\int_{\partial I_{n,L}}\sigma(\psi_{E;L}+Ex,\Sigma_{E;L})\nu=0,&\forall n,\\
\int_{\partial I_{n,L}}\Theta(x-\e x_{n,L})\cdot\sigma(\psi_{E;L}+Ex,\Sigma_{E;L})\nu=0,&\forall n,\,\forall\Theta\in\Md^\Skew.
\end{array}\right.
\end{equation}
It is easily checked that the map $\Bb_\act$ defined in~\eqref{eq:def-C+} can be reformulated as
\begin{gather}
E':2\Bb_\act(E)\,=\,
\lim_{L\uparrow\infty}E':2\Bb_{\act;L}(E),\nonumber\\
E':2\Bb_{\act;L}(E):=-\expecM{L^{-d}\sum_{n\in Q_L} \int_{I_{n,L}+B} (\psi_{E';L}+E'(x-x_{n,L}))\cdot f_{n,L}(E)}.\label{eq:C-limCL}
\end{gather}
We start by decomposing
\begin{equation}\label{eq:decomp-CL-dilut}
E':2\Bb_{\act;L}(E)\,=\,
E':\Bb_{\act;L}^{(1)}(E)
+E':R_{1;L}(E)+E':R_{2;L}(E),
\end{equation}
where we have set
\[E':\Bb_{\act;L}^{(1)}(E)\,:=\,-L^{-d}\expecM{\sum_n\int_{I_{n,L}+B}( \psi_{E';L}^n+E'(x-x_{n,L}))\cdot f_{n,L}(E)},\]
and where the remainders $R_{1;L}(E),R_{2;L}(E)$ are given by
\begin{eqnarray*}
E':R_{1;L}(E)&:=&-L^{-d}\expecM{\sum_{n\ne m}\int_{I_{n,L}+B} \psi_{E';L}^m\cdot f_{n,L}(E)},\\
E':R_{2;L}(E)&:=&-L^{-d}\expecM{\sum_n \int_{I_{n,L}+B} \Big(\psi_{E';L}-\sum_m\psi_{E';L}^m\Big)\cdot f_{n,L}(E)},
\end{eqnarray*}
in terms of the solution $\psi_{E';L}^n$ of the single-particle periodized problem
\begin{equation}\label{eq:single-part-per}
\left\{\begin{array}{ll}
-\triangle\psi_{E;L}^n+\nabla \Sigma_{E;L}^n=0,&\text{in $Q_L\setminus I_{n,L}$},\\
\Div(\psi_{E;L}^n)=0,&\text{in $Q_L\setminus I_{n,L}$},\\
\D(\psi_{E;L}^n+Ex)=0,&\text{in $I_{n,L}$},\\
\fint_{\partial I_{n,L}}\sigma(\psi_{E;L}^n+Ex,\Sigma_{E;L}^n)\nu=0,&\forall n,\\
\fint_{\partial I_{n,L}}\Theta(x-\e x_{n,L})\cdot\sigma(\psi_{E;L}^n+Ex,\Sigma_{E;L}^n)\nu=0,&\forall n,\,\forall\Theta\in\Md^\Skew.
\end{array}\right.
\end{equation}

\medskip
\step2 Proof that
\begin{equation}\label{eq:bnd-R1L}
|R_{1;L}|\,\lesssim\,\lambda_2\big(|\!\log\lambda_1|+\tfrac{\log L}L\big).
\end{equation}
As by assumption the point process is independent of particles' shapes and swimming forces, we can write, in terms of the two-point intensity $g_2$, cf.~\eqref{eq:def-g2},
\begin{equation*}
E':R_{1;L}(E)\,=\,-L^{-d}\iint_{Q_{L-4}\times Q_{L-4}} \expecM{\int_{Q_L}\psi_{E';L}^\circ(\cdot+x-y)\cdot\tilde f^\circ(E)}\,g_2(x,y)\,dxdy,
\end{equation*}
where $\tilde f^\circ$ is an iid copy of $f^\circ$, hence independent of $\psi_{E';L}^\circ$.
Noting that the periodicity of~$\psi_{E';L}^\circ$ yields
\begin{equation}\label{eq:cancel}
\int_{Q_L}\Big(\int_{Q_L}\psi_{E';L}^\circ(\cdot+x-y)\cdot \tilde f^\circ(E)\Big)\,dx=0,
\end{equation}
we can replace the two-point density $g_2$ by the correlation function $h_2=g_2-\lambda_1^2$, to the effect of
\begin{multline*}
E':R_{1;L}(E)\,=\,-L^{-d}\iint_{Q_{L}\times Q_{L-4}} \expecM{ \int_{Q_L}\psi_{E';L}^\circ(\cdot+x-y)\cdot \tilde f^\circ(E)}\,h_2(x,y)\,dxdy\\
+L^{-d}\iint_{(Q_{L}\setminus Q_{L-4})\times Q_{L-4}} \expecM{ \int_{Q_L}\psi_{E';L}^\circ(\cdot+x-y)\cdot \tilde f^\circ(E)}\,g_2(x,y)\,dxdy.
\end{multline*}
The neutrality condition~\eqref{eq:neutrality} entails
\[\Big|\int_{Q_L}\psi_{E';L}^\circ(\cdot+x-y)\cdot\tilde f^\circ(E)\Big|\,\lesssim\,\langle E\rangle\Big(\int_{2B}|\nabla\psi_{E';L}^\circ(\cdot+x-y)|^2\Big)^\frac12,\]
and thus, using standard decay estimates, see e.g.~\cite[Lemma~4.2]{DG-20},
\[\Big|\int_{Q_L}\psi_{E';L}^\circ(\cdot+x-y)\cdot \tilde f^\circ(E)\Big|\,\lesssim\,\langle E\rangle\langle x-y\rangle^{-d}.\]
The above then becomes
\begin{multline*}
|R_{1;L}|\,\lesssim\,L^{-d}\iint_{Q_{L}\times Q_{L}}\langle x-y\rangle^{-d}\,|h_2(x,y)|\,dxdy\\
+L^{-d}\iint_{(Q_{L}\setminus Q_{L-4})\times Q_{L}} \langle x-y\rangle^{-d}\,g_2(x,y)\,dxdy.
\end{multline*}
The definition of two-point intensity~\eqref{eq:def-lambda2} and the decay of correlations~\eqref{eq:h2-bnd} yield
\begin{gather*}
|h_2(x',y')|\,\lesssim\,(\lambda_1^2+g_2(x,y))\wedge\langle x-y\rangle^{-\gamma},\\
g_2(x,y)\,=\,g_2(x,y)\mathds1_{|x-y|\ge2\ell},\\
\fint_{B_\ell(x)\times B_\ell(y)}g_2(x',y')\,dx'dy'\,\lesssim\,\lambda_2,
\end{gather*}
and the claim~\eqref{eq:bnd-R1L} easily follows.

\medskip
\step3 Proof that
\begin{equation}\label{eq:estim-R2L}
|R_{2;L}|\,\lesssim\,(\lambda_1)^\frac12(\lambda_2+\lambda_3|\!\log\lambda_1|)^\frac12.
\end{equation}
We start by decomposing
\begin{align}\label{eq:decomp-R2L-0}
&E':R_{2;L}(E)\,=\,-L^{-d}\expecM{\sum_n \int_{(I_{n,L}+B)\setminus I_{n,L}} \Big(\psi_{E';L}-\sum_m\psi_{E';L}^m\Big)\cdot f_{n,L}(E)}\\
&\hspace{2cm}+L^{-d}\expecM{\sum_n\int_{\partial I_{n,L}}\Big(\psi_{E';L}-\sum_m\psi_{E';L}^m\Big)\cdot\sigma(\phi_{E;L},\Pi_{E;L})\nu}\nonumber\\
&\quad-L^{-d}\expecM{\sum_n \int_{I_{n,L}+B} \Big(\psi_{E';L}-\sum_m\psi_{E';L}^m\Big)\cdot\Big( f_{n,L}(E)\mathds1_{I_{n,L}}+\delta_{\partial I_{n,L}}\sigma(\phi_{E;L},\Pi_{E;L})\nu\Big)}.\nonumber
\end{align}
The first two terms can be recovered in the weak formulation of the equation for $\phi_{E;L}$ (periodized version of~\eqref{eq:phi-cor} as in~\eqref{e.cor-1-per}), when tested with $\psi_{E';L}-\sum_m\psi_{E';L}^m$,
\begin{multline*}
-L^{-d}\expecM{\sum_n \int_{(I_{n,L}+B)\setminus I_{n,L}} \Big(\psi_{E';L}-\sum_m\psi_{E';L}^m\Big)\cdot f_{n,L}(E)}\\
+L^{-d}\expecM{\sum_n\int_{\partial I_{n,L}}\Big(\psi_{E';L}-\sum_m\psi_{E';L}^m\Big)\cdot\sigma(\phi_{E;L},\Pi_{E;L})\nu}\\
\,=\,-\int_{Q_L}\nabla\phi_{E;L}:\nabla\Big(\psi_{E';L}-\sum_m\psi_{E';L}^m\Big),
\end{multline*}
and thus, similarly testing the equation for $\psi_{E';L}-\sum_m\psi_{E';L}^m$ with $\phi_{E;L}$, using boundary conditions and the rigidity condition for $\phi_{E;L}$ in $\Ic_L$,
\begin{multline*}
-L^{-d}\expecM{\sum_n \int_{(I_{n,L}+B)\setminus I_{n,L}} \Big(\psi_{E';L}-\sum_m\psi_{E';L}^m\Big)\cdot f_{n,L}(E)}\\
+L^{-d}\expecM{\sum_n\int_{\partial I_{n,L}}\Big(\psi_{E';L}-\sum_m\psi_{E';L}^m\Big)\cdot\sigma(\phi_{E;L},\Pi_{E;L})\nu}
\,=\,0.
\end{multline*}
Next, using boundary conditions for $\phi_{E;L}$, noting that $\psi_{E';L}-\psi_{E';L}^n$ is rigid in $I_{n,L}$, the last term in~\eqref{eq:decomp-R2L-0} can be rewritten as
\begin{multline*}
-L^{-d}\expecM{\sum_n \int_{I_{n,L}+B} \Big(\psi_{E';L}-\sum_m\psi_{E';L}^m\Big)\cdot\Big( f_{n,L}(E)\mathds1_{I_{n,L}}+\delta_{\partial I_{n,L}}\sigma(\phi_{E;L},\Pi_{E;L})\nu\Big)}\\
\,=\,L^{-d}\expecM{\sum_{n\ne m} \int_{I_{n,L}+B}\psi_{E';L}^m\cdot\Big( f_{n,L}(E)\mathds1_{I_{n,L}}+\delta_{\partial I_{n,L}}\sigma(\phi_{E;L},\Pi_{E;L})\nu\Big)}.
\end{multline*}
Inserting these identities into~\eqref{eq:decomp-R2L-0}, we get
\[E':R_{2;L}(E)\,=\,L^{-d}\expecM{\sum_{n\ne m} \int_{I_{n,L}+B}\psi_{E';L}^m\cdot\Big( f_{n,L}(E)\mathds1_{I_{n,L}}+\delta_{\partial I_{n,L}}\sigma(\phi_{E;L},\Pi_{E;L})\nu\Big)}.\]
Using boundary conditions for $\phi_{E;L}$ to replace $\psi_{E';L}^m$ by $\psi_{E';L}^m-\fint_{I_{n,L}}\psi_{E';L}^m$, using the Poincaré inequality, and a trace estimate,
this can be estimated as
\begin{multline*}
|E':R_{2;L}(E)|\,\lesssim\,\expecM{L^{-d}\sum_{n}\int_{I_{n,L}+B}\Big|\sum_{m:m\ne n}\nabla\psi_{E';L}^m\Big|^2}^\frac12\\
\times\expecM{L^{-d}\sum_n\int_{I_{n,L}+B}|f_{n,L}(E)|^2+L^{-d}\sum_n\int_{I_{n,L}+B}|\nabla\phi_{E;L}|^2}^\frac12.
\end{multline*}
Taking advantage of explicit renormalizations as in~\cite[Section~4.4]{DG-20}, the first factor can easily be estimated by
\[\expecM{L^{-d}\sum_{n}\int_{I_{n,L}+B}\Big|\sum_{m:m\ne n}\nabla\psi_{E';L}^m\Big|^2}\,\lesssim\,(\lambda_2+\lambda_3|\!\log\lambda_1|)|E'|^2.\]
Further noting that
\begin{equation}\label{eq:estim-norm-fn}
\expecM{L^{-d}\sum_n\int_{I_{n,L}+B}|f_{n,L}(E)|^2}\,\lesssim\,\lambda_1\int_{2B}|f_\circ(E)|^2\,\lesssim\,\lambda_1\langle E\rangle^2,
\end{equation}
and using the hardcore condition, we deduce
\begin{equation}\label{eq:pre-estim-R2}
|R_{2;L}(E)|\,\lesssim\,(\lambda_2+\lambda_3|\!\log\lambda_1|)^\frac12\bigg(\lambda_1\langle E\rangle^2+\expecM{\fint_{Q_L}|\nabla\phi_{E;L}|^2}\bigg)^\frac12.
\end{equation}
It remains to estimate the last integral: starting from the energy identity for $\phi_{E;L}$,
\[\int_{Q_L}|\nabla\phi_{E;L}|^2=\sum_n\int_{I_{n,L}+B}\phi_{E;L}\cdot f_{n,L}(E),\]
and using boundary conditions and the Poincaré inequality to estimate the right-hand side, we find
\[\int_{Q_L}|\nabla\phi_{E;L}|^2\lesssim\bigg(\sum_n\int_{I_{n,L}+B}|f_{n,L}(E)|^2\bigg)^\frac12\bigg(\sum_n\int_{I_{n,L}+B}|\nabla\phi_{E;L}|^2\bigg)^\frac12.\]
Hence, using the hardcore condition, absorbing the last factor, taking the expectation, and combining with~\eqref{eq:estim-norm-fn}, we obtain
\[\expecM{\fint_{Q_L}|\nabla\phi_{E;L}|^2}\,\lesssim\,\expecM{L^{-d}\sum_n\int_{I_{n,L}+B}|f_{n,L}(E)|^2}\,\lesssim\lambda_1\langle E\rangle^2.\]
Inserting this into~\eqref{eq:pre-estim-R2}, the claim~\eqref{eq:estim-R2L} follows.

\medskip
\step4 Conclusion.\\
In view of Steps~1, 2 and~3, it remains to examine the limit of the main term $\Bb_{\act;L}^{(1)}$ in~\eqref{eq:decomp-CL-dilut}.
As by assumption the point process $\{x_n\}$ is independent of particles' shapes and swimming forces, we can write
\[E':\Bb_{\act;L}^{(1)}(E)\,=\,-\lambda_1 L^{-d}|Q_{L-4}|\expecM{\int_{2B} (\psi_{E';L}^\circ+E'x)\cdot f^\circ(E)},\]
and thus, as $L\uparrow\infty$,
\[\lim_{L\uparrow\infty} E':\Bb_{\act;L}^{(1)}(E)\,=\,-\lambda_1\expecM{\int_{2B} (\psi_{E'}^\circ+E'x)\cdot f^\circ(E)},\]
in terms of the solution $\psi_{E'}^\circ$ of the whole-space single-particle problem~\eqref{eq:single-part}.
In case of spherical particles, $I^\circ=B$, the latter is explicitly solvable, cf.~\cite[Section~2.1.3]{GM-11},
\[\psi_{E'}^\circ(x)\,=\,\left\{\begin{array}{lll}
-E'x&:&|x|\le1,\\
-\tfrac{d+2}2\tfrac{(x\cdot E'x)x}{|x|^{d+2}}\big(1-\tfrac1{|x|^2}\big)-\tfrac{E'x}{|x|^{d+2}}&:&|x|>1,
\end{array}\right.\]
and the conclusion follows.
\qed

\section*{Acknowledgements}
Mitia Duerinckx acknowledges financial support from F.R.S.-FNRS,
and Armand Bernou and Antoine Gloria from the European Research Council (ERC) under the European Union's Horizon 2020 research and innovation programme (Grant Agreement n$^\circ$~864066).

\bibliographystyle{plain}
\bibliography{biblio}

\begin{thebibliography}{10}

\bibitem{MR830199}
R.~B. Bird.
\newblock Polymeric liquids: from molecular models to constitutive equations.
\newblock In {\em Viscoelasticity and rheology ({M}adison, {W}is., 1984)},
  pages 105--123. Academic Press, Orlando, FL, 1985.

\bibitem{Cisneros_2007}
L.~H. Cisneros, R.~Cortez, C.~Dombrowski, R.~E. Goldstein, and J.~O. Kessler.
\newblock Fluid dynamics of self-propelled microorganisms, from individuals to
  concentrated populations.
\newblock {\em Experiments in Fluids}, 43(5):737--753, 2007.

\bibitem{DoiEdwards88}
M.~Doi and S.~F. Edwards.
\newblock {\em The theory of polymer dynamics}, volume~73.
\newblock Oxford University Press, 1988.

\bibitem{D21}
M.~Duerinckx.
\newblock Effective viscosity of random suspensions without uniform separation.
\newblock {\em Ann. Inst. H. Poincar\'{e} C Anal. Non Lin\'{e}aire},
  39(5):1009--1052, 2022.

\bibitem{MD-23}
M.~Duerinckx.
\newblock Semi-dilute rheology of particle suspensions: derivation of
  {D}oi-type models.
\newblock Preprint, arXiv:2302.01466, 2023.

\bibitem{DG-FI2}
M.~Duerinckx and A.~Gloria.
\newblock Multiscale functional inequalities in probability: concentration
  properties.
\newblock {\em ALEA Lat. Am. J. Probab. Math. Stat.}, 17(1):133--157, 2020.

\bibitem{DG-FI1}
M.~Duerinckx and A.~Gloria.
\newblock Multiscale functional inequalities in probability: constructive
  approach.
\newblock {\em Ann. H. Lebesgue}, 3:825--872, 2020.

\bibitem{DG-21}
M.~Duerinckx and A.~Gloria.
\newblock Corrector equations in fluid mechanics: effective viscosity of
  colloidal suspensions.
\newblock {\em Arch. Ration. Mech. Anal.}, 239(2):1025--1060, 2021.

\bibitem{DG-22-review}
M.~Duerinckx and A.~Gloria.
\newblock Effective viscosity of semi-dilute suspensions.
\newblock {\em Séminaire Laurent Schwartz, EDP et applications}, 2021-2022.
\newblock Expos\'e n$^\circ$III.

\bibitem{DG-21b}
M.~Duerinckx and A.~Gloria.
\newblock Quantitative homogenization theory for random suspensions in steady
  {S}tokes flow.
\newblock {\em J. Ec. Polytech. - Math.}, 9:1183--1244, 2022.

\bibitem{DG-21+}
M.~Duerinckx and A.~Gloria.
\newblock Sedimentation of random suspensions and the effect of
  hyperuniformity.
\newblock {\em Annals of PDE}, 8(2), 2022.

\bibitem{DG-21a}
M.~Duerinckx and A.~Gloria.
\newblock Continuum percolation in stochastic homogenization and the effective
  viscosity problem.
\newblock To appear in {\it Arch. Ration. Mech. Anal.}, 2023.

\bibitem{DG-20}
M.~Duerinckx and A.~Gloria.
\newblock {On Einstein's effective viscosity formula}.
\newblock To appear in {\it Memoirs of the EMS}, 2023.

\bibitem{FGLP}
G.~A. Francfort, A.~Gloria, and O.~Lopez-Pamies.
\newblock Enhancement of elasto-dielectrics by homogenization of active
  charges.
\newblock {\em J. Math. Pures Appl. (9)}, 156:392--419, 2021.

\bibitem{Frenkel-46}
J.~Frenkel.
\newblock {\em {Kinetic Theory of Liquids}}.
\newblock Clarendon Press, Oxford, 1946.

\bibitem{GV-20}
D.~G\'{e}rard-Varet.
\newblock Derivation of the {B}atchelor-{G}reen formula for random suspensions.
\newblock {\em J. Math. Pures Appl. (9)}, 152:211--250, 2021.

\bibitem{GVH}
D.~G\'{e}rard-Varet and M.~Hillairet.
\newblock Analysis of the viscosity of dilute suspensions beyond {E}instein's
  formula.
\newblock {\em Arch. Ration. Mech. Anal.}, 238(3):1349--1411, 2020.

\bibitem{GV-Hofer-20}
D.~G\'erard-Varet and R.~M. H\"ofer.
\newblock Mild assumptions for the derivation of {E}instein's effective
  viscosity formula.
\newblock {\em Comm. Partial Differential Equations}, 46(4):611--629, 2021.

\bibitem{GVM-20}
D.~G\'{e}rard-Varet and A.~Mecherbet.
\newblock On the correction to {E}instein's formula for the effective
  viscosity.
\newblock {\em Ann. Inst. H. Poincar\'{e} C Anal. Non Lin\'{e}aire},
  39(1):87--119, 2022.

\bibitem{Girodroux-Lavigne-22}
D.~Girodroux-Lavigne.
\newblock Derivation of an effective rheology for dilute suspensions of
  micro-swimmers.
\newblock Preprint, arXiv:2204.04967.

\bibitem{GNO-reg}
A.~Gloria, S.~Neukamm, and F.~Otto.
\newblock A regularity theory for random elliptic operators.
\newblock {\em Milan J. Math.}, 88(1):99--170, 2020.

\bibitem{GM-11}
E.~Guazzelli and J.~Morris.
\newblock {\em {A Physical Introduction to Suspension Dynamics}}.
\newblock Cambridge University Press, 2011.

\bibitem{Haines2008}
B.~M. Haines, I.~S. Aronson, L.~Berlyand, and D.~A. Karpeev.
\newblock Effective viscosity of dilute bacterial suspensions: a
  two-dimensional model.
\newblock {\em Phys. Biol.}, 5(4):046003, 2008.

\bibitem{Haines_2009}
B.~M. Haines, A.~Sokolov, I.~S. Aranson, L.~Berlyand, and D.~A. Karpeev.
\newblock Three-dimensional model for the effective viscosity of bacterial
  suspensions.
\newblock {\em Phys. Rev. E}, 80(4):041922, 2009.

\bibitem{Hatwalne2004}
Y.~Hatwalne, S.~Ramaswamy, M.~Rao, and R.~A. Simha.
\newblock Rheology of active-particle suspensions.
\newblock {\em Phys. Rev. Lett.}, 92(11):118101, 2004.

\bibitem{Hofer-19}
R.~M. H\"{o}fer.
\newblock Convergence of the method of reflections for particle suspensions in
  {S}tokes flows.
\newblock {\em J. Differ. Equ.}, 297:81--109, 2021.

\bibitem{Hofer-Leocata-Mecherbet-22}
R.~M. H\"ofer, M.~Leocata, and A.~Mecherbet.
\newblock Derivation of the viscoelastic stress in {S}tokes flows induced by
  non-spherical {B}rownian rigid particles through homogenization.
\newblock Preprint, arXiv:2202.09317.

\bibitem{Hofer-Mecherbet-Schubert-22}
R.~M. H\"ofer, A.~Mecherbet, and R.~Schubert.
\newblock Non-existence of mean-field models for particle orientations in
  suspensions.
\newblock Preprint, arXiv:2210.15382.

\bibitem{Hofer-Schubert-23}
R.~M. H\"ofer and R.~Schubert.
\newblock {Sedimentation of particles with very small inertia in Stokes
  flows~I: convergence to the transport-Stokes equations}.
\newblock Preprint, arXiv:2302.04637.

\bibitem{Hofer-Schubert-21}
R.~M. H\"ofer and R.~Schubert.
\newblock The influence of {E}instein's effective viscosity on sedimentation at
  very small particle volume fraction.
\newblock {\em Ann. Inst. H. Poincar\'e Anal. Non Lin\'eaire},
  38(6):1897--1927, 2021.

\bibitem{JKO94}
V.~V. Jikov, S.~M. Kozlov, and O.~A. Ole\u{\i}nik.
\newblock {\em Homogenization of differential operators and integral
  functionals}.
\newblock Springer-Verlag, Berlin, 1994.

\bibitem{MR2039220}
B.~Jourdain, T.~Leli\`evre, and C.~Le~Bris.
\newblock Existence of solution for a micro-macro model of polymeric fluid: the
  {FENE} model.
\newblock {\em J. Funct. Anal.}, 209(1):162--193, 2004.

\bibitem{MR2227763}
C.~Le~Bris.
\newblock {\em Syst\`emes multi-\'{e}chelles}, volume~47 of {\em
  Math\'{e}matiques \& Applications (Berlin) [Mathematics \& Applications]}.
\newblock Springer-Verlag, Berlin, 2005.
\newblock Mod\'{e}lisation et simulation. [Modelling and simulation].

\bibitem{MR2503655}
C.~Le~Bris and T.~Leli\`evre.
\newblock Multiscale modelling of complex fluids: a mathematical initiation.
\newblock In {\em Multiscale modeling and simulation in science}, volume~66 of
  {\em Lect. Notes Comput. Sci. Eng.}, pages 49--137. Springer, Berlin, 2009.

\bibitem{Clement-15}
H.~M. L\'opez, J.~Gachelin, C.~Douarche, H.~Auradou, and \'E. Cl\'ement.
\newblock Turning bacteria suspensions into superfluids.
\newblock {\em Phys. Rev. Lett.}, 115:028301, 2015.

\bibitem{MR3010381}
N.~Masmoudi.
\newblock Global existence of weak solutions to the {FENE} dumbbell model of
  polymeric flows.
\newblock {\em Invent. Math.}, 191(2):427--500, 2013.

\bibitem{Mendelson_1999}
N.~H. Mendelson, A.~Bourque, K.~Wilkening, K.~R. Anderson, and J.~C. Watkins.
\newblock Organized cell swimming motions in bacillus subtilis colonies:
  Patterns of short-lived whirls and jets.
\newblock {\em Journal of Bacteriology}, 181(2):600--609, 1999.

\bibitem{MR1383323}
H.~C. \"{O}ttinger.
\newblock {\em Stochastic processes in polymeric fluids}.
\newblock Springer-Verlag, Berlin, 1996.

\bibitem{Paxton_2004}
W.~F. Paxton, K.~C. Kistler, C.~C. Olmeda, A.~Sen, S.~K. St.~Angelo, Y.~Cao,
  T.~E. Mallouk, P.~E. Lammert, and V.~H. Crespi.
\newblock Catalytic nanomotors: Autonomous movement of striped nanorods.
\newblock {\em J. Am. Chem. Soc.}, 126(41):13424--13431, 2004.

\bibitem{Potomkin_2016}
M.~Potomkin, S.~D. Ryan, and L.~Berlyand.
\newblock Effective rheological properties in semi-dilute bacterial
  suspensions.
\newblock {\em Bull. Math. Biol.}, 78:580--615, 2016.

\bibitem{Potomkin_2017}
M.~Potomkin, M.~Tournus, L.~V. Berlyand, and I.~S. Aranson.
\newblock Flagella bending affects macroscopic properties of bacterial
  suspensions.
\newblock {\em J. R. Soc. Interface}, 14(130):20161031, 2017.

\bibitem{Rafai2010}
S.~Rafa{\"i}, L.~Jibuti, and P.~Peyla.
\newblock Effective viscosity of microswimmer suspensions.
\newblock {\em Phys. Rev. Lett.}, 104(9):098102, 2010.

\bibitem{Ryan-Aranson-11}
S.~D. Ryan, B.~M. Haines, L.~Berlyand, F.~Ziebert, and I.~S. Aranson.
\newblock {Viscosity of bacterial suspensions: Hydrodynamic interactions and
  self-induced noise}.
\newblock {\em Phys. Rev. E}, 83:050904, 2011.

\bibitem{Saintillan-10a}
D.~Saintillan.
\newblock The dilute rheology of swimming suspensions: {A} simple kinetic
  model.
\newblock {\em Exp. Mech.}, 50(9):1275--1281, 2010.

\bibitem{Saintillan_2018}
D.~Saintillan.
\newblock Rheology of active fluids.
\newblock {\em Annu. Rev. Fluid Mech.}, 50(1):563--592, 2018.

\bibitem{Saintillan_2008}
D.~Saintillan and M.~J. Shelley.
\newblock Instabilities, pattern formation, and mixing in active suspensions.
\newblock {\em Phys. Fluids}, 20(12):123304, 2008.

\bibitem{Saintillan2014}
D.~Saintillan and M.~J. Shelley.
\newblock Theory of active suspensions.
\newblock In {\em {Complex Fluids in Biological Systems}}, pages 319--355.
  Springer New York, 2014.

\bibitem{Sokolov-Aranson-09}
A.~Sokolov and I.~S. Aranson.
\newblock Reduction of viscosity in suspension of swimming bacteria.
\newblock {\em Phys. Rev. Lett.}, 103(3):148101, 2009.

\bibitem{Sokolov2009}
A.~Sokolov, R.~E. Goldstein, F.~I. Feldchtein, and I.~S. Aranson.
\newblock Enhanced mixing and spatial instability in concentrated bacterial
  suspensions.
\newblock {\em Phys. Rev. E}, 80(3):031903, 2009.

\bibitem{Yasa2018}
O.~Yasa, P.~Erkoc, Y.~Alapan, and M.~Sitti.
\newblock Microalga-powered microswimmers toward active cargo delivery.
\newblock {\em Adv. Mater.}, 30(45):1804130, 2018.

\end{thebibliography}

\end{document}